\documentclass[final]{siamart1116}

\usepackage{amsmath,amsbsy,amsgen,amscd,amssymb,bm}
\usepackage{mathtools}
\usepackage{dsfont}
\usepackage{stmaryrd}
\usepackage{pifont}

\usepackage[scaled=1.2]{urwchancal}

\usepackage{multicol}

\graphicspath{{figs/}}

\definecolor{dark-gray}{gray}{0.3}
\definecolor{dkgray}{rgb}{.4,.4,.4}
\definecolor{dkblue}{rgb}{0,0,.5}
\definecolor{medblue}{rgb}{0,0,.75}
\definecolor{rust}{rgb}{0.5,0.1,0.1}

\usepackage{enumitem}
\usepackage{float}
\usepackage{caption}
\usepackage{subcaption}
\usepackage[font=small,margin=0.25in,labelfont={sc},labelsep={colon}]{caption}

\usepackage[noend]{algpseudocode}
\usepackage{algorithm,algorithmicx}

\algrenewcommand\alglinenumber[1]{\sf\scriptsize\color{black}{#1}}
\algrenewcommand\algorithmicrequire{\textbf{Require:}}
\algrenewcommand\algorithmicensure{\textbf{Ensure:}}

\hypersetup{urlcolor=rust}
\hypersetup{citecolor=dkblue}
\hypersetup{linkcolor=dkblue}

\numberwithin{equation}{section}
\numberwithin{theorem}{section}
\numberwithin{figure}{section}

\newsiamthm{fact}{Fact}
\crefname{fact}{Fact}{Facts}

\newsiamremark{remark}{Remark}
\newsiamremark{warning}{Warning}

\newcommand{\R}{\mathbb{R}}
\newcommand{\C}{\mathbb{C}}
\newcommand{\F}{\mathbb{F}}

\newcommand{\eps}{\varepsilon}

\newcommand{\vct}[1]{\bm{#1}}
\newcommand{\mtx}[1]{\bm{#1}}

\newcommand{\Id}{\mathbf{I}}

\newcommand{\rank}{\operatorname{rank}}

\newcommand{\diag}{\operatorname{diag}}

\newcommand{\norm}[1]{\Vert #1 \Vert}

\newcommand{\fnorm}[1]{\norm{#1}_{\mathrm{F}}}
\newcommand{\fnormsq}[1]{\fnorm{#1}^2}

\newcommand{\lowrank}[2]{\llbracket {#1} \rrbracket_{#2}}

\newcommand{\Expect}{\operatorname{\mathbb{E}}}

\newcommand{\argmin}{\operatorname{arg\,min}}

\title{Practical Sketching Algorithms \\ for Low-Rank Matrix Approximation\thanks{Dated 30 August 2016.  Revised 13 January 2017 and 6 June 2017 and 4 September 2017.
\funding{JAT and MU were supported in part by ONR Award N00014-11-1002 and the Gordon \& Betty Moore Foundation.
MU was also supported in part by DARPA Award FA8750-17-2-0101.
AY and VC were supported in part by the European Commission under Grant ERC Future Proof,
SNF 200021-146750, and SNF CRSII2-147633.}}}

\author{Joel A.~Tropp\thanks{California Institute of Technology, Pasadena, CA (\email{jtropp@cms.caltech.edu}).}\and Alp Yurtsever\thanks{{\'E}cole Polytechnique F{\'e}d{\'e}ral de Lausanne, Lausanne, Switzerland (\email{alp.yurtsever@epfl.ch}).}\and Madeleine Udell\thanks{Cornell University, Ithaca, NY (\email{mru8@cornell.edu}).}\and Volkan Cevher\thanks{{\'E}cole Polytechnique F{\'e}d{\'e}ral de Lausanne, Lausanne, Switzerland (\email{volkan.cevher@epfl.ch}).}}

\headers{Sketching Algorithms for Matrix Approximation}
{Joel A.~Tropp et al.}

\ifpdf
\hypersetup{
  pdftitle={Practical Sketching Algorithms for Low-Rank Matrix Approximation},
  pdfauthor={J.~A.~Tropp, A.~Yurtsever, M.~Udell, and V.~Cevher}
}
\fi

\begin{document}

\maketitle

\begin{abstract}
This paper describes a suite of algorithms
for constructing low-rank approximations
of an input matrix from a random linear image of the matrix,
called a \emph{sketch}.
These methods can preserve structural properties of the
input matrix, such as positive-semidefiniteness,
and they can produce approximations with a user-specified rank.
The algorithms are simple, accurate,
numerically stable, and provably correct. Moreover, each method is accompanied by an informative
error bound that allows users to select parameters
\emph{a priori} to achieve a given approximation quality.
These claims are supported by numerical experiments
with real and synthetic data.
\end{abstract}



\section{Motivation}

This paper presents a framework for computing
structured low-rank approximations of a matrix
from a \emph{sketch}, which is a random low-dimensional linear image of the matrix.
Our goal is to develop simple, practical
algorithms that can serve as reliable modules
in other applications.
The methods apply for the real field ($\F = \R$)
and for the complex field ($\F = \C$).

\subsection{Low-Rank Matrix Approximation}

Suppose that $\mtx{A} \in \F^{m \times n}$ is an arbitrary matrix.
Let $r$ be a target rank parameter where $r \ll \min\{m ,n\}$.
The computational problem is to produce a low-rank approximation $\hat{\mtx{A}}$
of $\mtx{A}$ whose error is comparable to a best rank-$r$ approximation:
\begin{equation} \label{eqn:intro-low-rank}
\fnorm{ \mtx{A} - \hat{\mtx{A}}} \approx \min_{\rank(\mtx{B}) \leq r} \fnorm{ \mtx{A} - \mtx{B} }.
\end{equation}
The notation $\fnorm{\cdot}$ refers to the Frobenius norm.
We explicitly allow the rank of $\hat{\mtx{A}}$ to exceed
$r$ because we can obtain more accurate approximations of this form,
and the precise rank of $\hat{\mtx{A}}$ is unimportant in many applications.
There has been extensive research on randomized algorithms for \cref{eqn:intro-low-rank};
see Halko et al.~\cite{HMT11:Finding-Structure}.

\subsection{Sketching}

Here is the twist.
Imagine that our interactions with the matrix $\mtx{A}$ are severely constrained
in the following way.
We construct a linear map $\mathcal{L} : \F^{m \times n} \to \F^d$
that does not depend on the matrix $\mtx{A}$.  Our only mechanism for collecting
data $\mathsf{S}$ about $\mtx{A}$ is to apply the linear map $\mathcal{L}$:
\begin{equation} \label{eqn:lin-sketch}
\mathsf{S} := \mathcal{L}(\mtx{A}) \in \F^d.
\end{equation}
We refer to $\mathsf{S}$ as a \emph{sketch} of the matrix,
and $\mathcal{L}$ is called a \emph{sketching map}.  The
number $d$ is called the \emph{dimension} or \emph{size}
of the sketch.

The challenge is to make the sketch as small as possible
while collecting enough information to approximate the matrix accurately.
In particular, we want the sketch dimension $d$ to be much smaller than the total dimension $mn$
of the matrix $\mtx{A}$.  As a consequence, the sketching map $\mathcal{L}$ has a substantial null space.
Therefore, it is natural to draw the sketching map \emph{at random} so that we
are likely to extract useful information from any fixed input matrix.

\subsection{Why Sketch?\nopunct}

There are a number of situations where the sketching model \cref{eqn:lin-sketch}
is a natural mechanism for acquiring data about an input matrix.

First, imagine that $\mtx{A}$ is a huge matrix
that can only be stored outside of core memory.
The cost of data transfer may be substantial
enough that we can only afford to read the matrix
into core memory once~\cite[Sec.~5.5]{HMT11:Finding-Structure}.
We can build a sketch as we scan through the matrix.
Other types of algorithms for this problem appear
in~\cite{FSS12:Turning-Big,FRV16:Dimensionality-Reduction}.

Second, there are applications where the columns of the matrix $\mtx{A}$
are revealed one at a time, and we must be able to compute an approximation at any instant.
One approach is to maintain a sketch that is updated when a new column arrives.
Other types of algorithms for this problem appear
in~\cite{BGKL15:Online-Principal,JJK+16:Streaming-PCA}.

Third, we may encounter a setting where the matrix $\mtx{A}$
is presented as a sum of ordered updates:
\begin{equation} \label{eqn:additive-update}
\mtx{A} = \mtx{H}_1 + \mtx{H}_2 + \mtx{H}_3 + \mtx{H}_4 + \cdots.
\end{equation}
We must discard each innovation $\mtx{H}_i$ after it is
processed~\cite{CW09:Numerical-Linear,Woo14:Sketching-Tool}.
In this case, the random linear sketch \cref{eqn:lin-sketch}
is more or less the only way to maintain a representation
of $\mtx{A}$ through an arbitrary sequence of updates~\cite{LNW14:Turnstile-Streaming}.
Our research was motivated by a variant~\cite{YUTC16:Sketchy-Decisions} of the model~\eqref{eqn:additive-update};
see \cref{sec:updates}.

\subsection{Overview of Algorithms}

Let us summarize our basic approach to sketching and low-rank approximation
of a matrix.
Fix a target rank $r$ and an input matrix $\mtx{A} \in \F^{m \times n}$.
Select sketch size parameters $k$ and $\ell$.
Draw and fix independent standard normal matrices $\mtx{\Omega} \in \F^{n \times k}$
and $\mtx{\Psi} \in \F^{\ell \times m}$; see~\cref{def:std-normal-mtx}.
We realize the randomized linear sketch~\cref{eqn:lin-sketch} via left and right
matrix multiplication:
\begin{equation} \label{eqn:sketch-intro}
\mtx{Y} := \mtx{A\Omega}
\quad\text{and}\quad
\mtx{W} := \mtx{\Psi A}
\end{equation}
We can store the random matrices and the sketch using $(k+\ell)(m+n)$ scalars.
The arithmetic cost of forming the sketch is $\Theta((k+\ell)mn)$ floating-point
operations (flops) for a general matrix $\mtx{A}$.

Given the random matrices $(\mtx{\Omega}, \mtx{\Psi})$ and the sketch $(\mtx{Y}, \mtx{W})$,
we compute an approximation $\hat{\mtx{A}}$ in three steps:
\vspace{0.5pc}
\begin{enumerate}
\item	Form an orthogonal--triangular factorization $\mtx{Y} =: \mtx{QR}$ where $\mtx{Q} \in \F^{m \times k}$.
\item	Solve a least-squares problem to obtain $\mtx{X} := (\mtx{\Psi} \mtx{Q})^\dagger \mtx{W} \in \F^{k \times n}$.
\item	Construct the rank-$k$ approximation $\hat{\mtx{A}} := \mtx{QX}$.
\end{enumerate}
\vspace{0.5pc}
The total cost of this computation is $\Theta(kl (m + n))$ flops.
See \cref{sec:intuition} for the intuition behind this approach.

Now, suppose that we set the sketch size parameters $k = 2r+1$ and $\ell = 4r+2$.
For this choice, \cref{thm:err-frob} yields the error bound
$$
\Expect \fnorm{\mtx{A} - \hat{\mtx{A}}} \leq 2 \cdot \min_{\rank(\mtx{B}) \leq r} \fnorm{ \mtx{A} - \mtx{B} }.
$$
In other words, we typically obtain an approximation with rank $\approx 2r$ whose error lies
within twice the optimal rank-$r$ error!   Moreover, the total storage cost is about $6r(m+n)$,
which is comparable with the number of degrees of freedom in an $m \times n$ matrix with rank $r$,
so the sketch size cannot be reduced substantially.

\subsection{Our Contributions}
\label{sec:contributions}

This paper presents a systematic treatment of sketching algorithms
for low-rank approximation of a matrix.
All of the methods rely on the simple sketch~\cref{eqn:sketch-intro}
of the input matrix (\cref{sec:sketch}).
The main algorithm uses this sketch to compute a high-quality low-rank approximation $\hat{\mtx{A}}$
of the input matrix  (\cref{alg:detailed-low-rank-recon}).  We prove that this method automatically
takes advantage of spectral decay in the input matrix (\cref{thm:err-frob});
this result is new.

We also explain how to compute approximations
with additional structure---such as symmetry, positive semidefiniteness, or fixed rank---by
projecting the initial low-rank approximation onto the
family of structured matrices (\cref{sec:structured,sec:fixed-rank}).
This approach ensures that the structured approximations also exploit spectral decay
(\cref{fact:convex-structure,prop:fixed-rank-err}).  In the sketching context,
this idea is new.

Each algorithm is accompanied by an informative error bound
that provides a good description of its actual behavior.
As a consequence, we can offer the first concrete guidance on algorithm
parameters for various types of input matrices (\cref{sec:best-parameters}),
and we can implement the methods with confidence.
We also include pseudocode and an accounting of computational costs.

The paper includes a collection of numerical experiments (\cref{sec:experiments}).
This work demonstrates that the recommended algorithms
can significantly outperform alternative methods,
especially when the input matrix has spectral decay.
The empirical work also confirms our guidance on parameter choices.

Our technical report~\cite{TYUC17:Randomized-Single-View-TR}
contains some more error bounds for the reconstruction
algorithms.  It also documents additional numerical experiments.

\subsection{Limitations}
\label{sec:limitations}

The algorithms in this paper are not designed for all low-rank matrix
approximation problems.  They are specifically intended for environments
where we can only make a single pass over the input matrix or where
the data matrix is presented as a stream of linear updates.
When it is possible to make multiple passes over the input matrix, we recommend
the low-rank approximation algorithms documented in~\cite{HMT11:Finding-Structure}.
Multi-pass methods are significantly more accurate because they drive the error
of the low-rank approximation down to the optimal low-rank approximation error
exponentially fast in the number of passes.

\subsection{Overview of Related Work}
\label{sec:related}

Randomized algorithms for matrix approximation date back
to research~\cite{PRTV00:Latent-Semantic,FKV04:Fast-Monte-Carlo}
in theoretical computer science (TCS) in the late 1990s.
Starting around 2004, this work inspired numerical
analysts to develop practical algorithms for
matrix approximation and related problems
~\cite{MRT11:Randomized-Algorithm}.
See the paper~\cite[Sec.~2]{HMT11:Finding-Structure} for
a comprehensive historical discussion.
The surveys~\cite{Mah11:Randomized-Algorithms,Woo14:Sketching-Tool}
provide more details about the development of these ideas within the TCS literature.

\subsubsection{Sketching Algorithms for Matrix Approximation}

To the best of our knowledge, the first sketching algorithm
for low-rank matrix approximation appears in
Woolfe et al.~\cite[Sec.~5.2]{WLRT08:Fast-Randomized}.
Their primary motivation was to compute a low-rank matrix approximation faster than
any classical algorithm, rather than to work under the constraints of a sketching model. A variant of their approach is outlined in \cite[Sec.~5.5]{HMT11:Finding-Structure}.

Clarkson \& Woodruff~\cite{CW09:Numerical-Linear}
explicitly frame the question of how to perform numerical linear algebra tasks
under the sketching model \cref{eqn:lin-sketch}.
Among other things, they develop algorithms and lower bounds
for low-rank matrix approximation.
Some of the methods that we recommend are
algebraically---but not numerically---equivalent
to formulas~\cite[Thm.~4.7, 4.8]{CW09:Numerical-Linear}
that they propose.
Their work focuses on obtaining \emph{a priori} error bounds.
In contrast, we also aim to help users
implement the methods, choose parameters, and obtain
good empirical performance in practice.
Additional details appear throughout our presentation.

There are many subsequent theoretical papers on
sketching algorithms for low-rank matrix approximation,
including~\cite{Woo14:Sketching-Tool,CEM+15:Dimensionality-Reduction,
BWZ16:Optimal-Principal-STOC}.  This line of research exploits a variety
of tricks to obtain algorithms that, theoretically,
attain better asymptotic upper bounds on computational resource usage.
\Cref{sec:other-algs} contains a representative selection
of these methods and their guarantees.

\subsubsection{Added in Press}

When we wrote this paper, the literature did not contain
sketching methods tailored for symmetric or positive-semidefinite
matrix approximation.
A theoretical paper~\cite{CW17:Low-Rank-PSD}
on algorithms for low-rank approximation
of a sparse psd matrix was released after our work appeared.

\subsubsection{Error Bounds}

Almost all previous papers in this area have centered on the
following problem.
Let $\mtx{A} \in \F^{m \times n}$ be an input matrix,
let $r$ be a target rank,  and let $\eps > 0$ be an error tolerance.
Given a randomized linear sketch \cref{eqn:lin-sketch} of the input matrix,
produce a rank-$r$ approximation $\hat{\mtx{A}}_{\mathrm{eps}}$ that satisfies
\begin{equation} \label{eqn:eps-subopt}
\fnormsq{\mtx{A} - \hat{\mtx{A}}_{\mathrm{eps}}}
	\leq (1 + \eps) \cdot \min_{\rank( \mtx{B} ) \leq r} \fnormsq{ \mtx{A} - \mtx{B} }
	\quad\text{with high probability}.
\end{equation}
To achieve~\cref{eqn:eps-subopt} for a general input,
the sketch must have dimension $\Omega(r(m+n)/\eps)$~\cite[Thm.~4.10]{CW09:Numerical-Linear}.
Furthermore, the analogous error bound for the spectral norm cannot be achieved
for all input matrices under the sketching model~\cite[Ch.~6.2]{Woo14:Sketching-Tool}.
Nevertheless, Gu~\cite[Thm.~3.4]{Gu15:Subspace-Iteration}
has observed that \cref{eqn:eps-subopt} implies a
weak error bound in the spectral norm.

Li et al.~\cite[App.]{LLS+17:Algorithm-971}
caution that the guarantee \cref{eqn:eps-subopt} is often vacuous.
For example,
we frequently encounter matrices for which the Frobenius-norm
error of an optimal rank-$r$ approximation is larger than the Frobenius norm
of the approximation itself.  In other settings, it may be necessary
to compute an approximation with very high accuracy.
Either way, $\eps$ must be tiny before the bound \cref{eqn:eps-subopt}
sufficiently constrains the approximation error.
For a general input matrix, to achieve a small value of $\eps$, the sketch size must be exorbitant.
We tackle this issue by providing alternative error estimates (e.g., \cref{thm:err-frob}) that yield big improvements for
most examples.

\subsubsection{Questions...\nopunct}

Our aim is to address questions that arise
when one attempts to use sketching algorithms in practice.
For instance,
how do we implement these methods?  Are they numerically stable?
How should algorithm parameters depend on the input matrix?
What is the right way to preserve structural properties? Which methods produce the best approximations in practice?
How small an approximation error can we actually achieve?
Does existing theoretical analysis predict performance?
Can we obtain error bounds that are more illuminating than \cref{eqn:eps-subopt}?
These questions have often been neglected in the literature.

Our empirical study (\cref{sec:experiments}) highlights
the importance of this inquiry.
Surprisingly, numerical experiments reveal that the
pursuit of theoretical metrics has been counterproductive.
More recent algorithms often perform worse in practice,
even though---in principle---they offer better
performance guarantees.

\section{Background}

In this section, we collect notation and conventions,
as well as some background on random matrices.

\subsection{Notation and Conventions}

We write $\F$ for the scalar field, which is either $\R$ or $\C$.
The letter $\Id$ signifies the identity matrix; its dimensions are determined by context.
The star ${}^*$ refers to the (conjugate) transpose operation on vectors and matrices.
The dagger ${}^\dagger$ is the Moore--Penrose pseudoinverse.
The symbol $\fnorm{\cdot}$ denotes the Frobenius norm.

The expression ``$\mtx{M}$ has rank $r$'' and its variants
mean that the rank of $\mtx{M}$ does not exceed $r$.
The symbol $\lowrank{\mtx{M}}{r}$ represents an optimal rank-$r$
approximation of $\mtx{M}$ with respect to Frobenius norm;
this approximation need not be unique~\cite[Sec.~6]{Hig89:Matrix-Nearness}.

It is valuable to introduce notation for the error incurred by a
best rank-$r$ approximation in the Frobenius norm.
For each natural number $j$, we define the \emph{$j$th tail energy}
\begin{equation} \label{eqn:tail-energy}
\tau_{j}^2(\mtx{A}) := \min_{\rank(\mtx{B}) < j} \fnormsq{ \mtx{A} - \mtx{B} }
	= \sum\nolimits_{i \geq j} \sigma_i^2(\mtx{A}).
\end{equation}
We have written $\sigma_i(\mtx{A})$ for the $i$th largest singular value of $\mtx{A}$.
The equality follows from the Eckart--Young Theorem;
for example, see~\cite[Sec.~6]{Hig89:Matrix-Nearness}.

The symbol $\Expect$ denotes expectation with respect to all random variables.
For a given random variable $Z$, we write $\Expect_{Z}$ to denote expectation
with respect to the randomness in $Z$ only.  Nonlinear functions bind before
the expectation.

In the description of algorithms in the text, we primarily use standard
mathematical notation.  In the pseudocode, we rely on some
\textsc{Matlab R2017a} functions in an effort to make the presentation
more concise.

We use the computer science interpretation of
$\Theta(\cdot)$ to refer to the class of functions
whose growth is bounded above and below up to a constant.

\subsection{Standard Normal Matrices}

Let us define an ensemble of random matrices that plays a
central role in this work.

\begin{definition}[Standard Normal Matrix] \label{def:std-normal-mtx}
A matrix $\mtx{G} \in \R^{m \times n}$ has the real standard normal distribution
if the entries form an independent family of standard normal random variables
(i.e., Gaussian with mean zero and variance one).

A matrix $\mtx{G} \in \C^{m \times n}$ has the complex standard normal distribution
if it has the form $\mtx{G} = \mtx{G}_1 + \mathrm{i} \mtx{G}_2$
where $\mtx{G}_1$ and $\mtx{G}_2$ are independent, real standard normal matrices.

Standard normal matrices are also known as Gaussian matrices.
\end{definition}

We introduce numbers $\alpha$ and $\beta$
that reflect the field over which the random matrix is defined:
\begin{equation} \label{eqn:alpha-parameter}
\alpha := \alpha(\F) := \begin{cases} 1, & \F = \R \\ 0, & \F = \C \end{cases}
\quad\text{and}\quad
\beta := \beta(\F) := \begin{cases} 1, & \F = \R \\ 2, & \F = \C \end{cases}.
\end{equation}
This notation allows us to treat the real and complex case simultaneously.
The number $\beta$ is a standard parameter in random matrix theory.

Last, we introduce notation to help make our theorem statements more succinct:
\begin{equation} \label{eqn:f-intro}
f(s, t) := \frac{s}{t - s - \alpha}
\quad\text{for integers that satisfy $t > s + \alpha > \alpha$.}
\end{equation}
Observe that the function $f(s, \cdot)$ is decreasing, with range $(0, s]$.

\section{Sketching the Input Matrix}
\label{sec:mult-sketching}

First, we discuss how to collect enough data about an input matrix
to compute a low-rank approximation.  We summarize the matrix by multiplying it on
the right and the left by random test matrices.
The dimension and distribution of these random test matrices
together determine the potential accuracy of the approximation.

\subsection{The Input Matrix}

Let $\mtx{A} \in \F^{m \times n}$ be a matrix
that we wish to approximate.
Our algorithms work regardless of the relative dimensions
of $\mtx{A}$, but there may sometimes be small benefits
if we apply them to $\mtx{A}^*$ instead.

\subsection{The Target Rank}

Let $r$ be a target rank parameter with $1 \leq r \leq \min\{m,n\}$.
We aim to construct a low-rank approximation of $\mtx{A}$ whose error is
close to the optimal rank-$r$ error. We explicitly allow approximations with rank somewhat
larger than $r$ because they may be significantly more accurate.

Under the sketching model~\cref{eqn:lin-sketch},
the practitioner must use prior knowledge
about the input matrix $\mtx{A}$ to determine a target rank $r$
that will result in satisfactory error guarantees.
This decision is outside the scope of our work.

\subsection{Parameters for the Sketch}

The sketch consists of two parts: a summary of the range of $\mtx{A}$
and a summary of the co-range.
The parameter $k$ controls the size of the range sketch,
and the parameter $\ell$ controls the size of the co-range sketch.
They should satisfy the conditions
\begin{equation} \label{eqn:param-assumption}
r \leq k \leq \ell
\quad\text{and}\quad
k \leq n
\quad\text{and}\quad
\ell \leq m.
\end{equation}
We often choose $k \approx r$ and $\ell \approx k$.
See \cref{eqn:my-param-choice,sec:best-parameters} below.

The parameters $k$ and $\ell$ do not play symmetrical roles.
We need $\ell \geq k$ to ensure that a certain $\ell \times k$ matrix has full column rank.
Larger values of both $k$ and $\ell$ result in better approximations
at the cost of more storage and arithmetic.  These tradeoffs are quantified in the sequel.

\subsection{The Test Matrices}

To form the sketch of the input matrix,
we draw and fix two (random) test matrices:
\begin{equation} \label{eqn:test-matrices}
\mtx{\Omega} \in \F^{n \times k}
\quad\text{and}\quad
\mtx{\Psi} \in \F^{\ell \times m}.
\end{equation}
This paper contains a detailed analysis of the case where the
test matrices are statistically independent and follow the standard normal distribution.
\Cref{sec:distributions} describes other potential distributions for the test matrices.
We always state when we are making distributional assumptions on the test matrices.

\subsection{The Sketch}
\label{sec:sketch}

The sketch of $\mtx{A} \in \F^{m \times n}$
consists of two matrices:
\begin{equation} \label{eqn:sketches}
\mtx{Y} := \mtx{A} \mtx{\Omega} \in \F^{m \times k}
\quad\text{and}\quad
\mtx{W} := \mtx{\Psi} \mtx{A} \in \F^{\ell \times n}.
\end{equation}
The matrix $\mtx{Y}$ collects information about the action
of $\mtx{A}$, while the matrix $\mtx{W}$ collects information
about the action of $\mtx{A}^*$.  Both parts are necessary.

\begin{remark}[Prior Work]
The matrix sketching
algorithms that appear in~\cite[Sec.~5.2]{WLRT08:Fast-Randomized}
and~\cite[Thm.~4.9]{CW09:Numerical-Linear}
and~\cite[Sec.~5.5]{HMT11:Finding-Structure}
and~\cite[Thm.~4.3]{Woo14:Sketching-Tool}
all involve a sketch of the form~\cref{eqn:sketches}.
In contrast, the most recent approaches (\cite[Sec.~6.1.2]{BWZ16:Optimal-Principal-STOC}
and~\cite[Sec.~3]{Upa16:Fast-Space-Optimal}) use more complicated
sketches; see \cref{sec:optimal-alg}.
\end{remark}

\subsection{The Sketch as an Abstract Data Type}

We present the sketch as an abstract data type
using ideas from object-oriented programming.
\textsc{Sketch} is an object that contains
information about a specific matrix $\mtx{A}$.
The test matrices $(\mtx{\Omega}, \mtx{\Psi})$
and the sketch matrices $(\mtx{Y}, \mtx{W})$
are private variables that are only accessible to the
\textsc{Sketch} methods.
A user interacts with the \textsc{Sketch} object by initializing it
with a specific matrix and by applying linear
updates.  The user can query the \textsc{Sketch} object
to obtain an approximation of the matrix $\mtx{A}$
with specific properties.  The individual algorithms
described in this paper are all methods that
belong to the \textsc{Sketch} object.

\subsection{Initializing the Sketch and its Costs}

See \cref{alg:sketch} for pseudocode that implements
the sketching procedure~\cref{eqn:test-matrices} and~\cref{eqn:sketches}
with either standard normal test matrices (default) or
random orthonormal test matrices (optional steps).
Note that the orthogonalization step requires additional
arithmetic and communication.

The storage cost
for the sketch $(\mtx{Y}, \mtx{W})$
is $mk + \ell n$ floating-point numbers in the field $\F$.
The storage cost for two standard normal test matrices is
$nk + \ell m$ floating point numbers in $\F$.  Some other types of
test matrices $(\mtx{\Omega}, \mtx{\Psi})$ have lower storage costs,
but the sketch $(\mtx{Y}, \mtx{W})$ remains the same size.

For standard normal test matrices,
the arithmetic cost of forming the sketch~\cref{eqn:sketches}
is $\Theta((k+\ell)mn)$ flops when $\mtx{A}$ is dense.
If $\mtx{A}$ is sparse, the cost is proportional to the number
$\texttt{nnz}(\mtx{A})$ of nonzero entries:
$\Theta((k + \ell) \, \texttt{nnz}(\mtx{A}))$ flops.
Other types of test matrices sometimes yield lower arithmetic costs.

\begin{algorithm}[tb]
  \caption{\textsl{Sketch for Low-Rank Approximation.}  Implements~\cref{eqn:test-matrices} and~\cref{eqn:sketches}.
  \label{alg:sketch}}
  \begin{algorithmic}[1]
    \Require{Input matrix $\mtx{A} \in \F^{m \times n}$;
    sketch size parameters $k \leq \ell$}
    \Ensure{Constructs test matrices $\mtx{\Omega} \in \F^{n \times k}$ and $\mtx{\Psi} \in \F^{\ell \times m}$,
    range sketch $\mtx{Y} = \mtx{A\Omega} \in \F^{m \times k}$, and co-range sketch $\mtx{W} = \mtx{\Psi A} \in \F^{\ell \times n}$ as private variables}
\vspace{0.5pc}

	\State \textbf{private:} $\mtx{\Omega}, \mtx{\Psi}, \mtx{Y}, \mtx{W}$
		\Comment{Internal variables for \textsc{Sketch} object}
	\Statex
		\Comment{Accessible to all \textsc{Sketch} methods}
	\Statex
	\Function{Sketch}{$\mtx{A}; k, \ell$}
		\Comment{Constructor}
    \If{$\F = \R$}
    \State	$\mtx{\Omega} \gets \texttt{randn}(n, k)$
		    \State	$\mtx{\Psi} \gets \texttt{randn}(\ell, m)$
			\EndIf
	\If{$\F = \C$}
    \State	$\mtx{\Omega} \gets \texttt{randn}(n, k) + {\rm i} \, \texttt{randn}(n, k)$
		    \State	$\mtx{\Psi} \gets \texttt{randn}(\ell, m) + {\rm i} \, \texttt{randn}(\ell, m)$
			\EndIf
	\State	$\mtx{\Omega} \gets \texttt{orth}(\mtx{\Omega})$
		\Comment{(optional) Improve numerical stability}
	\State	$\mtx{\Psi}^* \gets \texttt{orth}(\mtx{\Psi}^*)$
		\Comment{(optional) Improve numerical stability}
	\State	$\mtx{Y} \gets \mtx{A\Omega}$
			\State	$\mtx{W} \gets \mtx{\Psi A}$
			\EndFunction

	\vspace{0.25pc}

\end{algorithmic}
\end{algorithm}

\begin{algorithm}[tb]
  \caption{\textsl{Linear Update to Sketch.}  Implements~\cref{eqn:linear-update}.
  \label{alg:sketch-update}}
  \begin{algorithmic}[1]
    \Require{Update matrix $\mtx{H} \in \F^{m \times n}$; scalars $\theta, \eta \in \F$}
    \Ensure{Modifies sketch $(\mtx{Y}, \mtx{W})$
    to reflect linear update $\mtx{A} \gets \theta \mtx{A} + \eta \mtx{H}$}
\vspace{0.5pc}

	\Function{Sketch.LinearUpdate}{$\mtx{H}; \theta, \eta$}
	\State	$\mtx{Y} \gets \theta \mtx{Y} + \eta \mtx{H \Omega}$
		\Comment{Linear update to range sketch}
	\State	$\mtx{W} \gets \theta \mtx{W} + \eta \mtx{\Psi H}$
		\Comment{Linear update to co-range sketch}
	\EndFunction

	\vspace{0.25pc}

\end{algorithmic}
\end{algorithm}

\subsection{Processing Linear Updates}
\label{sec:updates}

The sketching model~\cref{eqn:sketches} supports a linear update
that is more general than \cref{eqn:additive-update}.
Suppose the input matrix $\mtx{A}$ is modified as
$$
\mtx{A} \gets \theta \mtx{A} + \eta \mtx{H}
\quad\text{where $\theta, \eta \in \F$.}
$$
Then we update the sketch~\cref{eqn:sketches} via the rule
\begin{equation} \label{eqn:linear-update}
\mtx{Y} \gets \theta \mtx{Y} + \eta \mtx{H}\mtx{\Omega}
\quad\text{and}\quad
\mtx{W} \gets \theta \mtx{W} + \eta  \mtx{\Psi} \mtx{H}.
\end{equation}
The precise cost of the computation depends on the structure of $\mtx{H}$.
See \cref{alg:sketch-update} for pseudocode.
This type of update is crucial for certain applications~\cite{YUTC16:Sketchy-Decisions}.

\subsection{Choosing the Distribution of the Test Matrices}
\label{sec:distributions}

Our analysis is specialized to the case where the test matrices
$\mtx{\Omega}$ and $\mtx{\Psi}$ are standard normal so that
we can obtain highly informative error bounds.

But there are potential benefits from implementing the sketch using test matrices
drawn from another distribution. The choice of distribution leads to some tradeoffs in
the range of permissible parameters; the costs of randomness, arithmetic, and communication to generate
the test matrices; the storage costs for the test matrices and the sketch;
the arithmetic costs for sketching and updates;
the numerical stability of matrix approximation algorithms;
and the quality of \emph{a priori} error bounds.

Let us list some of the contending distributions along with background references.
We have ranked these in decreasing order of reliability.

\vspace{0.5pc}

\begin{itemize} \setlength{\itemsep}{0.5pc}
\item	\textbf{Orthonormal.}  The optional steps in \cref{alg:sketch} generate
matrices $\mtx{\Omega}$ and $\mtx{\Psi}^*$ with orthonormal columns that span uniformly
random subspaces of dimension $k$ and $\ell$.  When $k$ and $\ell$ are very large,
these matrices result in smaller errors and better numerical stability than
Gaussians~\cite{DDH07:Fast-Linear,HMT11:Finding-Structure}.

\item	\textbf{Gaussian.}  Following~\cite{MRT11:Randomized-Algorithm,HMT11:Finding-Structure},
this paper focuses on test matrices with the standard normal distribution.
Benefits include excellent practical performance and accurate \emph{a priori} error bounds.

\item	\textbf{Rademacher.}  These test matrices have independent Rademacher\footnote{A Rademacher random variable
takes the values $\pm 1$ with equal probability.}
entries.
Their behavior is similar to Gaussian test matrices, but there
are minor improvements in the cost of storage and arithmetic, as well as
the amount of randomness required.  For example, see~\cite{CW09:Numerical-Linear}.

\item	\textbf{Subsampled Randomized Fourier Transform (SRFT).}
These test matrices take the form
\begin{equation} \label{eqn:srft}
\mtx{\Omega} = \mtx{D}_1 \mtx{F}_1 \mtx{P}_1
\quad\text{and}\quad
\mtx{\Psi} = \mtx{P}_2 \mtx{F}_2^* \mtx{D}_2
\end{equation}
where $\mtx{D}_1 \in \F^{n \times n}$ and $\mtx{D}_2 \in \F^{m \times m}$
are diagonal matrices with independent Rademacher entries; $\mtx{F}_1 \in \F^{n \times n}$
and $\mtx{F}_2 \in \F^{m \times m}$ are discrete cosine transform ($\F = \R$)
or discrete Fourier transform $(\F = \C)$ matrices;
and $\mtx{P}_1 \in \F^{n \times k}$ and $\mtx{P}_2 \in \F^{\ell \times m}$
are restrictions onto $k$ and $\ell$ coordinates, chosen uniformly at random.
These matrices work well in practice; they require a modest amount of storage;
and they support fast arithmetic.
See~\cite{AC06:Approximate-Nearest,WLRT08:Fast-Randomized,AC09:Fast-Johnson-Lindenstrauss,HMT11:Finding-Structure,Tro11:Improved-Analysis,BG13:Improved-Matrix,CNW16:Optimal-Approximate}.

\item	\textbf{Ultra-Sparse Rademacher.}  Let $s$ be a sparsity parameter.
In each row of $\mtx{\Omega}$ and column of $\mtx{\Psi}$,
we place independent Rademacher random variables in $s$ uniformly random locations;
the remaining entries of the test matrices are zero.
These matrices help control storage, arithmetic,
and randomness costs.  On the other hand, they are somewhat less reliable. For more details, see~\cite{CW13:Low-Rank-Approximation,NN13:OSNAP-Faster,
MM13:Low-Distortion-Subspace,NN14:Lower-Bounds,Woo14:Sketching-Tool,BDN15:Toward-Unified,Coh16:Nearly-Tight}.
\end{itemize}

\vspace{0.5pc}

\noindent
Except for ultra-sparse Rademacher matrices, these distributions often behave
quite like a Gaussian distribution in practice~\cite[Sec.~7.4]{HMT11:Finding-Structure}.
An exhaustive comparison of distributions for the test matrices
is outside the scope of this paper; see \cite{Lib09:Accelerated-Dense}.

\section{Low-Rank Approximation from the Sketch} \label{sec:low-rank-recon}

Suppose that we have acquired a sketch $(\mtx{Y}, \mtx{W})$
of the input matrix $\mtx{A}$, as in~\cref{eqn:test-matrices} and~\cref{eqn:sketches}.
This section presents the most basic algorithm for computing a low-rank approximation
of $\mtx{A}$ from the data in the sketch.  This simple approach
is similar to earlier proposals;
see~\cite[Sec.~5.2]{WLRT08:Fast-Randomized},
\cite[Thm.~4.7]{CW09:Numerical-Linear},
\cite[Sec.~5.5]{HMT11:Finding-Structure},
\cite[Thm.~4.3, display 1]{Woo14:Sketching-Tool}.

We have obtained the first accurate error bound for this method.
Our result shows how the spectrum of the input matrix affects
the approximation quality.  This analysis allows us to
make parameter recommendations for specific input matrices.

In \cref{sec:structured}, we explain how to refine this algorithm to obtain approximations
with additional structure.
In \cref{sec:fixed-rank}, we describe modifications
of the procedures that produce approximations
with fixed rank and additional structure.
Throughout, we maintain the notation of \cref{sec:mult-sketching}.

\subsection{The Main Algorithm}

Our goal is to produce a low-rank approximation of the input matrix $\mtx{A}$
using only the knowledge of the test matrices $(\mtx{\Omega}, \mtx{\Psi})$
and the sketch $(\mtx{Y}, \mtx{W})$.  Here is the basic method.

The first step in the procedure is to compute an orthobasis for the range of $\mtx{Y}$ by means of an orthogonal--triangular factorization:
\begin{equation} \label{eqn:def-Q}
\mtx{Y} =: \mtx{QR}
\quad\text{where}\quad
\mtx{Q} \in \F^{m \times k}.
\end{equation}
The matrix $\mtx{Q}$ has orthonormal columns; we discard the triangular matrix $\mtx{R}$.
The second step uses the co-range sketch $\mtx{W}$ to form the matrix
\begin{equation} \label{eqn:def-X}
\mtx{X} := (\mtx{\Psi} \mtx{Q})^{\dagger} \mtx{W} \in \F^{k \times n}.
\end{equation}
The random matrix $\mtx{\Psi} \mtx{Q} \in \F^{\ell \times k}$
is very well-conditioned when $\ell \gg k$, so we can perform
this computation accurately by solving a least-squares problem.
We report the rank-$k$ approximation
\begin{equation} \label{eqn:Ahat}
\hat{\mtx{A}} := \mtx{QX} \in \F^{m \times n}
\quad\text{where}\quad
\mtx{Q} \in \F^{m \times k}
\quad\text{and}\quad
\mtx{X} \in \F^{k \times n}.
\end{equation}
The factors $\mtx{Q}$ and $\mtx{X}$ are defined in~\cref{eqn:def-Q,eqn:def-X}.

\begin{remark}[Prior Work]
The approximation $\hat{\mtx{A}}$ is algebraically, but not numerically,
equivalent with the approximation that appears in
Clarkson \& Woodruff~\cite[Thm.~4.7]{CW09:Numerical-Linear};
see also~\cite[Thm.~4.3, display 1]{Woo14:Sketching-Tool}.
Our formulation improves on theirs by avoiding
a badly conditioned least-squares problem.
\end{remark}

\subsection{Intuition} \label{sec:intuition}

To motivate the algorithm, we recall a familiar
heuristic~\cite[Sec.~1]{HMT11:Finding-Structure}
from randomized linear algebra, which states that
\begin{equation} \label{eqn:A-QQA}
\mtx{A} \approx \mtx{QQ}^* \mtx{A}.
\end{equation}
Although we would like to form the rank-$k$ approximation $\mtx{Q} (\mtx{Q}^* \mtx{A})$,
we cannot compute the factor $\mtx{Q}^* \mtx{A}$ without revisiting the input matrix $\mtx{A}$.
Instead, we exploit the information in the co-range sketch $\mtx{W} = \mtx{\Psi}\mtx{A}$.
Notice that
$$
\mtx{W} = \mtx{\Psi}(\mtx{Q} \mtx{Q}^*\mtx{A}) + \mtx{\Psi}(\mtx{A} - \mtx{QQ}^* \mtx{A})
	\approx (\mtx{\Psi} \mtx{Q})(\mtx{Q}^*\mtx{A}).
$$
The heuristic~\cref{eqn:A-QQA} justifies dropping the second term.
Multiplying on the left by the pseudoinverse $(\mtx{\Psi} \mtx{Q})^{\dagger}$,
we arrive at the relation
$$
\mtx{X} = (\mtx{\Psi} \mtx{Q})^{\dagger} \mtx{W} \approx \mtx{Q}^* \mtx{A}.
$$
These considerations suggest that
$$
\hat{\mtx{A}} = \mtx{QX} \approx \mtx{QQ}^*\mtx{A} \approx \mtx{A}.
$$
One of our contributions is to give substance to these nebulae.

\begin{remark}[Prior Work]
This intuition is inspired by the discussion in~\cite[Sec.~5.5]{HMT11:Finding-Structure},
and it allows us to obtain sharp error bounds.
Our approach is quite different from that of~\cite[Thm.~4.7]{CW09:Numerical-Linear}
or \cite[Thm.~4.3]{Woo14:Sketching-Tool}. \end{remark}

\subsection{Algorithm and Costs}

\Cref{alg:simple-low-rank-recon,alg:detailed-low-rank-recon}
give pseudocode for computing the approximation~\cref{eqn:Ahat}.
The first presentation uses \textsc{Matlab} functions to abbreviate some of the steps,
while the second includes more implementation details.
Note that the use of the $\texttt{orth}$ command may result in an approximation with
rank $q$ for some $q \leq k$, but the quality of the approximation does not change.

\begin{algorithm}[tb]
  \caption{\textsl{Simplest Low-Rank Approximation.}
  Implements~\cref{eqn:Ahat}.
  \label{alg:simple-low-rank-recon}}
  \begin{algorithmic}[1]
    \Ensure{For some $q \leq k$, returns factors $\mtx{Q} \in \F^{m \times q}$ with orthonormal columns and $\mtx{X} \in \F^{q \times n}$ that form a rank-$q$ approximation $\hat{\mtx{A}}_{\rm out} = \mtx{QX}$ of the sketched matrix}
\vspace{0.5pc}

	\Function{Sketch.SimpleLowRankApprox}{\,}
    \State	$\mtx{Q} \gets \texttt{orth}( \mtx{Y} )$
		\Comment{Orthobasis for range of $\mtx{Y}$}
	\State	$\mtx{X} \gets (\mtx{\Psi} \mtx{Q}) \backslash \mtx{W}$
		\Comment{Multiply $(\mtx{\Psi}\mtx{Q})^{\dagger}$ on left side of $\mtx{W}$}
	\State \Return{$(\mtx{Q}, \mtx{X})$}
			\EndFunction

	\vspace{0.25pc}

\end{algorithmic}
\end{algorithm}

\begin{algorithm}[tb]
  \caption{\textsl{Low-Rank Approximation.}  Implements~\cref{eqn:Ahat}.
  \label{alg:detailed-low-rank-recon}}
  \begin{algorithmic}[1]
    \Ensure{Returns factors $\mtx{Q} \in \F^{m \times k}$ with orthonormal columns and $\mtx{X} \in \F^{k \times n}$ that form a rank-$k$ approximation $\hat{\mtx{A}}_{\rm out} = \mtx{QX}$ of the sketched matrix}
\vspace{0.5pc}

	\Function{Sketch.LowRankApprox}{\,}
    \State	$(\mtx{Q}, \sim) \gets \texttt{qr}( \mtx{Y}, \texttt{0} )$
    	\label{line:dlr-qr1}
		\Comment{Orthobasis for range of $\mtx{Y}$}
	\State	$(\mtx{U}, \mtx{T}) \gets \texttt{qr}(\mtx{\Psi} \mtx{Q}, \texttt{0})$
		\label{line:dlr-qr2}
		\Comment{Orthogonal--triangular factorization} 	\State	$\mtx{X} \gets \mtx{T}^{\dagger} (\mtx{U}^* \mtx{W})$
		\label{line:dlr-ls}
		\Comment{Apply inverse by back-substitution}
	\State \Return{$(\mtx{Q}, \mtx{X})$}
			\EndFunction

	\vspace{0.25pc}

\end{algorithmic}
\end{algorithm}

Let us summarize the costs of the
approximation procedure~\cref{eqn:def-Q,eqn:def-X,eqn:Ahat},
as implemented in \cref{alg:detailed-low-rank-recon}.
The algorithm has working storage of $\mathcal{O}(k(m + n))$ floating point numbers.
The arithmetic cost is $\Theta(k\ell (m + n))$ flops, which is dominated
by the matrix--matrix multiplications. The orthogonalization step and the back-substitution require $\Theta(k^2 (m+n))$ flops, which is almost as significant.

\subsection{A Bound for the Frobenius-Norm Error} 
We have established a very accurate error bound for the
approximation~\cref{eqn:Ahat}
that is implemented in \cref{alg:simple-low-rank-recon,alg:detailed-low-rank-recon}.
This analysis is one of the key contributions of this paper.

\begin{theorem}[Low-Rank Approximation: Frobenius Error] \label{thm:err-frob}
Assume that the sketch size parameters satisfy $\ell > k + \alpha$.
Draw random test matrices $\mtx{\Omega} \in \F^{n \times k}$
and $\mtx{\Psi} \in \F^{\ell \times m}$
independently from the standard normal distribution.
Then the rank-$k$ approximation $\hat{\mtx{A}}$ obtained from formula~\cref{eqn:Ahat}
satisfies
\begin{equation} \label{eqn:err-frob}
\begin{aligned}
\Expect \fnormsq{ \mtx{A} - \hat{\mtx{A}} }
	&\leq (1 + f(k,\ell)) \cdot \min_{\varrho < k - \alpha} (1 + f(\varrho,k)) \cdot \tau_{\varrho+1}^2(\mtx{A}) \\
	&= \frac{k}{\ell - k - \alpha} \cdot \min_{\varrho < k - \alpha} \frac{k}{k - \varrho - \alpha}
	\cdot \tau_{\varrho+1}^2(\mtx{A}).
\end{aligned}
\end{equation}
The index $\varrho$ ranges over natural numbers.
The quantity $\alpha(\R) := 1$ and $\alpha(\C) := 0$;
the function $f(s, t) := s/(t-s-\alpha)$;
the tail energy $\tau_j^2$ is defined in~\cref{eqn:tail-energy}.
\end{theorem}

\noindent
The proof of \cref{thm:err-frob} appears below in \cref{sec:proof-low-rank-recon}.

To begin to understand \cref{thm:err-frob},
it is helpful to consider a specific parameter choice.
Let $r$ be the target rank of the approximation, and select
\begin{equation} \label{eqn:my-param-choice}
k = 2r + \alpha
\quad\text{and}\quad
\ell = 2k + \alpha.
\end{equation}
For these sketch size parameters, with $\varrho = r$, \cref{thm:err-frob} implies that
$$
\Expect \fnormsq{ \mtx{A} - \hat{\mtx{A}} }
	\leq 4 \cdot \tau_{r+1}^2(\mtx{A}).
$$
In other words, for $k \approx 2r$, we can construct a rank-$k$ approximation of $\mtx{A}$
that has almost the same quality as a best rank-$r$ approximation.
This parameter choice balances the sketch size against the quality of approximation.

But the true meaning of \cref{thm:err-frob} lies deeper.
The minimum in \cref{eqn:err-frob} reveals that the approximation \cref{eqn:Ahat} automatically
takes advantage of decay in the tail energy.  This fundamental fact
explains the strong empirical performance of \cref{eqn:Ahat}
and other approximations derived from it.
Our analysis is the first to identify this feature.

\begin{remark}[Prior Work]
The analysis in~\cite[Thm.~3.7]{CW09:Numerical-Linear} shows
that $\hat{\mtx{A}}$ achieves a bound of the form \cref{eqn:eps-subopt}
when the sketch size parameters scale as $k = \Theta(r/\eps)$
and $\ell = \Theta(k/\eps)$.  A precise variant of
the same statement follows from \cref{thm:err-frob}.
\end{remark}

\begin{remark}[High-Probability Error Bound]
The expectation bound presented in \cref{thm:err-frob}
also describes the typical behavior of the approximation~\cref{eqn:Ahat} because
of measure concentration effects.  It is possible to develop a high-probability
bound using the methods from~\cite[Sec.~10.3]{HMT11:Finding-Structure}.
\end{remark}

\begin{remark}[Spectral-Norm Error Bound]
It is also possible to develop bounds for the spectral-norm
error incurred by the approximation \cref{eqn:Ahat}.  These
results depend on the decay of both the singular values and the
tail energies.  See~\cite[Thm.~4.2]{TYUC17:Randomized-Single-View-TR}.
\end{remark}

\subsection{Theoretical Guidance on the Sketch Size}
\label{sec:best-parameters}

\Cref{thm:err-frob} is precise enough to predict the performance
of the approximation \cref{eqn:Ahat} for many types of input matrices.
As a consequence, we can offer concrete guidance on the best sketch size parameters
$(k, \ell)$ for various applications.

Observe that the storage cost of the sketch~\cref{eqn:sketches}
is directly proportional to the sum $T := k + \ell$ of the sketch size parameters $k$ and $\ell$.
In this section, we investigate the best way to apportion $k$ and $\ell$
when we fix the target rank $r$ and the total sketch size $T$.
Throughout this discussion, we assume that $T \geq 2r + 3\alpha + 3$.
See \cref{tab:theory-params} for a summary of these rules;
see \cref{sec:theory-params} for an empirical evaluation.

\begin{table}[t]
\begin{center}
\caption{\textbf{Theoretical Sketch Size Parameters.}  This table summarizes
how to choose the sketch size parameters $(k, \ell)$ to exploit
prior information about the spectrum of the input matrix $\mtx{A}$.}
\label{tab:theory-params}
\begin{tabular}{|c|c|c|}
\hline
Problem Regime & Notation & Equation \\
\hline\hline
General purpose & $(k_{\natural}, \ell_{\natural})$ & \cref{eqn:decay-params} \\
\hline
Flat spectrum & $(k_{\flat}, \ell_{\flat})$ & \cref{eqn:flat-params,eqn:flat-params-R} \\
Decaying spectrum & $(k_{\natural}, \ell_{\natural})$ & \cref{eqn:decay-params} \\
Rapidly decaying spectrum & $(k_{\sharp}, \ell_{\sharp})$ & \cref{eqn:fast-decay-params} \\
\hline
\end{tabular}
\end{center}
\end{table}

\subsubsection{Flat Spectrum}

First, suppose that the singular values $\sigma_j(\mtx{A})$ of the input matrix $\mtx{A}$
do not decay significantly for $j > r$.  This situation occurs, for example,
when the input is a rank-$r$ matrix plus white noise.

In this setting, the minimum in \cref{eqn:err-frob} is likely to occur when $\varrho \approx r$.
It is natural to set $\varrho = r$ and to minimize the resulting bound
subject to the constraints $k + \ell = T$ and $k > r + \alpha$ and $\ell > k + \alpha$.
For $\F = \C$, we obtain the parameter recommendations
\begin{equation} \label{eqn:flat-params}
k_{\flat} := \max\left\{ r + 1, \
\left \lfloor T \cdot \frac{\sqrt{r(T - r)} - r}{T - 2r} \right \rfloor \right\}
\quad\text{and}\quad
\ell_{\flat} := T - k_{\flat}.
\end{equation}
In case $\F = \R$, we modify the formula~\cref{eqn:flat-params}
so that
\begin{equation} \label{eqn:flat-params-R}
k_{\flat} := \max\left\{ r + 2, \left \lfloor (T-1) \cdot \frac{\sqrt{r(T - r - 2)(1 - 2/(T-1))} - (r-1)}{T - 2r - 1} \right \rfloor \right\}.
\end{equation}
We omit the routine details behind these calculations.

\subsubsection{Decaying Spectrum or Spectral Gap}

Suppose that the singular values $\sigma_j(\mtx{A})$ decay
at a slow to moderate rate for $j > r$.
Alternatively, we may suppose that there is a
gap in the singular value spectrum at an index $j > r$.

In this setting, we want to exploit decay in the tail energy by setting $k \gg r$,
but we need to ensure that the term $f(k, \ell)$ in \cref{eqn:err-frob}
remains small by setting $\ell \approx 2k + \alpha$.
This intuition leads to the parameter recommendations
\begin{equation} \label{eqn:decay-params}
k_{\natural} := \max\{ r + \alpha + 1, \ \lfloor (T-\alpha)/3 \rfloor \}
\quad\text{and}\quad
\ell_{\natural} := T - k_{\natural}.
\end{equation} This is the best single choice for handling a range of examples.
The parameter recommendation \cref{eqn:my-param-choice} is an
instance of \cref{eqn:decay-params} with a minimal value of $T$.

\subsubsection{Rapidly Decaying Spectrum}

Last, assume that the singular values $\sigma_j(\mtx{A})$
decay very quickly for $j > r$. This situation occurs in the application~\cite{YUTC16:Sketchy-Decisions}
that motivated us to write this paper.

In this setting, we want to exploit decay in the tail energy
fully by setting $k$ as large as possible;
the benefit outweighs the increase in $f(k, \ell)$
from choosing $\ell = k + \alpha + 1$, the minimum possible value.
This intuition leads to the parameter recommendations
\begin{equation} \label{eqn:fast-decay-params}
k_{\sharp} := \lfloor (T-\alpha-1)/2 \rfloor
\quad\text{and}\quad
\ell_{\sharp} := T - k_{\sharp}.
\end{equation} Note that the choice~\eqref{eqn:fast-decay-params} is unwise unless the input matrix has sharp spectral decay.

\section{Low-Rank Approximations with Convex Structure}
\label{sec:structured}

In many instances, we need to reconstruct an input matrix
that has additional structure, such as symmetry or positive-semidefiniteness.
The approximation formula~\cref{eqn:Ahat}
from \cref{sec:low-rank-recon}
produces an approximation with
no special properties aside from a bound on its rank.
Therefore, we may have to reform our approximation to instill additional virtues.

In this section, we consider a class of problems where
the input matrix belongs to a convex set and we seek an
approximation that belongs to the same set.  To accomplish this
goal, we replace our initial approximation with the
closest point in the convex set.  This procedure always
improves the Frobenius-norm error.

We address two specific examples: (i) the case where
the input matrix is conjugate symmetric and (ii) the case
where the input matrix is positive semidefinite.  In both
situations, we must design the algorithm carefully
to avoid forming large matrices.

\subsection{Projection onto a Convex Set}
\label{sec:pocs}

Let $C$ be a closed and convex set of matrices in $\F^{m \times n}$.
Define the projector $\mtx{\Pi}_C$ onto the set $C$ to be the map
$$
\mtx{\Pi}_C : \F^{m \times n} \to C
\quad\text{where}\quad
\mtx{\Pi}_C(\mtx{M}) := \argmin\big\{ \fnormsq{ \mtx{C} - \mtx{M} } : \mtx{C} \in C \big\}.
$$
The $\argmin$ operator returns the matrix $\mtx{C}_{\star} \in C$ that solves the optimization problem.
The solution $\mtx{C}_{\star}$ is uniquely determined because the squared Frobenius norm is strictly
convex and the constraint set $C$ is closed and convex.

\subsection{Structure via Convex Projection}

Suppose that the input matrix $\mtx{A}$ belongs to the closed, convex set $C \subset \F^{m \times n}$.
Let $\hat{\mtx{A}}_{\rm in} \in \F^{m \times n}$ be an initial approximation of $\mtx{A}$.
We can produce a new approximation $\mtx{\Pi}_C(\hat{\mtx{A}}_{\rm in})$
by projecting the initial approximation onto the constraint set.
This procedure always improves the approximation quality in Frobenius norm.

\begin{fact}[Convex Structure Reduces Error] \label{fact:convex-structure}
Let $C \in \F^{m \times n}$ be a closed convex set, and suppose that $\mtx{A} \in C$.
For any initial approximation $\hat{\mtx{A}}_{\rm in} \in \F^{m \times n}$,
\begin{equation} \label{eqn:convex-structure}
\fnorm{ \mtx{A} - \mtx{\Pi}_C(\hat{\mtx{A}}_{\rm in}) }
	\leq \fnorm{ \mtx{A} - \hat{\mtx{A}}_{\rm in} }.
\end{equation}
\end{fact}

This result is well known in convex analysis.  It follows directly from
the first-order optimality conditions~\cite[Sec.~4.2.3]{BV04:Convex-Optimization}
for the Frobenius-norm projection of a matrix onto the set $C$.
We omit the details.

\begin{warning}[Spectral Norm]
\Cref{fact:convex-structure} does not hold if we replace the Frobenius
norm by the spectral norm.
\end{warning}

\subsection{Low-Rank Approximation with Conjugate Symmetry}

When the input matrix is conjugate symmetric,
it is often critical to produce a conjugate symmetric approximation.
We can do so by combining the simple approximation from \cref{sec:low-rank-recon}
with the projection step outlined in \cref{sec:pocs}.

\subsubsection{Conjugate Symmetric Projection}

Define the set $\mathbb{H}^n(\F)$ of conjugate symmetric matrices with dimension $n$ over the field $\F$:
$$
\mathbb{H}^n := \mathbb{H}^n(\F) := \{ \mtx{C} \in \F^{n \times n} : \mtx{C} = \mtx{C}^* \}.
$$
The set $\mathbb{H}^n(\F)$ is convex because it forms a real-linear subspace in $\F^{n \times n}$.
In the sequel, we omit the field $\F$ from the notation unless there is a possibility of confusion.

The projection $\mtx{M}_{\rm sym}$ of a matrix $\mtx{M} \in \F^{n \times n}$
onto the set $\mathbb{H}^n$ takes the form
\begin{equation} \label{eqn:sym-part}
\mtx{M}_{\rm sym} := \mtx{\Pi}_{\mathbb{H}^n}(\mtx{M}) = \frac{1}{2} (\mtx{M} + \mtx{M}^*).
\end{equation}
For example, see~\cite[Sec.~2]{Hig89:Matrix-Nearness}.

\subsubsection{Computing a Conjugate Symmetric Approximation}
\label{sec:sym-recon}

Assume that the input matrix $\mtx{A} \in \mathbb{H}^n$ is conjugate symmetric.
Let $\hat{\mtx{A}} := \mtx{QX}$ be an initial
rank-$k$ approximation of $\mtx{A}$
obtained from the approximation procedure \cref{eqn:Ahat}.
We can form a better Frobenius-norm approximation $\hat{\mtx{A}}_{\rm sym}$
by projecting $\hat{\mtx{A}}$ onto $\mathbb{H}^n$:
\begin{equation} \label{eqn:Ahat-sym}
\hat{\mtx{A}}_{\rm sym} := \mtx{\Pi}_{\mathbb{H}^n}(\hat{\mtx{A}})
	= \frac{1}{2}(\hat{\mtx{A}} + \hat{\mtx{A}}^*)
	= \frac{1}{2}(\mtx{QX} + \mtx{X}^* \mtx{Q}^*).
\end{equation}
The second relation follows from~\cref{eqn:sym-part}.

In most cases, it is preferable to present the approximation~\cref{eqn:Ahat-sym}
in factored form.  To do so, we observe that
$$
\frac{1}{2}(\mtx{QX} + \mtx{X}^* \mtx{Q}^*) = \frac{1}{2}
	\begin{bmatrix} \mtx{Q} & \mtx{X}^* \end{bmatrix}
	\begin{bmatrix} \mtx{0} & \Id \\ \Id & \mtx{0} \end{bmatrix}
	\begin{bmatrix} \mtx{Q} & \mtx{X}^* \end{bmatrix}^*.
$$
Concatenate $\mtx{Q}$ and $\mtx{X}^*$,
and compute the orthogonal--triangular factorization
\begin{equation} \label{eqn:qx-qr}
\begin{bmatrix} \mtx{Q} & \mtx{X}^* \end{bmatrix}
=: \mtx{U} \begin{bmatrix} \mtx{T}_1 & \mtx{T}_2 \end{bmatrix}
\quad\text{where}\quad
\mtx{U} \in \F^{n \times 2k}
\text{ and }
\mtx{T}_1 \in \F^{2k \times k}.
\end{equation}
Of course, we only need to orthogonalize the $k$ columns of $\mtx{X}^*$,
which permits some computational efficiencies.
Next, introduce the matrix
\begin{equation} \label{eqn:def-S}
\mtx{S} := \frac{1}{2} \begin{bmatrix} \mtx{T}_1 & \mtx{T}_2 \end{bmatrix}
	\begin{bmatrix} \mtx{0} & \Id \\ \Id & \mtx{0} \end{bmatrix}
	\begin{bmatrix} \mtx{T}_1 & \mtx{T}_2 \end{bmatrix}^*
	= \frac{1}{2}( \mtx{T}_1 \mtx{T}_2^* + \mtx{T}_2 \mtx{T}_1^* )
	\in \F^{2k \times 2k}.
\end{equation}
Combine the last four displays to obtain the rank-$(2k)$ conjugate symmetric approximation
\begin{equation} \label{eqn:Ahat-sym-factored}
\hat{\mtx{A}}_{\rm sym} 	= \mtx{USU}^*.
\end{equation}
From this expression, it is easy to obtain other types of factorizations,
such as an eigenvalue decomposition, by further processing. 
\begin{algorithm}[tb]
  \caption{\textsl{Low-Rank Symmetric Approximation.}  Implements~\cref{eqn:Ahat-sym-factored}.
  \label{alg:symm-low-rank-recon}}
  \begin{algorithmic}[1]
  	\Require{Matrix dimensions $m = n$}
    \Ensure{For $q = 2k$, returns factors $\mtx{U} \in \F^{n \times q}$ with orthonormal columns and $\mtx{S} \in \mathbb{H}^q$ that form a rank-$q$ conjugate symmetric approximation $\hat{\mtx{A}}_{\rm out}= \mtx{USU}^*$ of the sketched matrix}
\vspace{0.5pc}

	\Function{Sketch.LowRankSymApprox}{\,}
    \State	$(\mtx{Q}, \mtx{X}) \gets \textsc{LowRankApprox}(\,)$
		 \Comment{Get $\hat{\mtx{A}}_{\rm in} = \mtx{QX}$}
	\State	$(\mtx{U}, \mtx{T}) \gets \texttt{qr}([\mtx{Q}, \mtx{X}^*], \texttt{0})$
		 \label{algl:symm-qr}
		 \Comment{Orthogonal factorization of concatenation}
	\State	$\mtx{T}_1 \gets \mtx{T}(\texttt{:}, 1\texttt{:}k)$
		and $\mtx{T}_2 \gets \mtx{T}(\texttt{:}, (k+1)\texttt{:}(2k) )$ 		\Comment{Extract submatrices}
	\State	$\mtx{S} \gets (\mtx{T}_1 \mtx{T}_2^* + \mtx{T}_2 \mtx{T}_1^*)/2$
		\label{algl:symm-symm}
		\Comment{Symmetrize}
	\State \Return{$(\mtx{U}, \mtx{S})$}
		\Comment{Return factors}
	\EndFunction
	\vspace{0.25pc}

\end{algorithmic}
\end{algorithm}

\subsubsection{Algorithm, Costs, and Error}

\Cref{alg:symm-low-rank-recon} contains pseudocode for producing
a conjugate symmetric approximation of the form~\cref{eqn:Ahat-sym-factored}
from a sketch of the input matrix.  One can make this algorithm slightly more
efficient by taking advantage of the fact that $\mtx{Q}$ already has orthogonal columns;
we omit the details.

For \cref{alg:symm-low-rank-recon},
the total working storage is $\Theta(kn)$
and the arithmetic cost is $\Theta( k \ell n )$.
These costs are dominated by the call to \textsc{Sketch.LowRankApprox}.

Combining \cref{thm:err-frob} with \cref{fact:convex-structure},
we have the following bound on the error of the symmetric
approximation~\cref{eqn:Ahat-sym-factored},
implemented in \cref{alg:symm-low-rank-recon}.
As a consequence, the parameter recommendations
from \cref{sec:best-parameters}
are also valid here.

\begin{corollary}[Low-Rank Symmetric Approximation] \label{cor:symm-recon}
Assume that the input matrix $\mtx{A} \in \mathbb{H}^n(\F)$ is conjugate symmetric,
and assume that the sketch size parameters satisfy $\ell > k + \alpha$.
Draw random test matrices $\mtx{\Omega} \in \F^{n \times k}$ and $\mtx{\Psi} \in \F^{\ell \times n}$
independently from the standard normal distribution.
Then the rank-$(2k)$ conjugate symmetric approximation $\hat{\mtx{A}}_{\rm sym}$
produced by~\cref{eqn:Ahat-sym} or~\cref{eqn:Ahat-sym-factored} satisfies
$$
\Expect \fnormsq{ \mtx{A} - \hat{\mtx{A}}_{\rm sym} }
	\leq (1 + f(k,\ell)) \cdot \min_{\varrho < k - \alpha} (1 + f(\varrho,k)) \cdot \tau_{\varrho+1}^2(\mtx{A}).
$$
The index $\varrho$ ranges over natural numbers.
The quantity $\alpha(\R) := 1$ and $\alpha(\C) := 0$;
the function $f(s, t) := s/(t-s-\alpha)$;
the tail energy $\tau_j^2$ is defined in~\cref{eqn:tail-energy}.
\end{corollary}

\subsection{Low-Rank Positive-Semidefinite Approximation}

We often encounter the problem of approximating a positive-semidefinite (psd)
matrix.  In many situations, it is important to produce an approximation
that maintains positivity.  Our approach combines the
approximation~\cref{eqn:Ahat} from \cref{sec:low-rank-recon}
with the projection step from \cref{sec:pocs}.

\subsubsection{PSD Projection}

We introduce the set $\mathbb{H}_+^n(\F)$ of psd matrices with dimension $n$
over the field $\F$:
$$
\mathbb{H}_+^n := \mathbb{H}_+^n(\F)
:= \big\{ \mtx{C} \in \mathbb{H}^n : \vct{z}^* \mtx{C} \vct{z} \geq 0
\text{ for each $\vct{z} \in \F^n$} \big\}.
$$
The set $\mathbb{H}_+^n(\F)$ is convex because it is an intersection of halfspaces.
In the sequel, we omit the field $\F$ from the notation unless there is a possibility
for confusion.

Given a matrix $\mtx{M} \in \F^{n \times n}$, we construct its projection
onto the set $\mathbb{H}_+^n$ in three steps.
First, form the projection $\mtx{M}_{\rm sym} := \mtx{\Pi}_{\mathbb{H}^n}(\mtx{M})$
onto the conjugate symmetric matrices, as in~\cref{eqn:sym-part}.
Second, compute an eigenvalue decomposition $\mtx{M}_{\rm sym} =: \mtx{VDV}^*$.
Third, form $\mtx{D}_+$ by zeroing out the negative entries of $\mtx{D}$.
Then the projection $\mtx{M}_+$ of the matrix $\mtx{M}$ onto $\mathbb{H}^n_+$ takes the form
$$
\mtx{M}_+ := \mtx{\Pi}_{\mathbb{H}_+^n}(\mtx{M}) = \mtx{V} \mtx{D}_+ \mtx{V}^*.
$$
For example, see~\cite[Sec.~3]{Hig89:Matrix-Nearness}.

\subsubsection{Computing a PSD Approximation}
\label{sec:psd-recon}

Assume that the input matrix $\mtx{A} \in \mathbb{H}_+^n$ is psd.
Let $\hat{\mtx{A}} := \mtx{QX}$ be an initial approximation
of $\mtx{A}$ obtained from the approximation procedure~\cref{eqn:Ahat}.
We can form a psd approximation $\hat{\mtx{A}}_+$ by projecting $\hat{\mtx{A}}$ onto
the set $\mathbb{H}_+^n$.

To do so, we repeat the computations~\cref{eqn:qx-qr} and~\cref{eqn:def-S}
to obtain the symmetric approximation $\hat{\mtx{A}}_{\rm sym}$
presented in~\cref{eqn:Ahat-sym-factored}.
Next, form an eigenvalue decomposition of the matrix $\mtx{S}$
given by~\cref{eqn:def-S}:
$$
\mtx{S} =: \mtx{VDV}^*.
$$
In view of~\cref{eqn:Ahat-sym-factored}, we obtain
an eigenvalue decomposition of $\hat{\mtx{A}}_{\rm sym}$:
$$
\hat{\mtx{A}}_{\rm sym} = (\mtx{UV}) \mtx{D} (\mtx{UV})^*.
$$
To obtain the psd approximation $\hat{\mtx{A}}_+$, we simply
replace $\mtx{D}$ by its nonnegative part $\mtx{D}_+$ to arrive at
the rank-$(2k)$ psd approximation
\begin{equation} \label{eqn:Ahat-psd-factored}
\hat{\mtx{A}}_+ := \mtx{\Pi}_{\mathbb{H}_+^n}(\hat{\mtx{A}}) = (\mtx{UV}) \mtx{D}_+ (\mtx{UV})^*.
\end{equation}
This formula delivers an approximate eigenvalue decomposition of the input matrix.

\begin{algorithm}[tb]
  \caption{\textsl{Low-Rank PSD Approximation.} Implements~\cref{eqn:Ahat-psd-factored}.
  \label{alg:psd-low-rank-recon}}
  \begin{algorithmic}[1]
  \Require{Matrix dimensions $m = n$}
    \Ensure{For $q = 2k$, returns factors $\mtx{U} \in \F^{n \times q}$ with orthonormal columns and nonnegative, diagonal $\mtx{D} \in \mathbb{H}_+^q$ that form a rank-$q$ psd approximation $\hat{\mtx{A}}_{\rm out} = \mtx{UDU}^*$ of the sketched matrix}
\vspace{0.5pc}

	\Function{Sketch.LowRankPSDApprox}{\,}
    \State	$(\mtx{U},\mtx{S}) \gets \textsc{LowRankSymApprox}(\,)$
		 \label{algl:psd-orth}
		 \Comment{Get $\hat{\mtx{A}}_{\rm in} = \mtx{USU}^*$}
	\State	$(\mtx{V}, \mtx{D}) \gets \texttt{eig}(\mtx{S})$
		\label{algl:psd-eig}
		\Comment{Form eigendecomposition}
	\State	$\mtx{U} \gets \mtx{U} \mtx{V}$
		\label{algl:psd-consol}
		\Comment{Consolidate orthonormal factors}
	\State	$\mtx{D} \gets \texttt{max}(\mtx{D}, \texttt{0})$
		\Comment{Zero out negative eigenvalues}
	\State \Return{$(\mtx{U}, \mtx{D})$}
			\EndFunction
	\vspace{0.25pc}

\end{algorithmic}
\end{algorithm}

\subsubsection{Algorithm, Costs, and Error}

\Cref{alg:psd-low-rank-recon} contains pseudocode for producing
a psd approximation of the form~\cref{eqn:Ahat-psd-factored}
from a sketch of the input matrix.
As in \cref{alg:symm-low-rank-recon}, some additional efficiencies are possible

The costs of \cref{alg:psd-low-rank-recon}
are similar with the symmetric approximation method, \cref{alg:symm-low-rank-recon}.
The working storage cost is $\Theta(kn)$,
and the arithmetic cost is $\Theta(k \ell n)$.

Combining \cref{thm:err-frob,fact:convex-structure},
we obtain a bound on the approximation error identical to
\cref{cor:symm-recon}.  We omit the details.

\section{Fixed-Rank Approximations from the Sketch}
\label{sec:fixed-rank}

The algorithms in \cref{sec:low-rank-recon,sec:structured}
produce high-quality approximations with rank $k$, but
we sometimes need to reduce the rank to match the target rank $r$.
At the same time, we may have to impose additional structure.
This section explains how to develop algorithms that produce
a rank-$r$ structured approximation.

The technique is conceptually similar to the approach in \cref{sec:structured}.
We project an initial high-quality approximation onto the set of rank-$r$
matrices.  This procedure preserves both conjugate symmetry and the psd property.
The analysis in \cref{sec:pocs} does not apply because
the set of matrices with fixed rank is not convex.
We present a general argument to show that the
cost is negligible.

\subsection{A General Error Bound for Fixed-Rank Approximation}

If we have a good initial approximation of the input matrix,
we can replace this initial approximation by a fixed-rank matrix
without increasing the error significantly.

\begin{proposition}[Error for Fixed-Rank Approximation] \label{prop:fixed-rank-err}
Let $\mtx{A} \in \F^{m \times n}$ be a input matrix,
and let $\hat{\mtx{A}}_{\rm in} \in \F^{m \times n}$ be an approximation.
For any rank parameter $r$,
\begin{equation} \label{eqn:fixed-rank-frob}
\fnorm{ \mtx{A} - \lowrank{\hat{\mtx{A}}_{\rm in}}{r} }
	\leq \tau_{r+1}(\mtx{A}) + 2 \fnorm{ \mtx{A} - \hat{\mtx{A}}_{\rm in} }.
\end{equation}
Recall that $\lowrank{\cdot}{r}$ returns a best rank-$r$ approximation
with respect to Frobenius norm.
\end{proposition}

\begin{proof}
Calculate that
$$
\begin{aligned}
\fnorm{ \mtx{A} - \lowrank{\hat{\mtx{A}}_{\rm in} }{r} }
	&\leq \fnorm{ \mtx{A} - \hat{\mtx{A}}_{\rm in} } + \fnorm{ \hat{\mtx{A}}_{\rm in}  - \lowrank{\hat{\mtx{A}}_{\rm in}}{r} } \\
	&\leq \fnorm{ \mtx{A} - \hat{\mtx{A}}_{\rm in}  } + \fnorm{ \hat{\mtx{A}}_{\rm in}  - \lowrank{\mtx{A}}{r} } \\
	&\leq 2 \fnorm{ \mtx{A} - \hat{\mtx{A}}_{\rm in}  } + \fnorm{ \mtx{A} - \lowrank{\mtx{A}}{r} }.
\end{aligned}
$$
The first and last relations are triangle inequalities.  To reach the second line,
note that $\lowrank{\hat{\mtx{A}}_{\rm in}}{r}$ is a best rank-$r$ approximation of $\hat{\mtx{A}}_{\rm in}$,
while $\lowrank{ \mtx{A}}{r}$ is an undistinguished rank-$r$ matrix.
Finally, identify the tail energy \cref{eqn:tail-energy}.
\end{proof}

\begin{remark}[Spectral Norm] \label{rem:fixed-rank-spec}
A result analogous to \cref{prop:fixed-rank-err} also holds
with respect to the spectral norm.  The proof is the same.
\end{remark}

\subsection{Fixed-Rank Approximation from the Sketch}

Suppose that we wish to compute a rank-$r$ approximation
of the input matrix $\mtx{A} \in \F^{m \times n}$.
First, we form an initial approximation $\hat{\mtx{A}} := \mtx{QX}$
using the procedure~\cref{eqn:Ahat}.  Then we obtain a rank-$r$
approximation $\lowrank{\hat{\mtx{A}}}{r}$ of the input matrix
by replacing $\hat{\mtx{A}}$ with its best rank-$r$ approximation
in Frobenius norm:
\begin{equation} \label{eqn:Ahat-fixed0}
\lowrank{\hat{\mtx{A}}}{r} = \lowrank{ \mtx{Q}\mtx{X} }{r}.
\end{equation}

We can complete this operation by working directly
with the factors.  Indeed, suppose that
$\mtx{X} = \mtx{U\Sigma V}^*$ is an SVD of $\mtx{X}$.
Then $\mtx{QX}$ has an SVD of the form
$$
\mtx{QX} = (\mtx{QU}) \mtx{\Sigma} \mtx{V}^*.
$$
As such, there is also a best rank-$r$ approximation of $\mtx{QX}$ that satisfies
$$
\lowrank{\mtx{QX}}{r} = (\mtx{QU}) \lowrank{\mtx{\Sigma}}{r} \mtx{V}^*
	= \mtx{Q} \lowrank{\mtx{X}}{r}.
$$
Therefore, the desired rank-$r$ approximation~\cref{eqn:Ahat-fixed0}
can also be expressed as
\begin{equation} \label{eqn:Ahat-fixed}
\lowrank{\hat{\mtx{A}}}{r} = \mtx{Q} \lowrank{\mtx{X}}{r}.
\end{equation}
The formula~\cref{eqn:Ahat-fixed} is more computationally
efficient than~\cref{eqn:Ahat-fixed0}
because the factor $\mtx{X} \in \F^{k \times n}$
is much smaller than the approximation $\hat{\mtx{A}} \in \F^{m \times n}$.

\begin{remark}[Prior Work]
The approximation $\lowrank{\hat{\mtx{A}}}{r}$ is algebraically, but not numerically,
equivalent to a formula proposed by Clarkson \& Woodruff~\cite[Thm.~4.8]{CW09:Numerical-Linear}.
As above, our formulation improves on theirs by avoiding
a badly conditioned least-squares problem.
\end{remark}

\subsubsection{Algorithm and Costs}

\Cref{alg:fixed-rank-recon} contains pseudocode for
computing the fixed-rank approximation~\cref{eqn:Ahat-fixed}.

The fixed-rank approximation in \cref{alg:fixed-rank-recon}
has storage and arithmetic costs on the same order as the simple
low-rank approximation (\cref{alg:simple-low-rank-recon}).
Indeed, to compute the truncated SVD and perform the matrix--matrix multiplication, we expend only $\Theta(k^2 n)$ additional flops.
Thus, the total working storage is $\Theta(k (m + n))$ numbers
and the arithmetic cost is $\Theta(k\ell(m+n))$ flops.

\begin{algorithm}[tb]
  \caption{\textsl{Fixed-Rank Approximation.}  Implements~\cref{eqn:Ahat-fixed}.
  \label{alg:fixed-rank-recon}}
  \begin{algorithmic}[1]
  	\Require{Target rank $r \leq k$}
    \Ensure{Returns factors $\mtx{Q} \in \F^{m \times r}$ and $\mtx{V} \in \F^{n \times r}$ with orthonormal columns and nonnegative diagonal $\mtx{\Sigma} \in \F^{r \times r}$ that form a rank-$r$ approximation $\hat{\mtx{A}}_{\rm out} = \mtx{Q \Sigma V}^*$ of the sketched matrix}
\vspace{0.5pc}

	\Function{Sketch.FixedRankApprox}{$r$}
    \State	$(\mtx{Q}, \mtx{X}) \gets \textsc{LowRankApprox}(\,)$
    	\Comment{Get $\hat{\mtx{A}}_{\rm in} = \mtx{QX}$}
	\State	$(\mtx{U}, \mtx{\Sigma}, \mtx{V}) \gets \texttt{svds}(\mtx{X}, r)$
		\label{line:fr-svd}
		\Comment{Form full SVD and truncate}
	\State	$\mtx{Q} \gets \mtx{Q} \mtx{U}$
		\label{line:fr-consol}
		\Comment{Consolidate orthonormal factors}
	\State \Return{$(\mtx{Q}, \mtx{\Sigma}, \mtx{V})$}
			\EndFunction
	\vspace{0.25pc}
\end{algorithmic}
\end{algorithm}

\subsubsection{A Bound for the Error}

We can obtain an error bound for the rank-$r$ approximation~\cref{eqn:Ahat-fixed}
by combining \cref{thm:err-frob,prop:fixed-rank-err}.

\begin{corollary}[Fixed-Rank Approximation: Frobenius-Norm Error] \label{cor:fixed-rank-recon}
Assume the sketch size parameters satisfy $k > r + \alpha$ and $\ell > k + \alpha$.
Draw random test matrices $\mtx{\Omega} \in \F^{n \times k}$ and $\mtx{\Psi} \in \F^{\ell \times m}$
independently from the standard normal distribution.  Then the
rank-$r$ approximation $\lowrank{\hat{\mtx{A}}}{r}$ obtained from
the formula~\cref{eqn:Ahat-fixed} satisfies
\begin{multline} \label{eqn:fixed-rank-bound}
\Expect \fnorm{ \mtx{A} - \lowrank{ \hat{\mtx{A}} }{r} }
	\leq \tau_{r+1}(\mtx{A})
	+ 2 \sqrt{1+f(k,\ell)} \cdot \min_{\varrho < k - \alpha}
	\sqrt{1 + f(\varrho, k)} \cdot \tau_{\varrho+1}(\mtx{A}).
\end{multline}
The index $\varrho$ ranges over natural numbers.
The quantity $\alpha(\R) := 1$ and $\alpha(\C) := 0$;
the function $f(s, t) := s/(t-s-\alpha)$;
the tail energy $\tau_j^2$ is defined in~\cref{eqn:tail-energy}.
\end{corollary}

This result indicates that the fixed-rank approximation $\lowrank{\hat{\mtx{A}}}{r}$
automatically exploits spectral decay in the input matrix $\mtx{A}$.  Moreover,
we can still rely on the parameter recommendations from \cref{sec:best-parameters}.
Ours is the first theory to provide these benefits.

\begin{remark}[Prior Work] \label{rem:fixed-rank-prior}
The analysis~\cite[Thm.~4.8]{CW09:Numerical-Linear} of Clarkson \& Woodruff
implies that the approximation \cref{eqn:Ahat-fixed}
can achieve the bound \cref{eqn:eps-subopt} for any $\eps > 0$,
provided that $k = \Theta(r/\eps^2)$ and $\ell = \Theta(k/\eps^2)$.
It is possible to improve this scaling; see~\cite[Thm.~5.1]{TYUC17:Randomized-Single-View-TR}.
\end{remark}

\begin{remark}[Spectral-Norm Error Bound]
It is possible to obtain an error bound for the rank-$r$ approximation \cref{eqn:Ahat-fixed}
with respect to the spectral norm by combining \cite[Thm.~4.2]{TYUC17:Randomized-Single-View-TR}
and \cref{rem:fixed-rank-spec}.
\end{remark}

\subsection{Fixed-Rank Conjugate Symmetric Approximation}

Assume that the input matrix $\mtx{A} \in \mathbb{H}^n$ is conjugate symmetric
and we wish to compute a rank-$r$ conjugate symmetric approximation.
First, form an initial approximation $\hat{\mtx{A}}_{\rm sym}$ using
the procedure~\cref{eqn:Ahat-sym-factored} in \cref{sec:sym-recon}.
Then compute an $r$-truncated eigenvalue decomposition of the matrix $\mtx{S}$
defined in~\cref{eqn:def-S}:
$$
\mtx{S} =: \mtx{V} \lowrank{\mtx{D}}{r} \mtx{V}^* \ +\ \textrm{approximation error}.
$$
In view of the representation~\cref{eqn:Ahat-sym-factored},
\begin{equation} \label{eqn:Ahat-sym-fixed}
\lowrank{\hat{\mtx{A}}_{\rm sym}}{r} = (\mtx{UV}) \lowrank{\mtx{D}}{r} (\mtx{UV})^*.
\end{equation}
\Cref{alg:sym-fixed-rank-recon} contains pseudocode for the
fixed-rank approximation~\cref{eqn:Ahat-sym-fixed}.
The total working storage is $\Theta(k n)$,
and the arithmetic cost is $\Theta(k \ell n)$.

If $\mtx{A}$ is conjugate symmetric, then \cref{cor:symm-recon,prop:fixed-rank-err}
shows that $\lowrank{\hat{\mtx{A}}_{\rm sym}}{r}$
admits an error bound identical to \cref{cor:fixed-rank-recon}.
We omit the details.

\begin{algorithm}[tb]
  \caption{\textsl{Fixed-Rank Symmetric Approximation.}  Implements~\cref{eqn:Ahat-sym-fixed}.  \label{alg:sym-fixed-rank-recon}}
  \begin{algorithmic}[1]
  \Require{Matrix dimensions $m = n$; target rank $r \leq k$}
    \Ensure{Returns factors $\mtx{U} \in \F^{n \times r}$ with orthonormal columns and diagonal $\mtx{D} \in \mathbb{H}^r$ that form a rank-$r$ conjugate symmetric approximation $\hat{\mtx{A}}_{\rm out} = \mtx{UDU}^*$ of the sketched matrix}\vspace{0.5pc}

	\Function{Sketch.FixedRankSymApprox}{$r$}
    \State	$(\mtx{U}, \mtx{S}) \gets \textsc{LowRankSymApprox}(\,)$
		 		 \Comment{Get $\hat{\mtx{A}}_{\rm in} = \mtx{USU}^*$}
	\State	$(\mtx{V}, \mtx{D}) \gets \texttt{eigs}(\mtx{S}, r, \texttt{'lm'})$
		\Comment{Truncate full eigendecomposition}
	\State	$\mtx{U} \gets \mtx{U} \mtx{V}$
		\Comment{Consolidate orthonormal factors}
	\State \Return{$(\mtx{U}, \mtx{D})$}
			\EndFunction
	\vspace{0.25pc}
\end{algorithmic}
\end{algorithm}

\subsection{Fixed-Rank PSD Approximation}

Assume that the input matrix $\mtx{A} \in \mathbb{H}_+^n$ is psd,
and we wish to compute a rank-$r$ psd approximation $\lowrank{\hat{\mtx{A}}_+}{r}$.
First, form an initial approximation $\hat{\mtx{A}}_{+}$ using
the procedure~\cref{eqn:Ahat-psd-factored} in \cref{sec:psd-recon}.
Then compute an $r$-truncated positive eigenvalue decomposition
of the matrix $\mtx{S}$ defined in~\cref{eqn:def-S}:
$$
\mtx{S} =: \mtx{V} \lowrank{\mtx{D}_+}{r} \mtx{V}^*
\ +\ \textrm{approximation error}.
$$
In view of the representation~\cref{eqn:Ahat-psd-factored},
\begin{equation} \label{eqn:Ahat-psd-fixed}
\lowrank{\hat{\mtx{A}}_{+}}{r} = (\mtx{UV}) \lowrank{\mtx{D}_+}{r} (\mtx{UV})^*.
\end{equation}
\Cref{alg:psd-fixed-rank-recon} contains pseudocode for the
fixed-rank psd approximation~\cref{eqn:Ahat-psd-fixed}.
The working storage is $\Theta(kn)$,
and the arithmetic cost is $\Theta(k \ell n)$.
If $\mtx{A}$ is psd, then \cref{cor:symm-recon,prop:fixed-rank-err}
show that $\lowrank{\hat{\mtx{A}}_{\rm psd}}{r}$ satisfies
an error bound identical to \cref{cor:fixed-rank-recon};
we omit the details.

\begin{algorithm}[tb]
  \caption{\textsl{Fixed-Rank PSD Approximation.}  Implements~\cref{eqn:Ahat-psd-fixed}.  \label{alg:psd-fixed-rank-recon}}
  \begin{algorithmic}[1]
  	\Require{Matrix dimensions $m = n$; target rank $r \leq k$}
    \Ensure{Returns factors $\mtx{U} \in \F^{n \times r}$ with orthonormal columns and nonnegative, diagonal $\mtx{D} \in \mathbb{H}_+^r$ that form a rank-$r$ psd approximation $\hat{\mtx{A}}_{\rm out} = \mtx{UDU}^*$ of the sketched matrix}\vspace{0.5pc}

	\Function{Sketch.FixedRankPSDApprox}{$r$}
    \State	$(\mtx{U},\mtx{S}) \gets \textsc{LowRankSymApprox}(\,)$
		 		 \Comment{Get $\hat{\mtx{A}}_{\rm in} = \mtx{USU}^*$}
	\State	$(\mtx{V}, \mtx{D}) \gets \texttt{eigs}(\mtx{S}, r, \texttt{'lr'})$
		\Comment{Truncate full eigendecomposition}
	\State	$\mtx{U} \gets \mtx{U} \mtx{V}$
				\Comment{Consolidate orthonormal factors}
	\State	$\mtx{D} \gets \texttt{max}(\mtx{D}, \texttt{0})$
		\Comment{Zero out negative eigenvalues}
	\State \Return{$(\mtx{U}, \mtx{D})$}
			\EndFunction

	\vspace{0.25pc}
\end{algorithmic}
\end{algorithm}

\section{Computational Experiments}
\label{sec:experiments}

This section presents the results of some numerical tests designed
to evaluate the empirical performance of our sketching algorithms
for low-rank matrix approximation.  We demonstrate that the approximation
quality improves when we impose structure, and we show that our theoretical
parameter choices are effective.
The presentation also includes comparisons with several other algorithms
from the literature.

\subsection{Overview of Experimental Setup}

For our numerical assessment, we work over the complex field ($\F = \C$).
Results for the real field ($\F = \R$) are similar.

Let us summarize the procedure
for studying the behavior of a specified
approximation method on a given input matrix.
Fix the input matrix $\mtx{A}$ and the target rank $r$.
Then select a pair $(k, \ell)$ of sketch size parameters where $k \geq r$ and $\ell \geq r$.

Each trial has the following form.
We draw (complex) standard normal test matrices $(\mtx{\Omega}, \mtx{\Psi})$
to form the sketch $(\mtx{Y}, \mtx{W})$ of the input matrix.
[We do not use the optional orthogonalization steps in \cref{alg:sketch}.]
Next compute an approximation $\hat{\mtx{A}}_{\rm out}$ of the matrix $\mtx{A}$
by means of a specified approximation algorithm.
Then calculate the error relative to the best rank-$r$ approximation:
\begin{equation} \label{eqn:relative-error}
\textrm{relative error} \quad:=\quad
\frac{\fnorm{ \mtx{A} - \hat{\mtx{A}}_{\rm out} }}{ \tau_{r+1}(\mtx{A}) }
 - 1.
\end{equation}
The tail energy $\tau_j$ is defined in~\eqref{eqn:tail-energy}.
If $\hat{\mtx{A}}_{\rm out}$ is a rank-$r$ approximation of $\mtx{A}$,
the relative error is always nonnegative.  To facilitate comparisons,
our experiments only examine fixed-rank approximation methods.

To obtain each data point, we repeat the procedure from the last paragraph 20 times,
each time with the same input matrix $\mtx{A}$
and an independent draw of the test matrices $(\mtx{\Omega}, \mtx{\Psi})$.
Then we report the average relative error over the 20 trials.

We include our \textsc{Matlab} implementations in the supplementary materials
for readers who seek more details on the methodology.

\subsection{Classes of Input Matrices}
\label{sec:input-matrix-examples}

We perform our numerical tests using several types of
complex-valued input matrices.
\Cref{fig:singVal} illustrates the singular spectrum
of a matrix from each of the categories.

\begin{figure}[tp!]
\begin{center}
\includegraphics[height=1.5in]{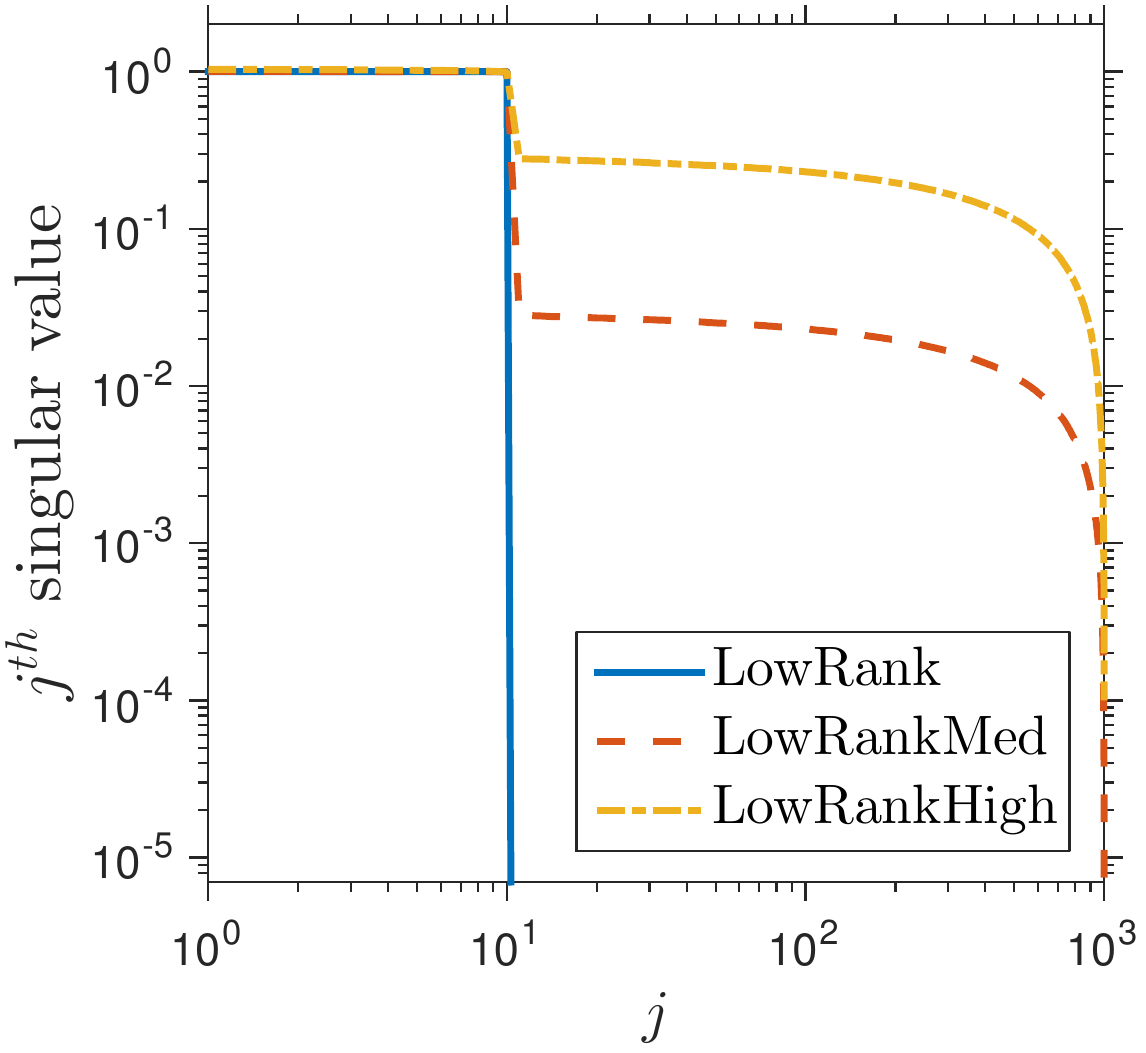} \hspace{0.5pc}
\includegraphics[height=1.5in]{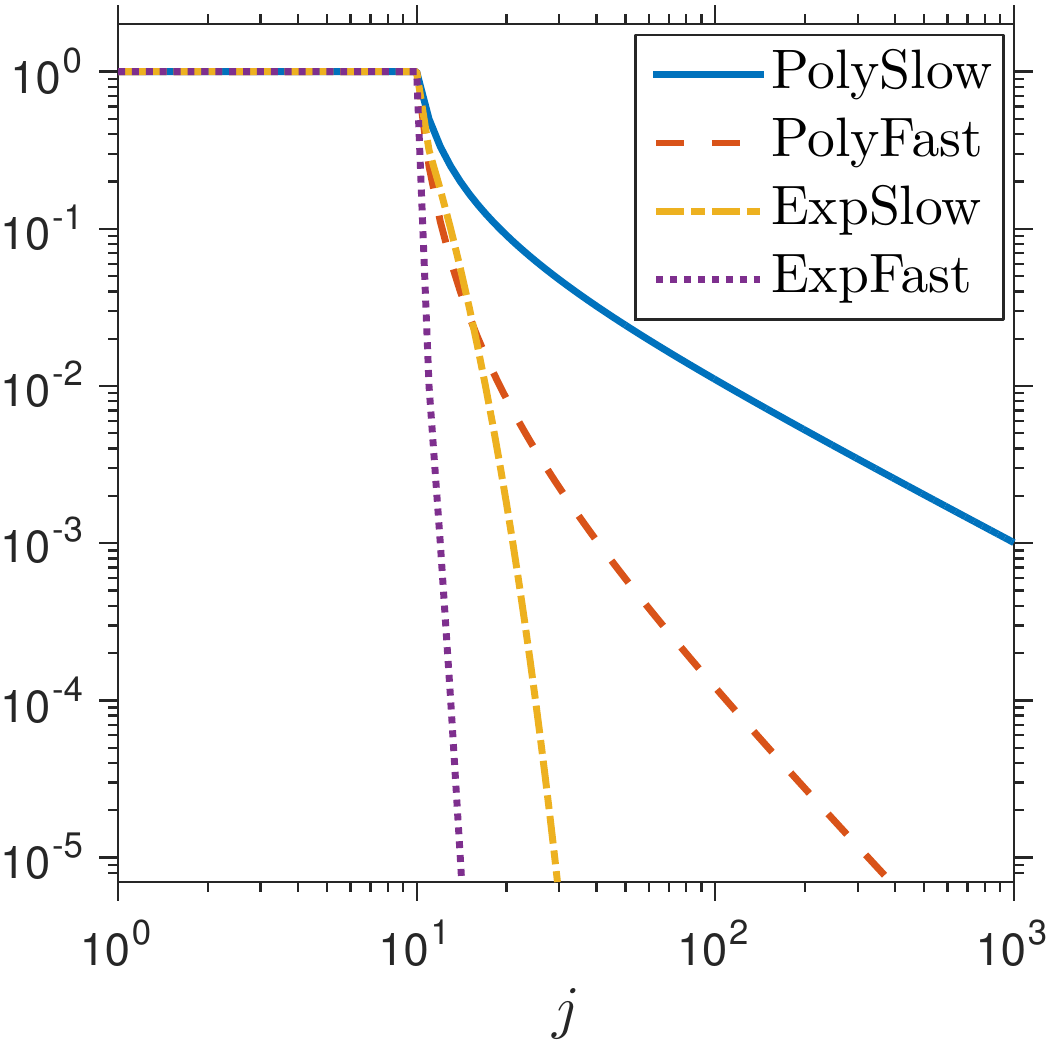} \hspace{0.5pc}
\includegraphics[height=1.5in]{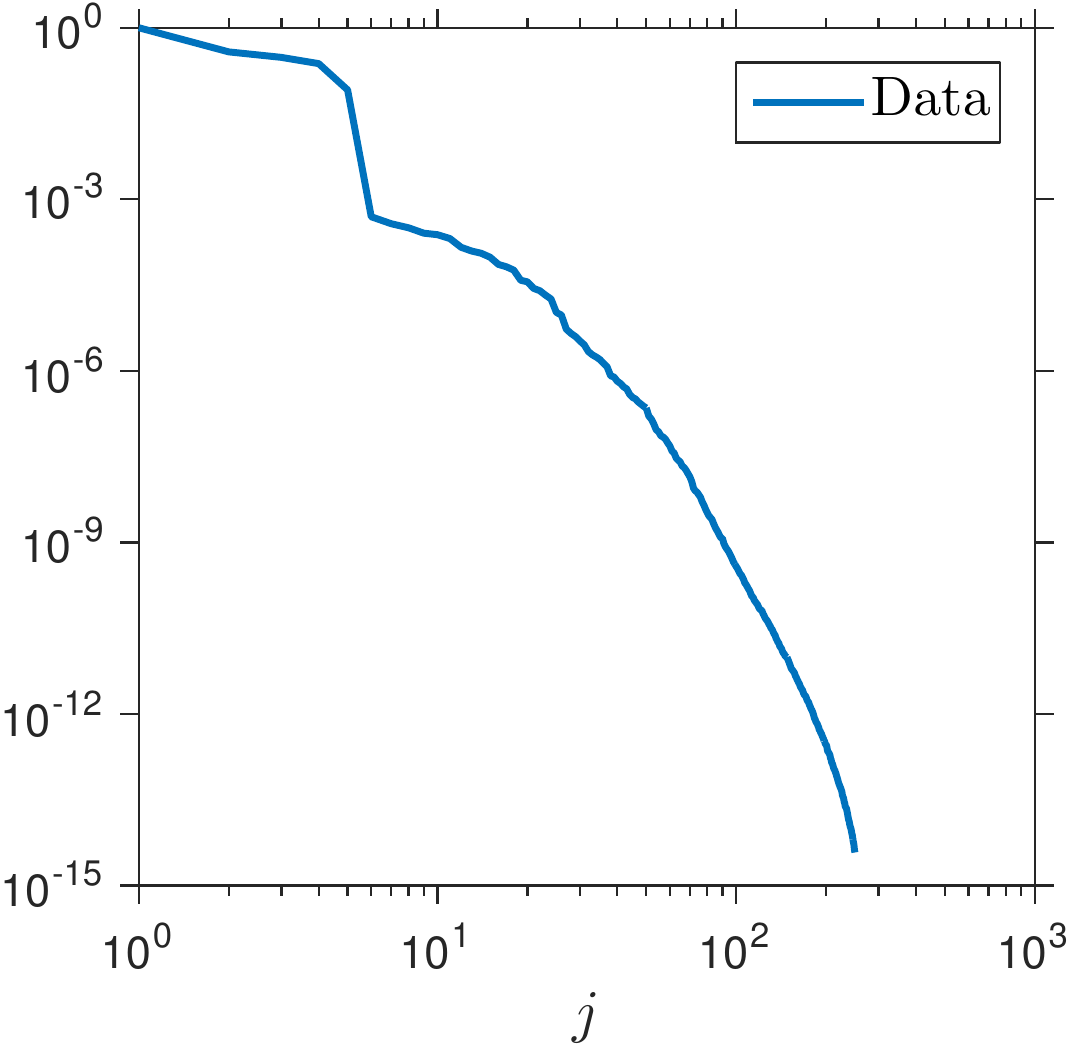}
\caption{\textbf{Spectra of input matrices.}
These plots display the singular value spectrum for an input
matrix from each of the classes (\texttt{LowRank}, \texttt{LowRankMedNoise}, \texttt{LowRankHiNoise},
\texttt{PolyDecaySlow}, \texttt{PolyDecayFast}, \texttt{ExpDecaySlow}, \texttt{ExpDecayFast}, \texttt{Data})
described in \cref{sec:input-matrix-examples}.}
\label{fig:singVal}
\end{center}
\end{figure}

\subsubsection{Synthetic Examples}

We fix a dimension parameter $n = 10^3$ and
a parameter $R = 10$ that controls the rank of
the ``significant part'' of the matrix.
In our experiments, we compute
approximations with target rank $r = 5$.
Similar results hold when the parameter $R = 5$
and when $n = 10^4$.

We construct the following synthetic input matrices:

\vspace{0.5pc}

\begin{enumerate}
\item	\textbf{Low-Rank + Noise:}  These matrices take the form
$$
\begin{bmatrix} \Id_R & \mtx{0} \\ \mtx{0} & \mtx{0} \end{bmatrix}
	\quad+\quad \sqrt{\frac{\gamma R}{2n^2}} (\mtx{G} + \mtx{G}^*)
	\quad\in\quad \C^{n \times n}.
$$
The matrix $\mtx{G}$ is complex standard normal.  The quantity $\gamma^{-2}$
can be interpreted as the signal-to-noise ratio (SNR).  We consider three cases:

\vspace{0.5pc}

\begin{enumerate}
\item	\textbf{No noise (\texttt{LowRank}):}  $\gamma = 0$.

\item	\textbf{Medium noise (\texttt{LowRankMedNoise}):}  $\gamma = 10^{-2}$.

\item	\textbf{High noise (\texttt{LowRankHiNoise}):}  $\gamma = 1$.
\end{enumerate}

\vspace{0.5pc}

\noindent
For these models, all the experiments are performed on a single exemplar
that is drawn at random and then fixed.

\vspace{0.5pc}

\item	\textbf{Polynomially Decaying Spectrum:}  These matrices take the form
$$
\diag\big( \underbrace{1,\ \dots,\ 1}_R,\ 2^{-p},\ 3^{-p},\ 4^{-p},\ \dots,\ (n-R+1)^{-p} \big) \quad\in\quad \C^{n \times n},
$$
where $p > 0$ controls the rate of decay.  We consider two cases:

\vspace{0.5pc}

\begin{enumerate}
\item	\textbf{Slow polynomial decay (\texttt{PolyDecaySlow}):} $p = 1$.
\item	\textbf{Fast polynomial decay (\texttt{PolyDecayFast}):} $p = 2$.
\end{enumerate}

\vspace{0.5pc}

\item	\textbf{Exponentially Decaying Spectrum:}  These matrices take the form
$$
\diag\big( \underbrace{1,\ \dots,\ 1}_R,\ 10^{-q},\ 10^{-2q},\ 10^{-3q},\ \dots,\ 10^{- (n-R) q} \big) \quad\in\quad \C^{n \times n},
$$
where $q > 0$ controls the rate of decay.  We consider two cases:

\vspace{0.5pc}

\begin{enumerate}
\item	\textbf{Slow exponential decay (\texttt{ExpDecaySlow}):} $q = 0.25$.
\item	\textbf{Fast exponential decay (\texttt{ExpDecayFast}):} $q = 1$.
\end{enumerate}
\end{enumerate}

\vspace{0.5pc}

We can focus on diagonal matrices because of the rotational invariance of the
test matrices $(\mtx{\Omega}, \mtx{\Psi})$.  Results for dense matrices
are similar.

\subsubsection{A Matrix from an Application in Optimization}
\label{eqn:cgm-matrix}

We also consider a dense, complex psd matrix (\texttt{Data})
obtained from a real-world phase retrieval application.
This matrix has dimension $n = 25,921$ and exact rank 250.
The first five singular values decrease from 1 to around 0.1;
there is a large gap between the fifth and sixth singular value;
the remaining nonzero singular values decay very fast.
See our paper~\cite{YUTC16:Sketchy-Decisions} for more details
about the role of sketching in this context.

\subsection{Alternative Sketching Algorithms for Matrix Approximation}
\label{sec:other-algs}

In addition to the algorithms we have presented,
our numerical study comprises other methods that have appeared in the literature.  We have modified
all of these algorithms to improve their numerical
stability and to streamline the computations.
To the extent possible, we adopt
the sketch \cref{eqn:sketches} for all the algorithms to
make their performance more comparable.

\subsubsection{Methods Based on the Sketch \cref{eqn:sketches}}

We begin with two additional methods that use
the same sketch~\cref{eqn:sketches} as our algorithms.

First, let us describe a variant of a fixed-rank approximation scheme that
was proposed by Woodruff~\cite[Thm.~4.3, display 2]{Woo14:Sketching-Tool}.
First, form a matrix product and compute its orthogonal--triangular factorization:
$\mtx{\Psi Q} =: \mtx{UT}$ where $\mtx{U} \in \F^{\ell \times k}$
has orthonormal columns.
Then construct the rank-$r$ approximation
\begin{equation} \label{eqn:woodruff-fixed}
\hat{\mtx{A}}_{\mathrm{woo}}
	:= \mtx{Q} \mtx{T}^{\dagger} \lowrank{\mtx{U}^*\mtx{W}}{r}.
\end{equation}
Woodruff shows that $\hat{\mtx{A}}_{\rm woo}$ satisfies \cref{eqn:eps-subopt}
when the sketch size scales as $k = \Theta(r/\eps)$
and $\ell = \Theta(k/\eps^2)$.
Compare this result with \cref{rem:fixed-rank-prior}.

Second, we outline a fixed-rank approximation method that
is implicit in Cohen et al.~\cite[Sec.~10.1]{CEM+15:Dimensionality-Reduction}.
First, compute the $r$ dominant left singular vectors
of the range sketch: $(\mtx{V}, \sim, \sim) := \texttt{svds}(\mtx{Y}, r)$.
Form a matrix product and compute its orthogonal--triangular factorization:
$\mtx{\Psi V} =: \mtx{UT}$ where $\mtx{U} \in \F^{\ell \times r}$.
Then form the rank-$r$ approximation
\begin{equation} \label{eqn:cohen-fixed}
\hat{\mtx{A}}_{\mathrm{cemmp}}
	:= \mtx{V} \mtx{T}^{\dagger} \lowrank{\mtx{U}^*\mtx{W}}{r}.
\end{equation}
The results in Cohen et al.~imply that $\hat{\mtx{A}}_{\mathrm{cemmp}}$
satisfies \cref{eqn:eps-subopt} when the sketch size scales
as $k = \Theta(r/\eps^2)$ and $\ell = \Theta(r/\eps^2)$.

The approximations \cref{eqn:woodruff-fixed,eqn:cohen-fixed} both appear
similar to our fixed-rank approximation, \cref{alg:fixed-rank-recon}. Nevertheless, they are derived from other principles,
and their behavior is noticeably different.

\subsubsection{A Method Based on an Extended Sketch}
\label{sec:optimal-alg}

Next, we present a variant of a recent approach that requires
a more complicated sketch and more elaborate computations.
The following procedure is adapted from \cite[Thm.~12]{BWZ16:Optimal-Principal-STOC},
using simplifications suggested in~\cite[Sec.~3]{Upa16:Fast-Space-Optimal}.

Let $\mtx{A} \in \F^{m \times n}$ be an input matrix,
and let $r$ be a target rank.  Choose integer parameters $k$ and $s$
that satisfy $r \leq k \leq s \leq \min\{m,n\}$.
For consistent notation, we also introduce a redundant parameter $\ell = k$.
Draw and fix \emph{four} test matrices:
\begin{equation} \label{eqn:fancy-test}
\mtx{\Psi} \in \F^{k \times m}; \quad
\mtx{\Omega} \in \F^{n \times \ell}; \quad
\mtx{\Phi} \in \F^{s \times m}; \quad\text{and}\quad
\mtx{\Xi} \in \F^{n \times s}.
\end{equation}
The matrices $(\mtx{\Psi}, \mtx{\Omega})$ are standard normal,
while $(\mtx{\Phi}, \mtx{\Xi})$ are SRFTs; see \cref{sec:distributions}.
The sketch now has \emph{three} components:
\begin{equation} \label{eqn:fancy-sketch}
\mtx{W} := \mtx{\Psi A}; \quad
\mtx{Y} := \mtx{A\Omega}; \quad\text{and}\quad
\mtx{Z} := \mtx{\Phi A \Xi}.
\end{equation}
To store the test matrices and the sketch, we need
$(2k + 1)(m + n) + s(s+2)$ numbers.

To obtain a rank-$r$ approximation of the input matrix $\mtx{A}$,
first compute four thin orthogonal--triangular factorizations:
$$
\begin{aligned}
\mtx{Y} &=: \mtx{Q}_1 \mtx{R}_1
&& \quad\text{and}\quad &
\mtx{W} &=: \mtx{R}_2^* \mtx{Q}_2^*; \\
\mtx{\Phi Q}_1 &=: \mtx{U}_1 \mtx{T}_1
&& \quad\text{and}\quad &
\mtx{Q}_2^* \mtx{\Xi} &=: \mtx{T}_2^* \mtx{U}_2^*.
\end{aligned}
$$
Then construct the rank-$r$ approximation
\begin{equation} \label{eqn:bwz}
\hat{\mtx{A}}_{\mathrm{bwz}} :=
	\mtx{Q}_1 \mtx{T}_1^{\dagger} \lowrank{\mtx{U}_1^* \mtx{Z} \mtx{U}_2}{r}
	(\mtx{T}_2^{*})^{\dagger} \mtx{Q}_2^*.
\end{equation}
By adapting and correcting~\cite[Thm.~12]{BWZ16:Optimal-Principal-STOC},
one can show that $\hat{\mtx{A}}_{\mathrm{bwz}}$ achieves \cref{eqn:eps-subopt}
for sketch size parameters that satisfy $k = \Theta(r/\eps)$
and $s = \Theta((r\log(1+ r))^2 / \eps^6)$.
With this scaling, the total storage cost for the random matrices
and the sketch is $\Theta((m+n)r/\eps + (r\log(1+r))^2 / \eps^6)$.

The authors of \cite{BWZ16:Optimal-Principal-STOC} refer to their method
as ``optimal'' because the scaling of the term $(m+n) r/\eps$ in the storage cost
cannot be improved~\cite[Thm.~4.10]{CW09:Numerical-Linear}.
Nevertheless, because of the $\eps^{-6}$ term, the bound is
incomparable with the storage costs achieved by other algorithms.

\subsection{Performance with Oracle Parameter Choices}
\label{sec:oracle-params}

It is challenging to compare the relative performance of
sketching algorithms for matrix approximation
because of the theoretical nature of previous research.
In particular, earlier work does not offer any practical
guidance for selecting the sketch size parameters.

The only way to make a fair comparison is to study
the \emph{oracle performance} of the algorithms.
That is, for each method, we fix the total storage,
and we determine the minimum relative error that the
algorithm can achieve.  This approach allows us to
see which techniques are most promising for further
development.  Nevertheless, we must emphasize that
the oracle performance is not achievable in practice.

\subsubsection{Computing the Oracle Error}
\label{sec:oracle-error}

It is straightforward to compare our fixed-rank approximation methods, \cref{alg:fixed-rank-recon,alg:sym-fixed-rank-recon,alg:psd-fixed-rank-recon},
with the alternatives \cref{eqn:woodruff-fixed,eqn:cohen-fixed} from the literature.
In each case, the sketch \cref{eqn:sketches} requires storage
of $(k + \ell)(m + n)$ numbers, so we can parameterize
the cost by $T := k + \ell$.  For a given choice of $T$,
we obtain the oracle performance by minimizing
the empirical approximation error for each
algorithm over all pairs $(k, \ell)$ where the sum $k + \ell = T$.

It is trickier to include the Boutsidis et al.~\cite[Thm.~12]{BWZ16:Optimal-Principal-STOC} method~\cref{eqn:bwz}.
For a given $T$, we obtain the oracle performance of \cref{eqn:bwz}
by minimizing the empirical approximation error
over pairs $(k, s)$ for which the storage cost of
the sketch \cref{eqn:fancy-sketch}
matches the cost of the simple sketch \cref{eqn:sketches}.
That is, $(2k+1)(m+n) + s(s+2) \approx T(m + n)$.

\subsubsection{Numerical Comparison with Prior Work}

For each input matrix described in \cref{sec:input-matrix-examples},
\cref{fig:oracle-performance} compares the oracle performance of
our fixed-rank approximation, \cref{alg:fixed-rank-recon},
against several alternative methods \cref{eqn:woodruff-fixed,eqn:cohen-fixed,eqn:bwz}
from the literature.  We make the following observations:

\vspace{0.5pc}

\begin{itemize} \setlength{\itemsep}{0.5pc}
\item	For matrices that are well-approximated by a low-rank matrix
(\texttt{LowRank}, \texttt{PolyDecayFast}, \texttt{ExpDecaySlow}, \texttt{ExpDecayFast}, \texttt{Data}),
our fixed-rank approximation, \cref{alg:fixed-rank-recon}, dominates
all other methods when the storage budget is adequate.
In particular, for the rank-1 approximation of the matrix
\texttt{Data}, our approach achieves relative errors 3--6 orders of magnitude
better than any competitor.

\item	When we consider matrices that are poorly approximated by a low-rank matrix
(\texttt{LowRankMedNoise}, \texttt{LowRankHiNoise}, \texttt{PolyDecaySlow}), the recent method \cref{eqn:bwz}
of Boutsidis et al.~\cite[Thm.~12]{BWZ16:Optimal-Principal-STOC} has the best performance, especially when the storage budget is small.
But see \cref{sec:structured-approx} for more texture.

\item	Our method, \cref{alg:fixed-rank-recon}, performs reliably
for all of the input matrices, and it is the only method that can achieve
high accuracy for the matrix \texttt{Data}.  Its behavior is less impressive for matrices that
have poor low-rank approximations (\texttt{LowRankMedNoise}, \texttt{LowRankHiNoise}, \texttt{PolyDecaySlow}),
but it is still competitive for these examples.

\item	The method \cref{eqn:bwz} of Boutsidis et al.~\cite[Thm.~12]{BWZ16:Optimal-Principal-STOC} offers mediocre performance for
matrices with good low-rank approximations (\texttt{LowRank}, \texttt{ExpDecaySlow}, \texttt{ExpDecayFast}, \texttt{Data}).
Strikingly, this approach fails to produce a high-accuracy rank-5 approximation
of the rank-10 matrix \texttt{LowRank}, even with a large storage budget.

\item	The method \cref{eqn:woodruff-fixed} of Woodruff~\cite[Thm.~4.3, display 2]{Woo14:Sketching-Tool}
is competitive for most synthetic examples, but it performs rather poorly on the matrix \texttt{Data}.

\item	The method \cref{eqn:cohen-fixed} of Cohen et al.~\cite[Sec.~10.1]{CEM+15:Dimensionality-Reduction}
has the worst performance for almost all the examples.
\end{itemize}

\vspace{0.5pc}

In summary, \cref{alg:fixed-rank-recon} has the best all-around behavior, while
the Boutsidis et al.~\cite[Thm.~12]{BWZ16:Optimal-Principal-STOC} method~\cref{eqn:bwz}
works best for matrices that have
a poor low-rank approximation.  See \cref{sec:recommendations} for more discussion.

\subsubsection{Structured Approximations}
\label{sec:structured-approx}

In this section, we investigate the effect of imposing structure
on the low-rank approximations.   \Cref{fig:oracle-structured}
compares the oracle performance of our fixed-rank approximation
methods, \cref{alg:fixed-rank-recon,alg:sym-fixed-rank-recon,alg:psd-fixed-rank-recon}.
We make the following observations:

\vspace{0.5pc}

\begin{itemize} \setlength{\itemsep}{0.5pc}
\item	The symmetric approximation method, \cref{alg:sym-fixed-rank-recon},
and the psd approximation method, \cref{alg:psd-fixed-rank-recon}, are very
similar to each other for all examples.

\item	The structured approximations, \cref{alg:sym-fixed-rank-recon,alg:psd-fixed-rank-recon},
always improve on the unstructured approximation, \cref{alg:fixed-rank-recon}.
The benefit is most significant for matrices
that have a poor low-rank approximation
(\texttt{LowRankMedNoise}, \texttt{LowRankHiNoise}, \texttt{PolyDecaySlow}).

\item	\Cref{alg:sym-fixed-rank-recon,alg:psd-fixed-rank-recon} match or
exceed the performance of the Boutsidis et al.~\cite[Thm.~12]{BWZ16:Optimal-Principal-STOC}
method \cref{eqn:bwz} for all examples.
\end{itemize}

\vspace{0.5pc}

In summary, if we know that the input matrix has structure, we can achieve
a decisive advantage by enforcing the structure in the approximation.

\subsection{Performance with Theoretical Parameter Choices}
\label{sec:theory-params}

It remains to understand how closely we can match the oracle performance
of \cref{alg:fixed-rank-recon,alg:sym-fixed-rank-recon,alg:psd-fixed-rank-recon}
in practice.  To that end, we must choose the sketch size parameters \emph{a priori}
using only the knowledge of the target rank $r$ and the total sketch size $T$.
In some instances, we may also have insight about the spectral decay of the
input matrix.
\Cref{fig:theory-params} shows how the fixed-rank approximation method, \cref{alg:fixed-rank-recon},
performs with the theoretical parameter choices outlined in \cref{sec:best-parameters}.
We make the following observations:

\vspace{0.5pc}

\begin{itemize} \setlength{\itemsep}{0.5pc}
\item	The parameter recommendation \cref{eqn:flat-params}, designed for a matrix
with a flat spectral tail, works well for the matrices
\texttt{LowRankMedNoise}, \texttt{LowRankHiNoise}, and \texttt{PolyDecaySlow}.
We also learn that this parameter choice should not be used for matrices with spectral decay.

\item	The parameter recommendation \cref{eqn:decay-params}, for a matrix
with a slowly decaying spectrum, is suited to the examples
\texttt{LowRankMedNoise}, \texttt{LowRankHiNoise}, \texttt{PolyDecaySlow}, and \texttt{PolyDecayFast}.
This parameter choice is effective for the remaining examples as well.

\item	The parameter recommendation \cref{eqn:fast-decay-params}, for a matrix
with a rapidly decaying spectrum, is appropriate for the examples
\texttt{PolyDecayFast}, \texttt{ExpDecaySlow}, \texttt{ExpDecayFast}, and \texttt{Data}.
This choice must not be used unless the spectrum decays quickly.

\item	We have observed that the same parameter recommendations allow us to achieve
near-oracle performance for the structured matrix approximations,
\cref{alg:sym-fixed-rank-recon,alg:psd-fixed-rank-recon}.  As in the unstructured
case, it helps if we tune the parameter choice to the type of input matrix.

\end{itemize}

\vspace{0.5pc}

In summary, we always achieve reasonably good performance using
the parameter choice \cref{eqn:decay-params}.  Furthermore,
if we match the parameter selection \cref{eqn:flat-params,eqn:decay-params,eqn:fast-decay-params}
to the spectral properties of the
input matrix, we can almost achieve the oracle performance in practice.

\subsection{Recommendations}
\label{sec:recommendations}

Among the fixed-rank approximation methods that we studied,
the most effective are \cref{alg:fixed-rank-recon,alg:sym-fixed-rank-recon,alg:psd-fixed-rank-recon}
and the Boutsidis et al.~\cite[Thm.~12]{BWZ16:Optimal-Principal-STOC} method \cref{eqn:bwz}.
Let us make some final observations based on our numerical experience.

\Cref{alg:fixed-rank-recon,alg:sym-fixed-rank-recon,alg:psd-fixed-rank-recon}
are superior to methods from the literature for input matrices that have
good low-rank approximations.  Although \cref{alg:fixed-rank-recon} suffers
when the input matrix has a poor low-rank approximation, the structured
variants, \cref{alg:sym-fixed-rank-recon,alg:psd-fixed-rank-recon}, match or
exceed other algorithms for all the examples we tested.  We have also
established that we can attain near-oracle performance for our methods
using the \emph{a priori} parameter recommendations from
\cref{sec:best-parameters}. Finally, our methods are simple and easy to
implement.

The Boutsidis et al.~\cite[Thm.~12]{BWZ16:Optimal-Principal-STOC} method \cref{eqn:bwz} exhibits the best performance
for matrices that have very poor low-rank approximations when the storage
budget is very small.  This benefit is diminished by its mediocre performance
for matrices that do admit good low-rank approximations.
The method \cref{eqn:bwz} requires more complicated sketches and additional computation.
Unfortunately, the analysis in~\cite{BWZ16:Optimal-Principal-STOC}
does not provide guidance on implementation.

In conclusion, we recommend the sketching methods,
\cref{alg:fixed-rank-recon,alg:sym-fixed-rank-recon,alg:psd-fixed-rank-recon},
for computing structured low-rank approximations.
In future research, we will try to design new methods
that simultaneously dominate our algorithms and \cref{eqn:bwz}.

\begin{figure}[htp!]
\begin{center}
\begin{subfigure}{.325\textwidth}
\begin{center}
\includegraphics[height=1.5in]{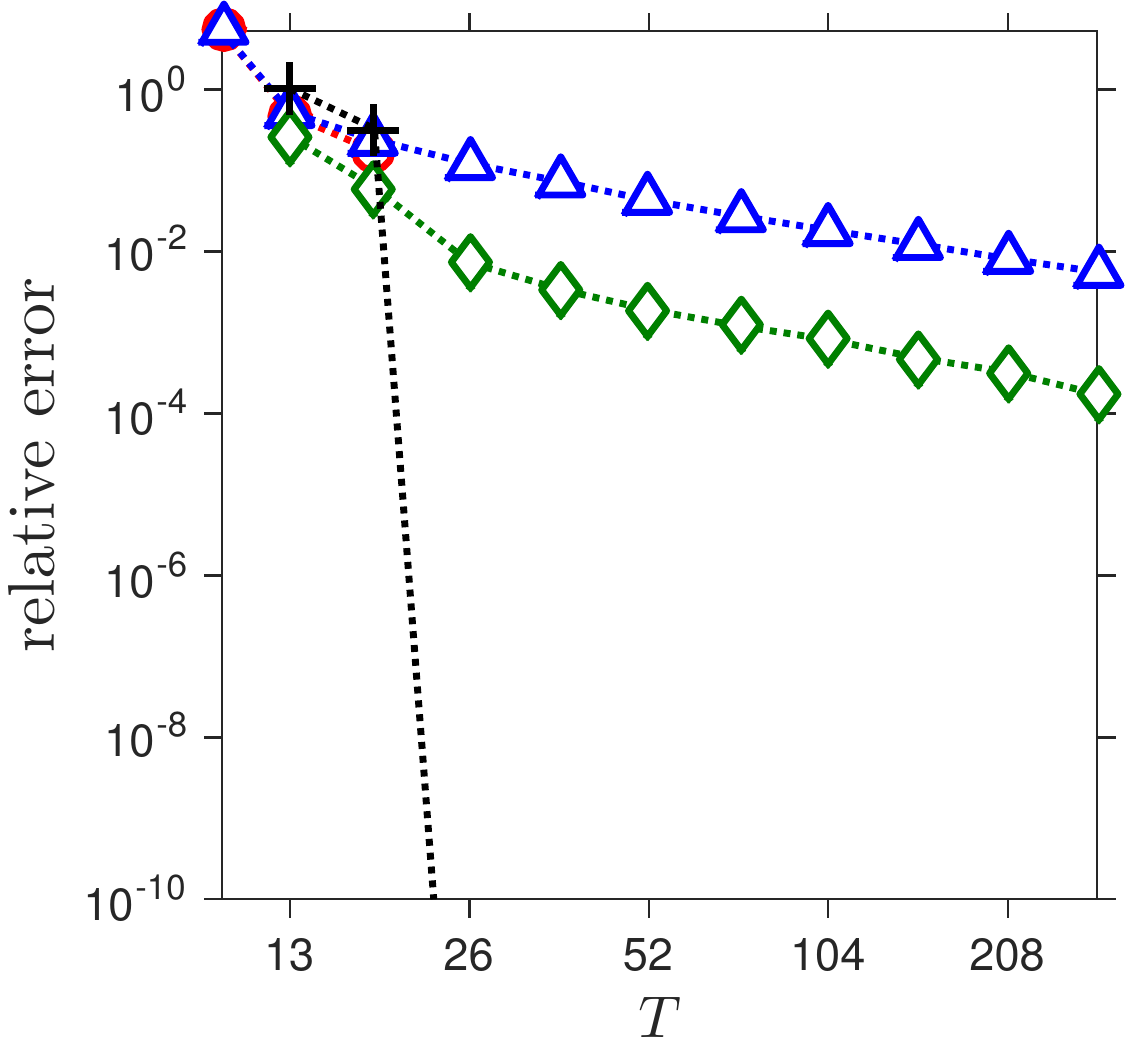}
\caption{\texttt{LowRank}}
\label{fig:LR-algs}
\end{center}
\end{subfigure}
\begin{subfigure}{.325\textwidth}
\begin{center}
\includegraphics[height=1.5in]{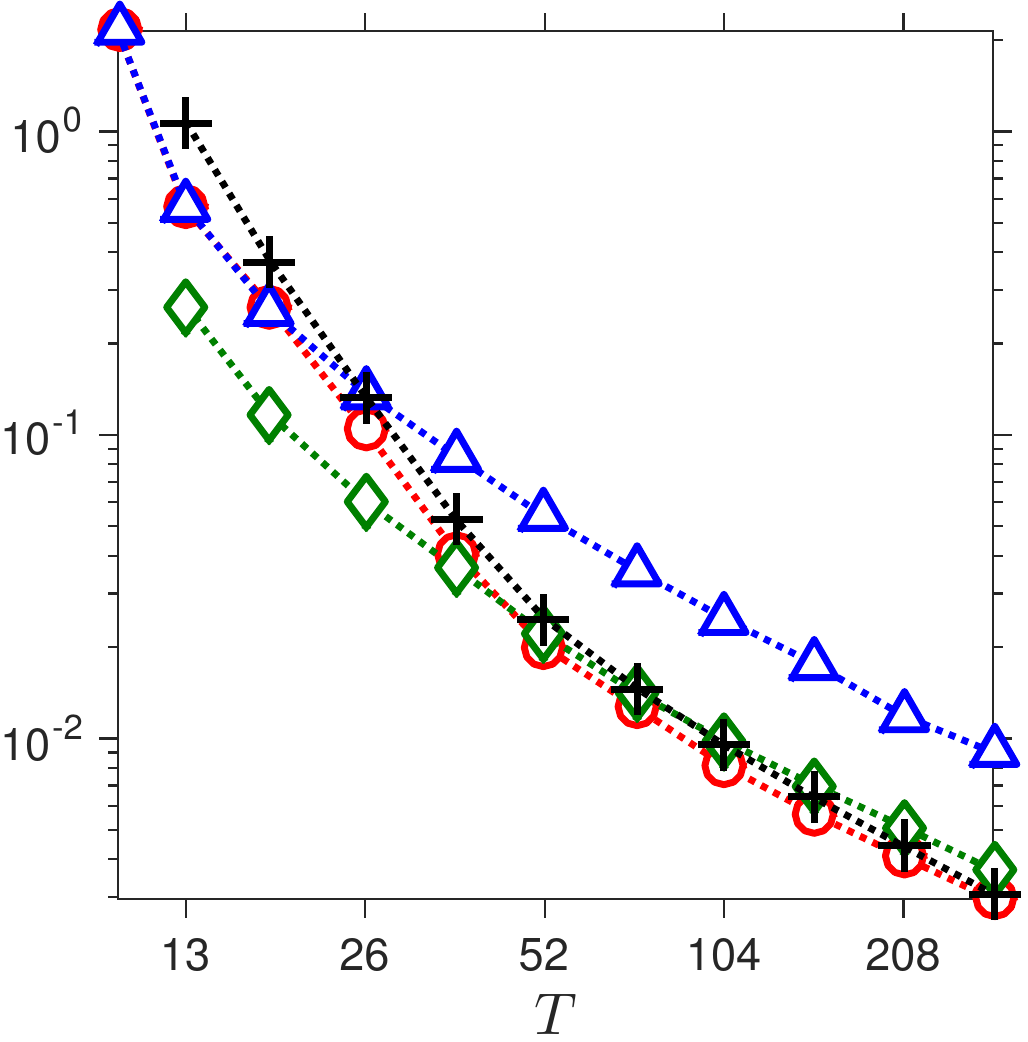}
\caption{\texttt{LowRankMedNoise}}
\label{fig:MED-algs}
\end{center}
\end{subfigure}
\begin{subfigure}{.325\textwidth}
\begin{center}
\includegraphics[height=1.5in]{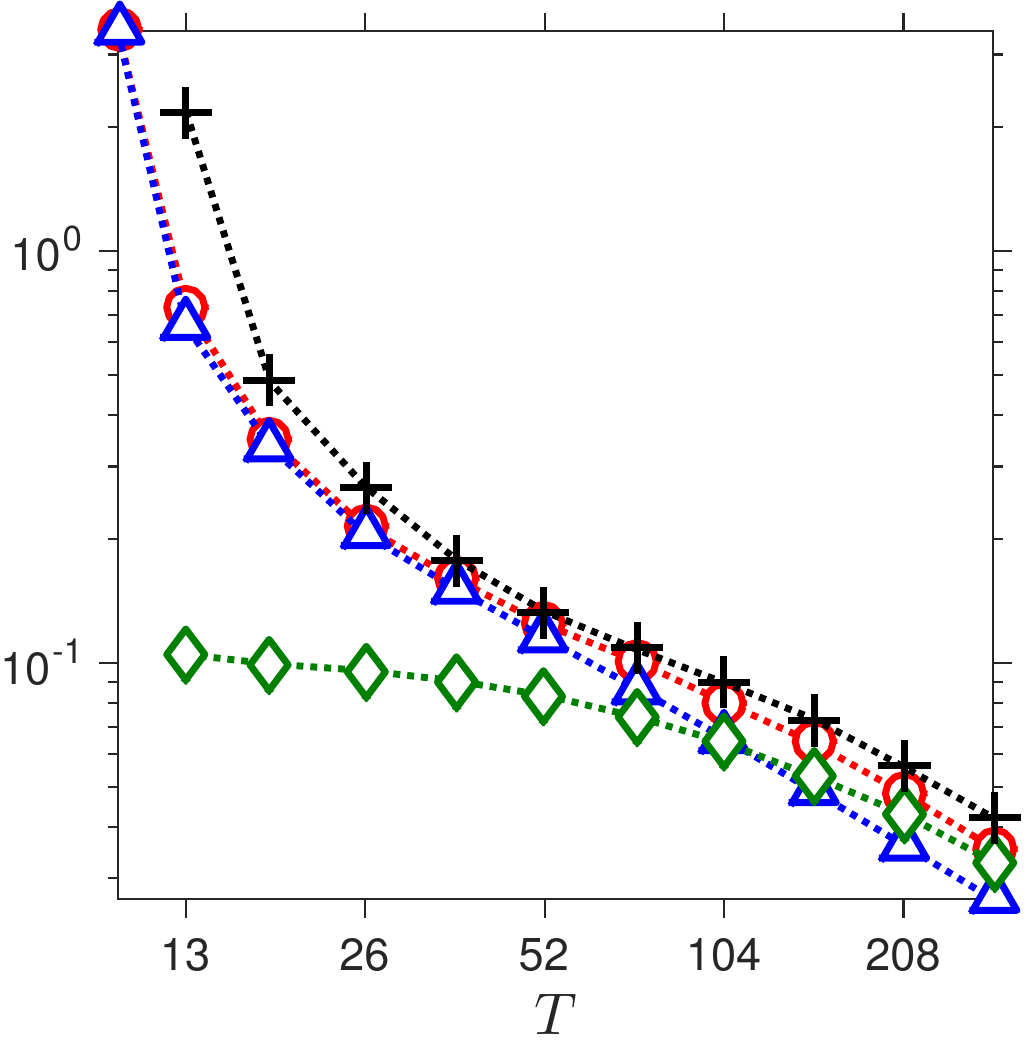}
\caption{\texttt{LowRankHiNoise}}
\label{fig:HI-algs}
\end{center}
\end{subfigure}
\end{center}

\vspace{.5em}

\begin{center}
\begin{subfigure}{.325\textwidth}
\begin{center}
\includegraphics[height=1.5in]{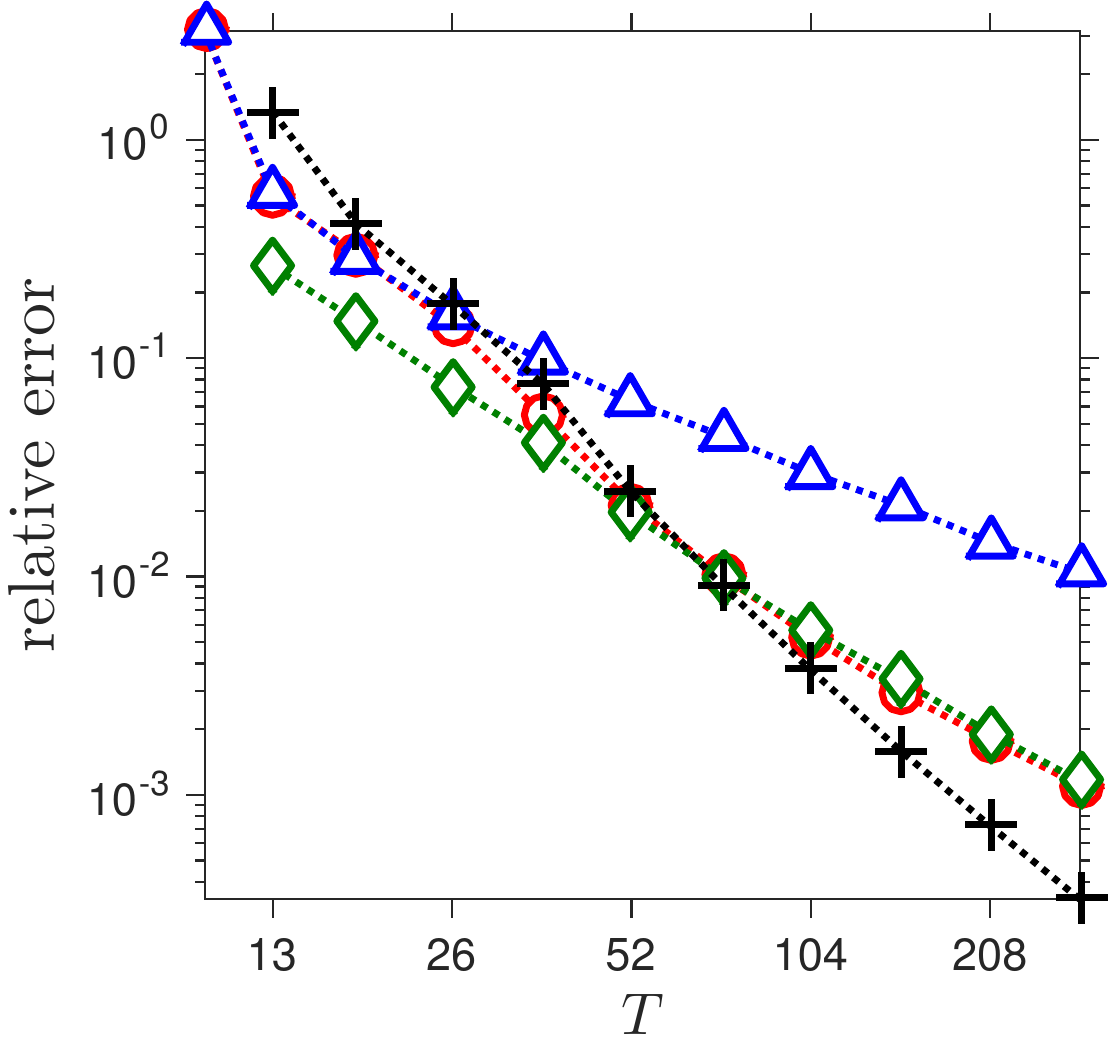}
\caption{\texttt{PolyDecaySlow}}
\label{fig:PSLOW-algs}
\end{center}
\end{subfigure}
\begin{subfigure}{.325\textwidth}
\begin{center}
\includegraphics[height=1.5in]{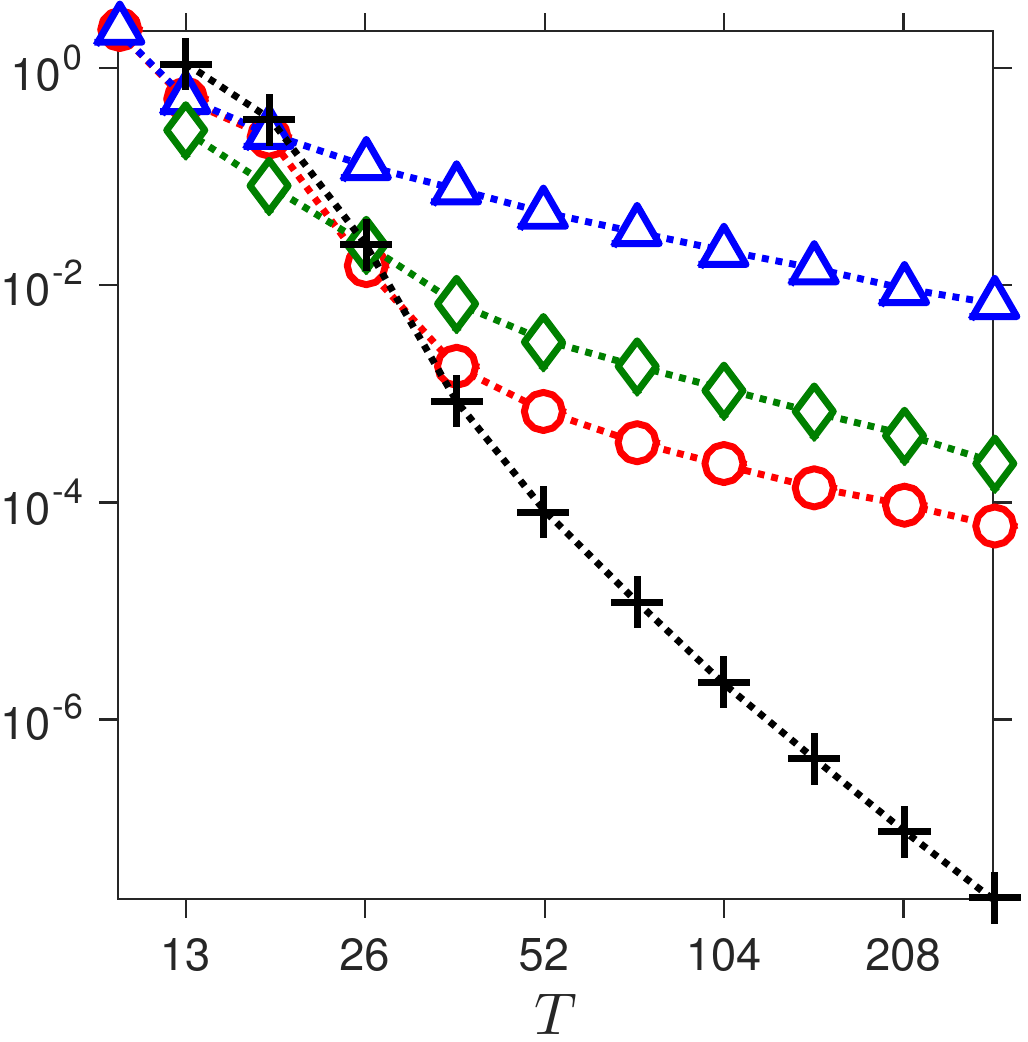}
\caption{\texttt{PolyDecayFast}}
\label{fig:PFAST-algs}
\end{center}
\end{subfigure}
\begin{subfigure}{.325\textwidth}
\begin{center}
\includegraphics[height=1.5in]{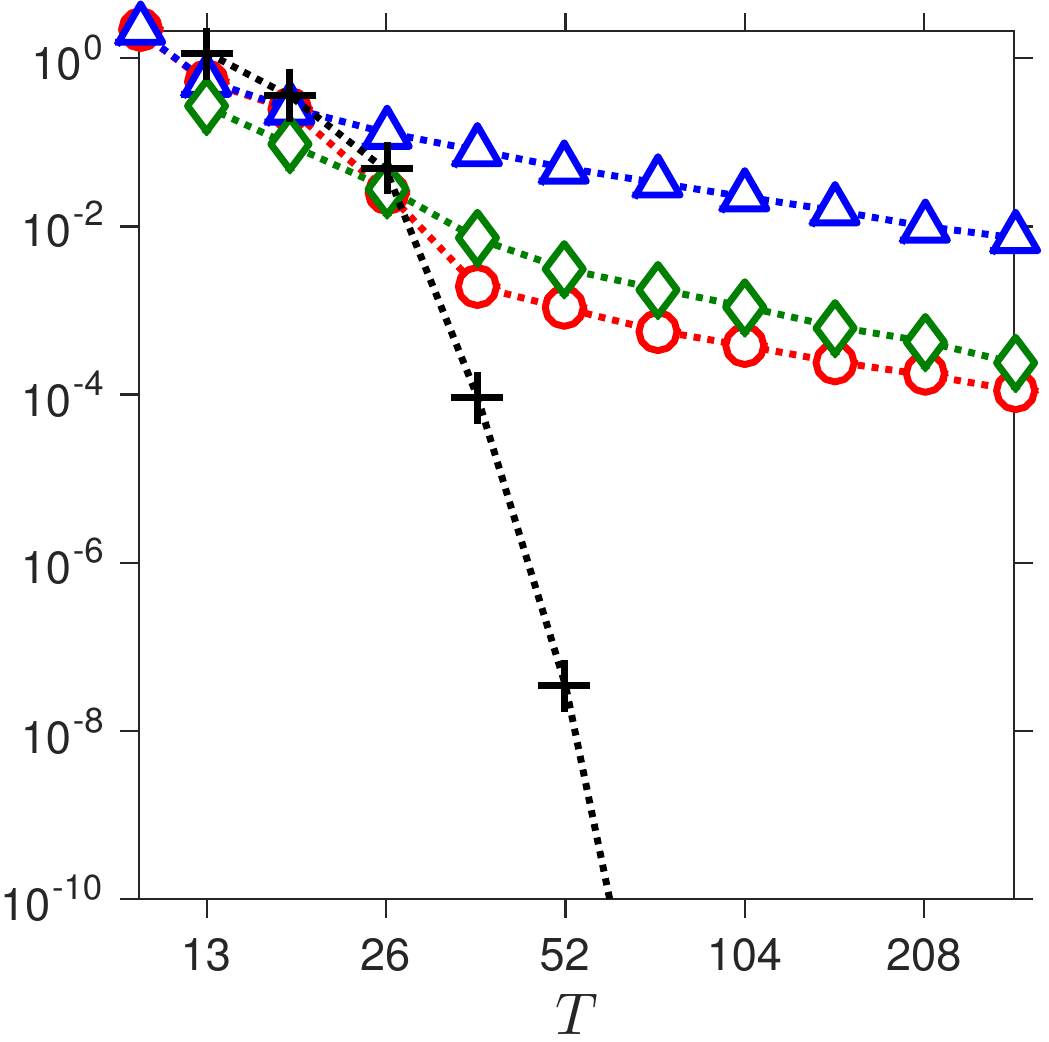}
\caption{\texttt{ExpDecaySlow}}
\label{fig:ESLOW-algs}
\end{center}
\end{subfigure}
\end{center}

\vspace{0.5em}

\begin{center}
\begin{subfigure}{.325\textwidth}
\begin{center}
\includegraphics[height=1.5in]{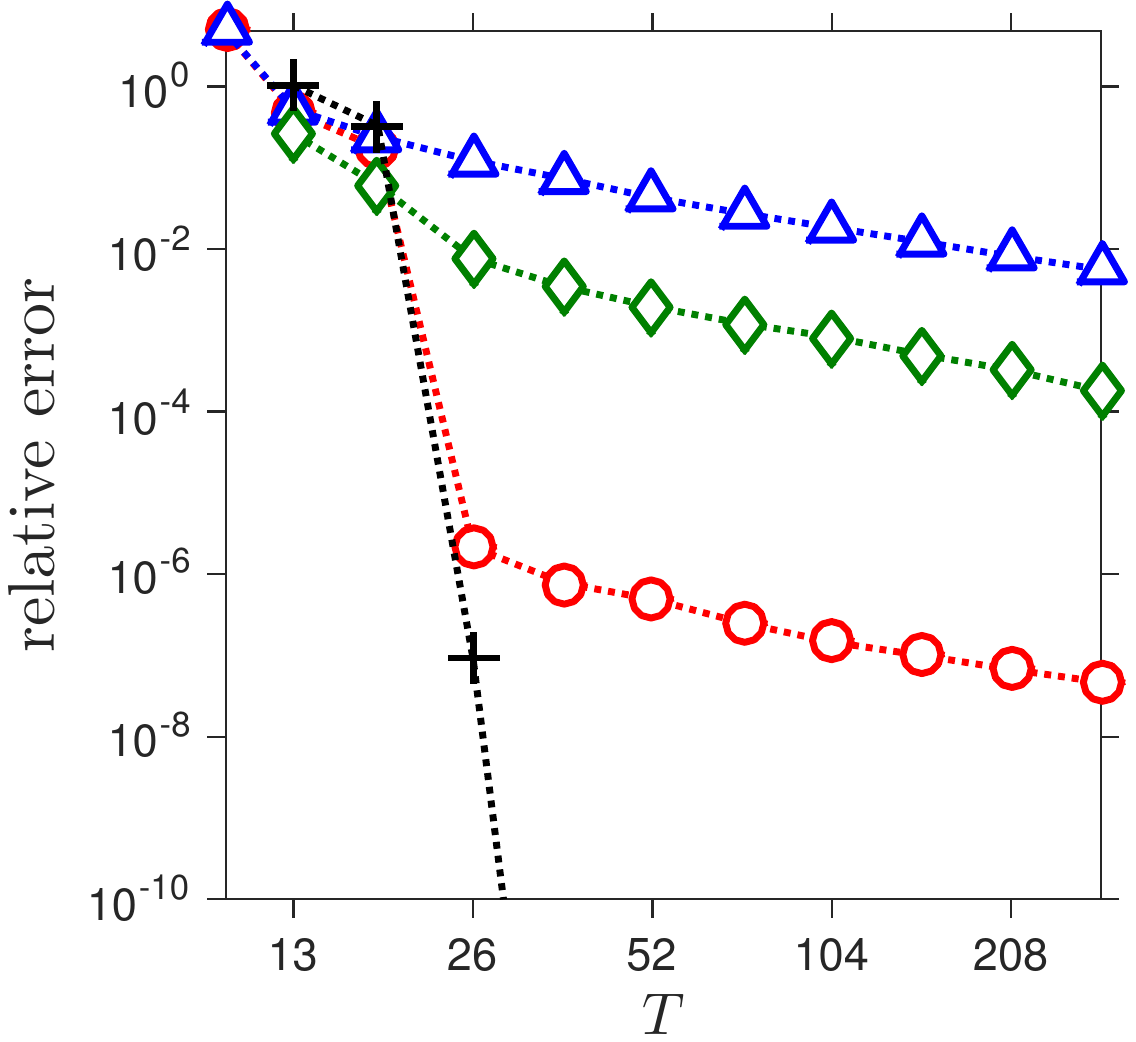}
\caption{\texttt{ExpDecayFast}}
\label{fig:EFAST-algs}
\end{center}
\end{subfigure}
\begin{subfigure}{.325\textwidth}
\begin{center}
\includegraphics[height=1.5in]{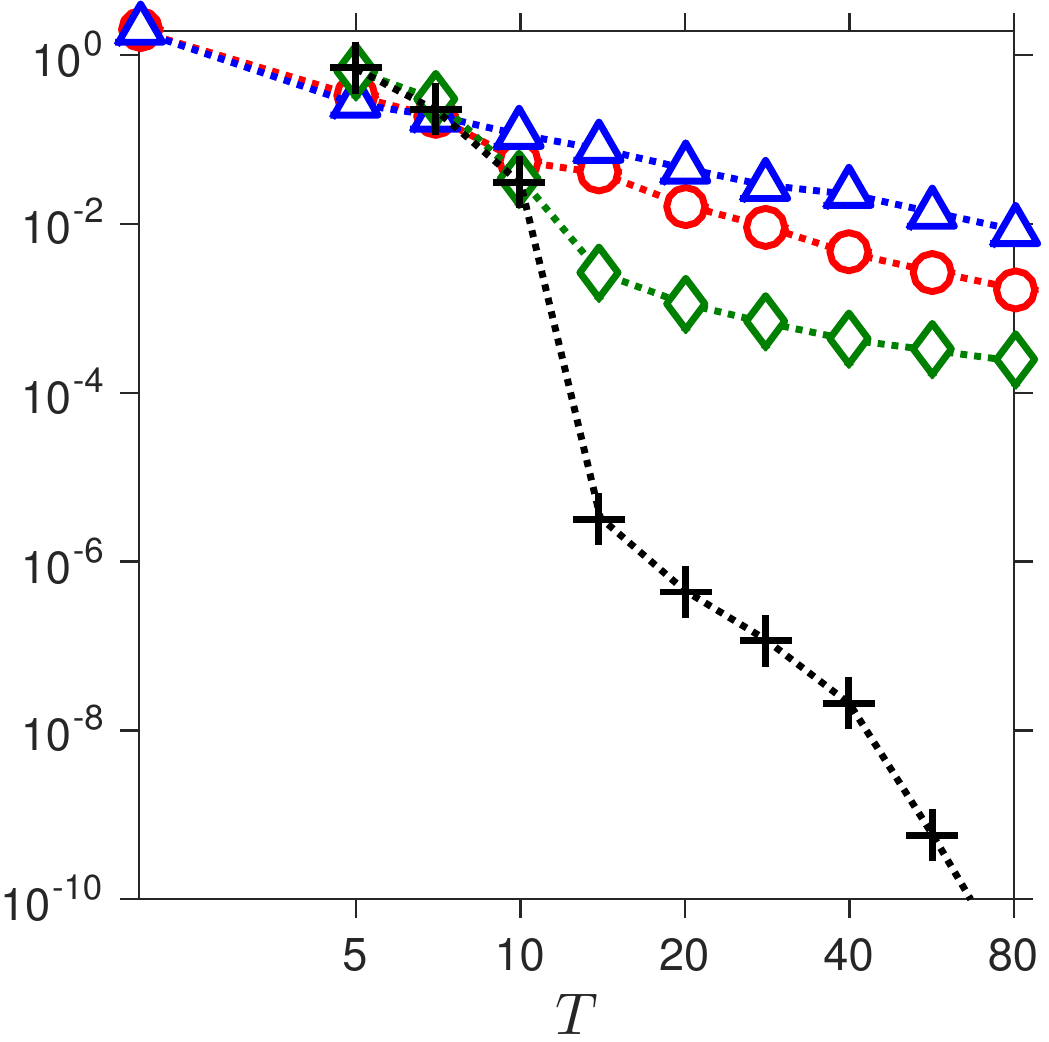}
\caption{\texttt{Data} ($r = 1$)}
\label{fig:DATA1-algs}
\end{center}
\end{subfigure}
\begin{subfigure}{.325\textwidth}
\begin{center}
\includegraphics[height=1.5in]{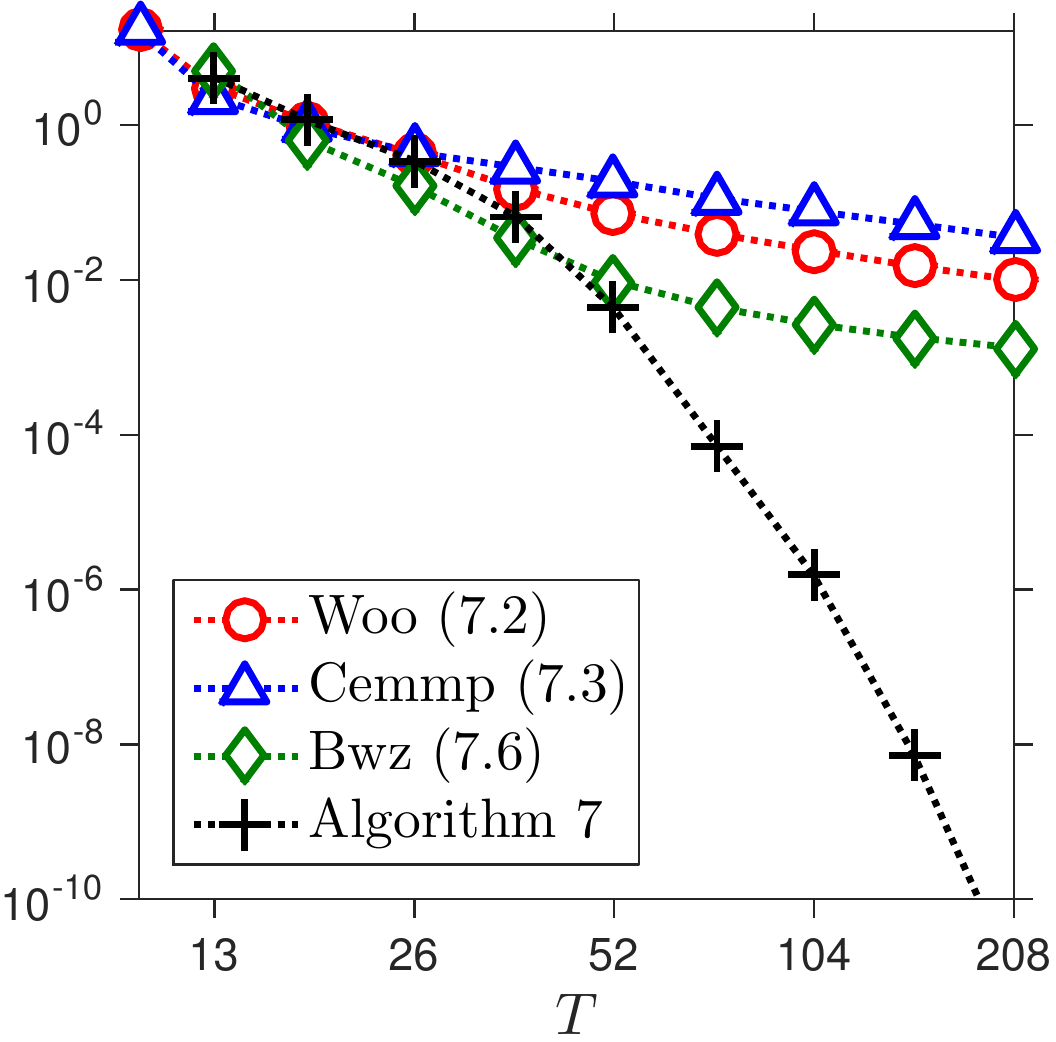}
\caption{\texttt{Data} ($r = 5$)}
\label{fig:DATA5-algs}
\end{center}
\end{subfigure}
\end{center}

\vspace{0.5em}

\caption{\textbf{Oracle performance of sketching algorithms for fixed-rank matrix approximation
as a function of storage cost.}
For each of the input matrices described in \cref{sec:input-matrix-examples},
we compare the oracle performance of our fixed-rank approximation,
\cref{alg:fixed-rank-recon}, against alternative methods
\cref{eqn:woodruff-fixed,eqn:cohen-fixed,eqn:bwz} from the literature.
The matrix dimensions are $m = n = 10^3$ for the synthetic examples
and $m = n = 25,921$ for the matrix \texttt{Data}
from the phase retrieval application.  Each approximation has rank $r = 5$, unless otherwise stated.
The variable $T$ on the horizontal axis is (proportional to) the total storage used by each
sketching method.  Each data series displays the best relative error
\cref{eqn:relative-error} that the specified algorithm can achieve with storage $T$.
See \cref{sec:oracle-error} for details.}
\label{fig:oracle-performance}
\end{figure}

\begin{figure}[htp!]
\begin{center}
\begin{subfigure}{.325\textwidth}
\begin{center}
\includegraphics[height=1.5in]{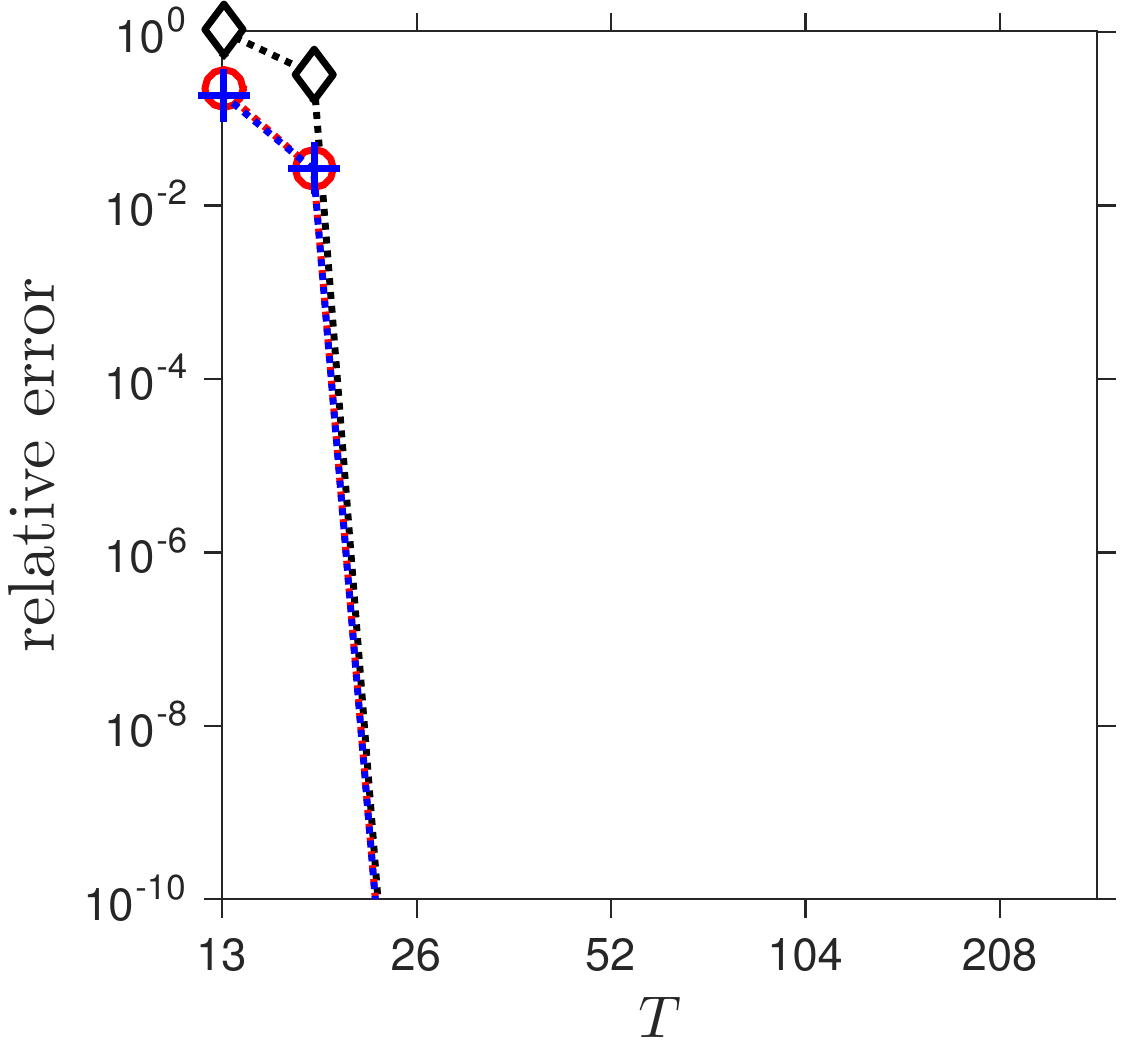}
\caption{\texttt{LowRank}}
\label{fig:LR-789}
\end{center}
\end{subfigure}
\begin{subfigure}{.325\textwidth}
\begin{center}
\includegraphics[height=1.5in]{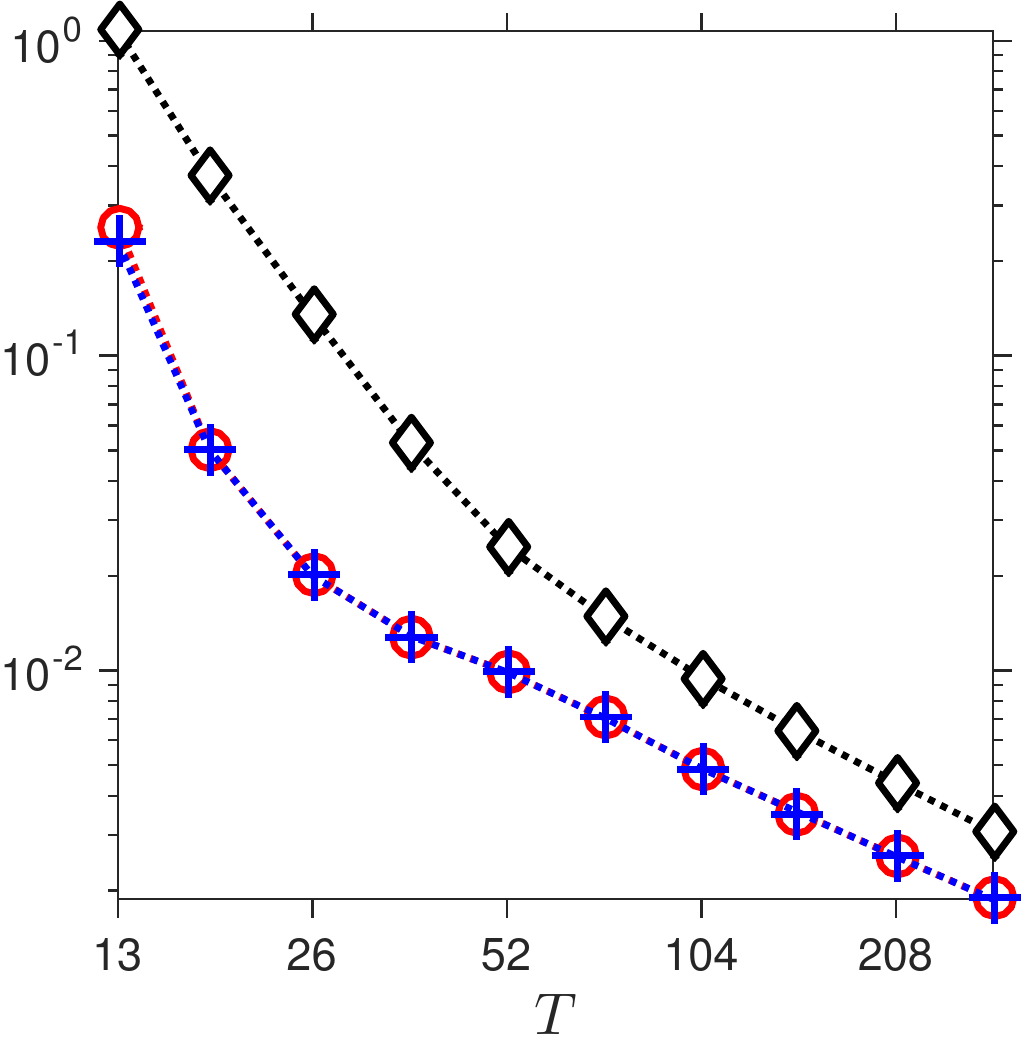}
\caption{\texttt{LowRankMedNoise}}
\label{fig:MED-789}
\end{center}
\end{subfigure}
\begin{subfigure}{.325\textwidth}
\begin{center}
\includegraphics[height=1.5in]{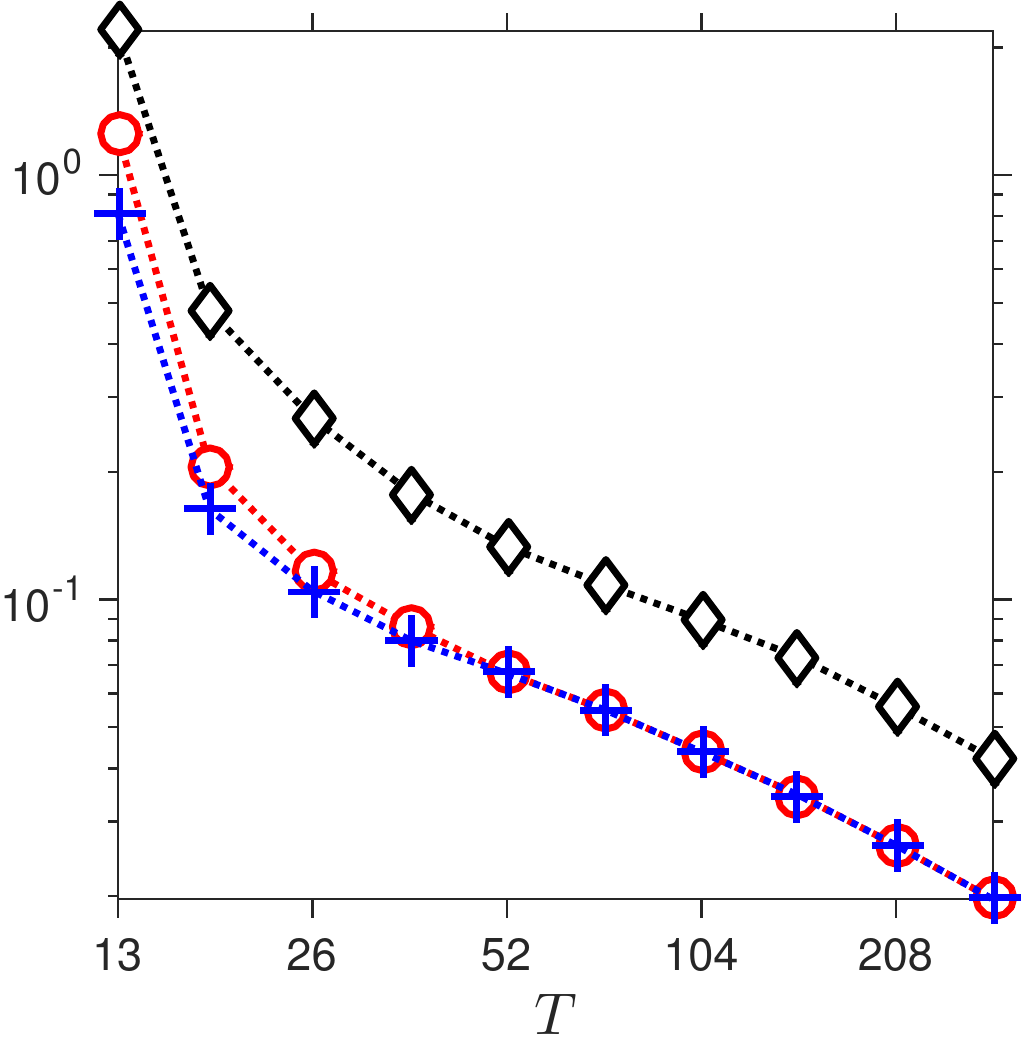}
\caption{\texttt{LowRankHiNoise}}
\label{fig:HI-789}
\end{center}
\end{subfigure}
\end{center}

\vspace{.5em}

\begin{center}
\begin{subfigure}{.325\textwidth}
\begin{center}
\includegraphics[height=1.5in]{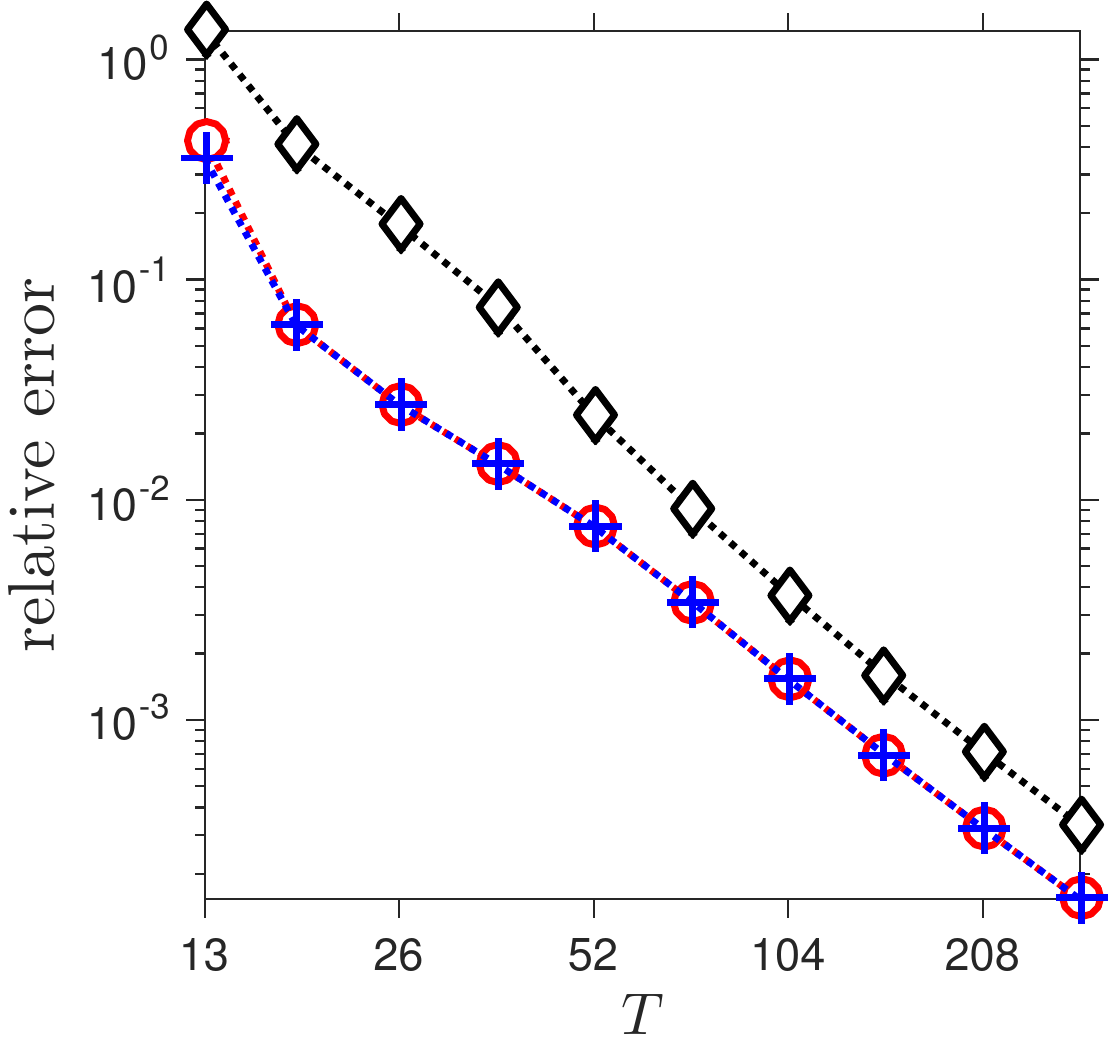}
\caption{\texttt{PolyDecaySlow}}
\label{fig:PSLOW-789}
\end{center}
\end{subfigure}
\begin{subfigure}{.325\textwidth}
\begin{center}
\includegraphics[height=1.5in]{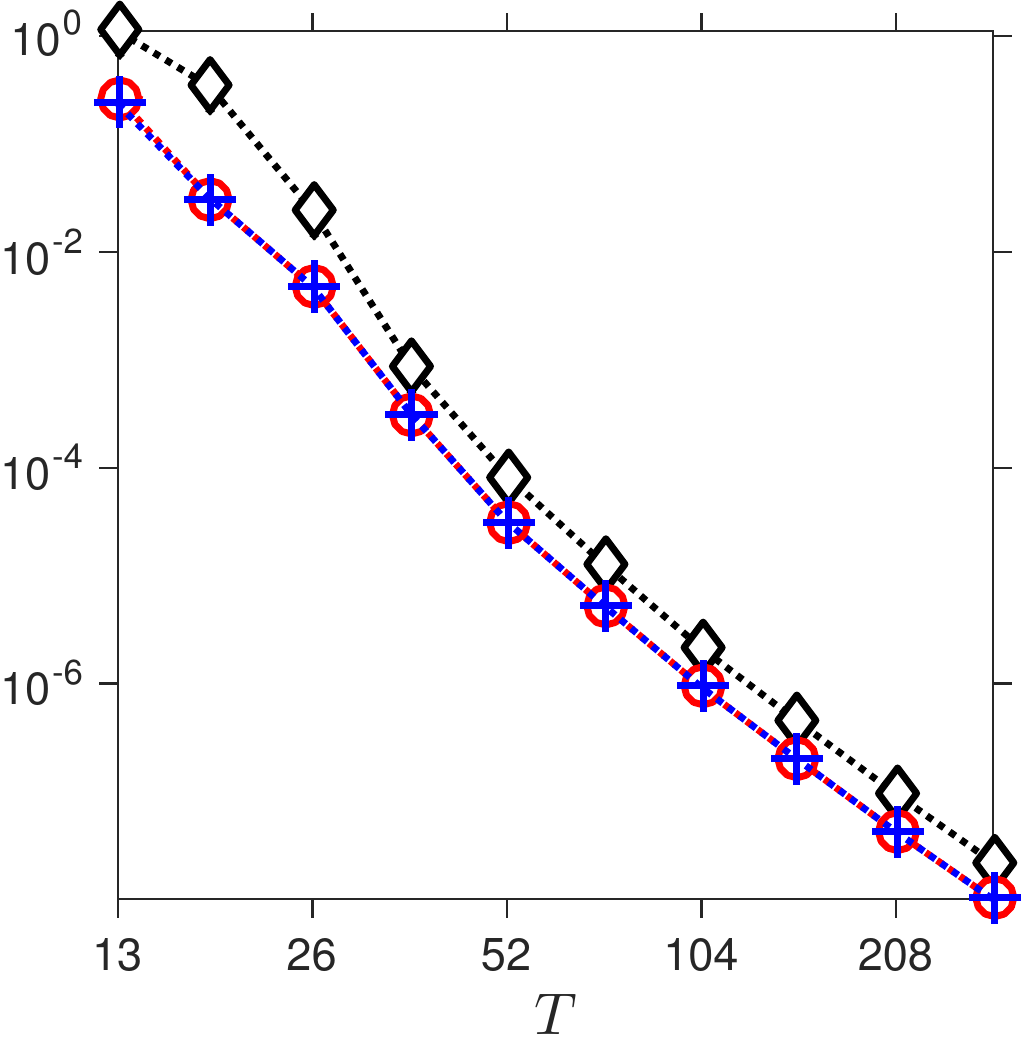}
\caption{\texttt{PolyDecayFast}}
\label{fig:PFAST-789}
\end{center}
\end{subfigure}
\begin{subfigure}{.325\textwidth}
\begin{center}
\includegraphics[height=1.5in]{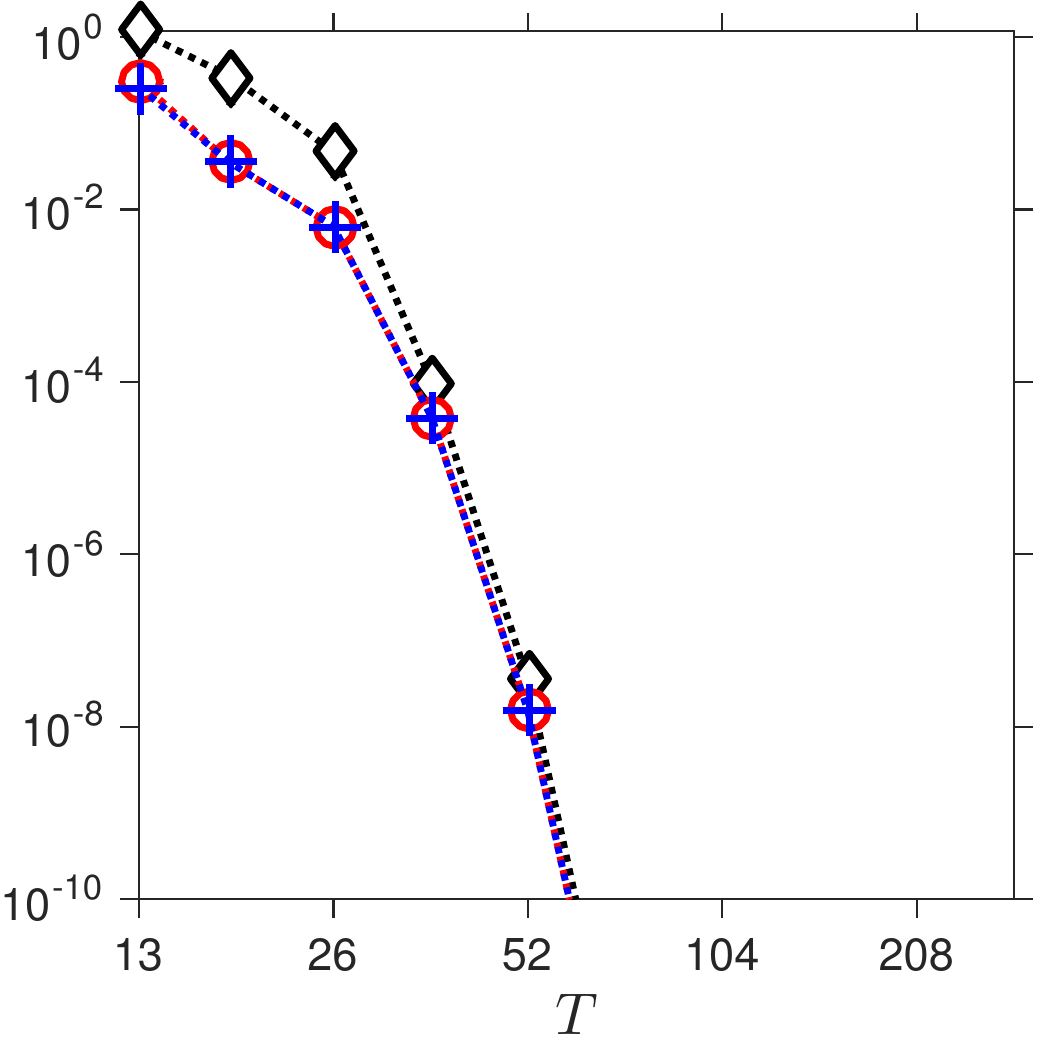}
\caption{\texttt{ExpDecaySlow}}
\label{fig:ESLOW-789}
\end{center}
\end{subfigure}
\end{center}

\vspace{0.5em}

\begin{center}
\begin{subfigure}{.325\textwidth}
\begin{center}
\includegraphics[height=1.5in]{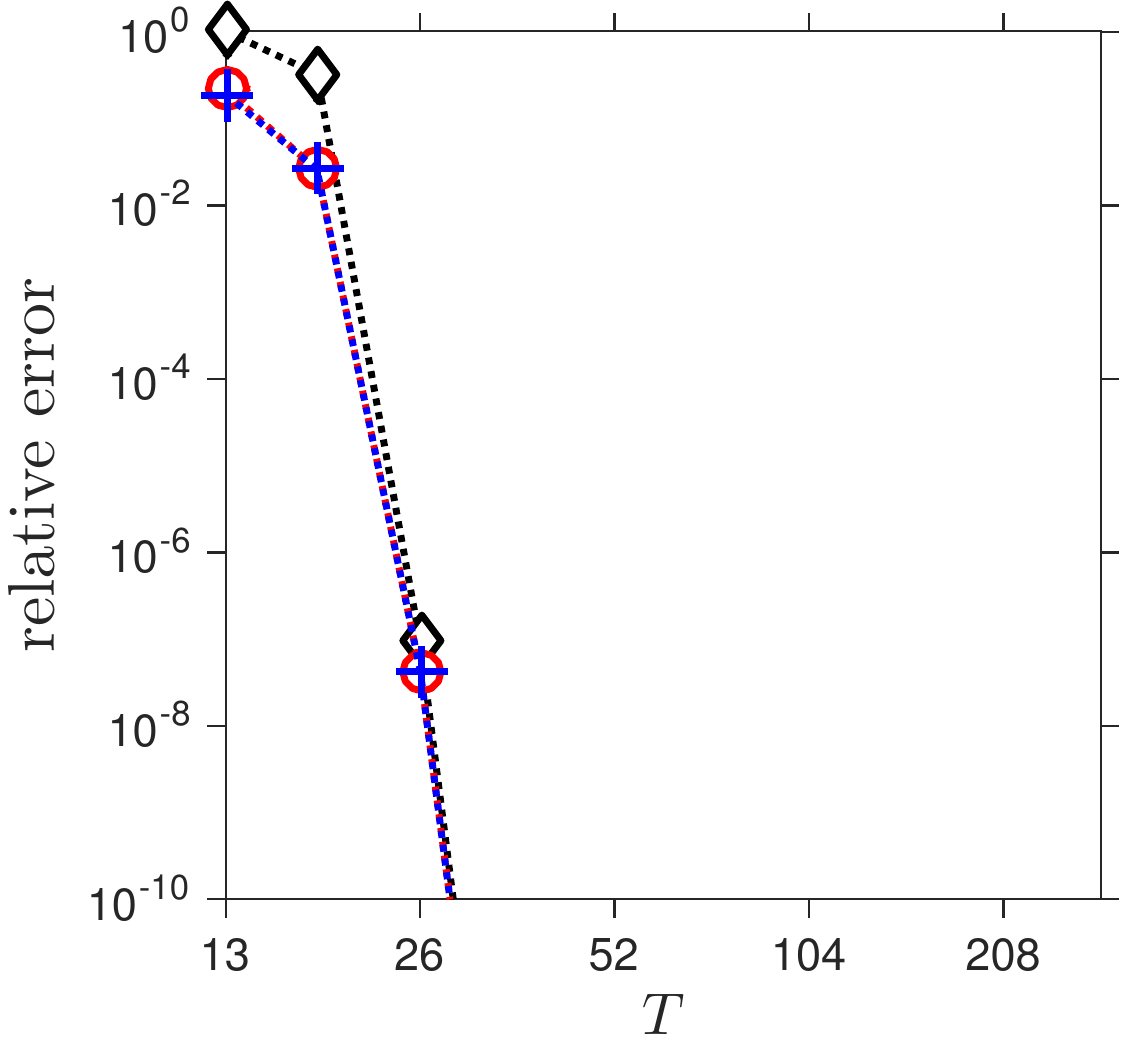}
\caption{\texttt{ExpDecayFast}}
\label{fig:EFAST-789}
\end{center}
\end{subfigure}
\begin{subfigure}{.325\textwidth}
\begin{center}
\includegraphics[height=1.5in]{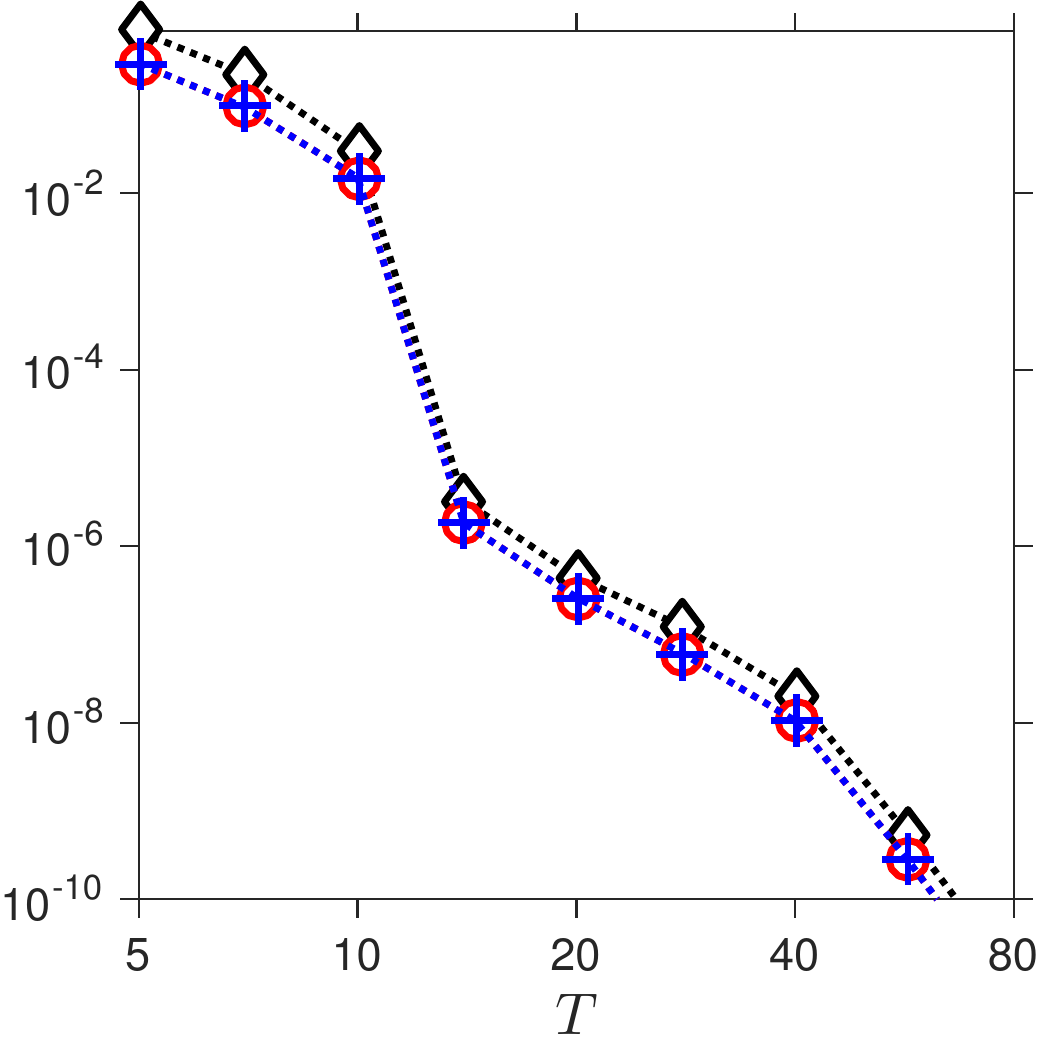}
\caption{\texttt{Data} ($r = 1$)}
\label{fig:DATA1-789}
\end{center}
\end{subfigure}
\begin{subfigure}{.325\textwidth}
\begin{center}
\includegraphics[height=1.5in]{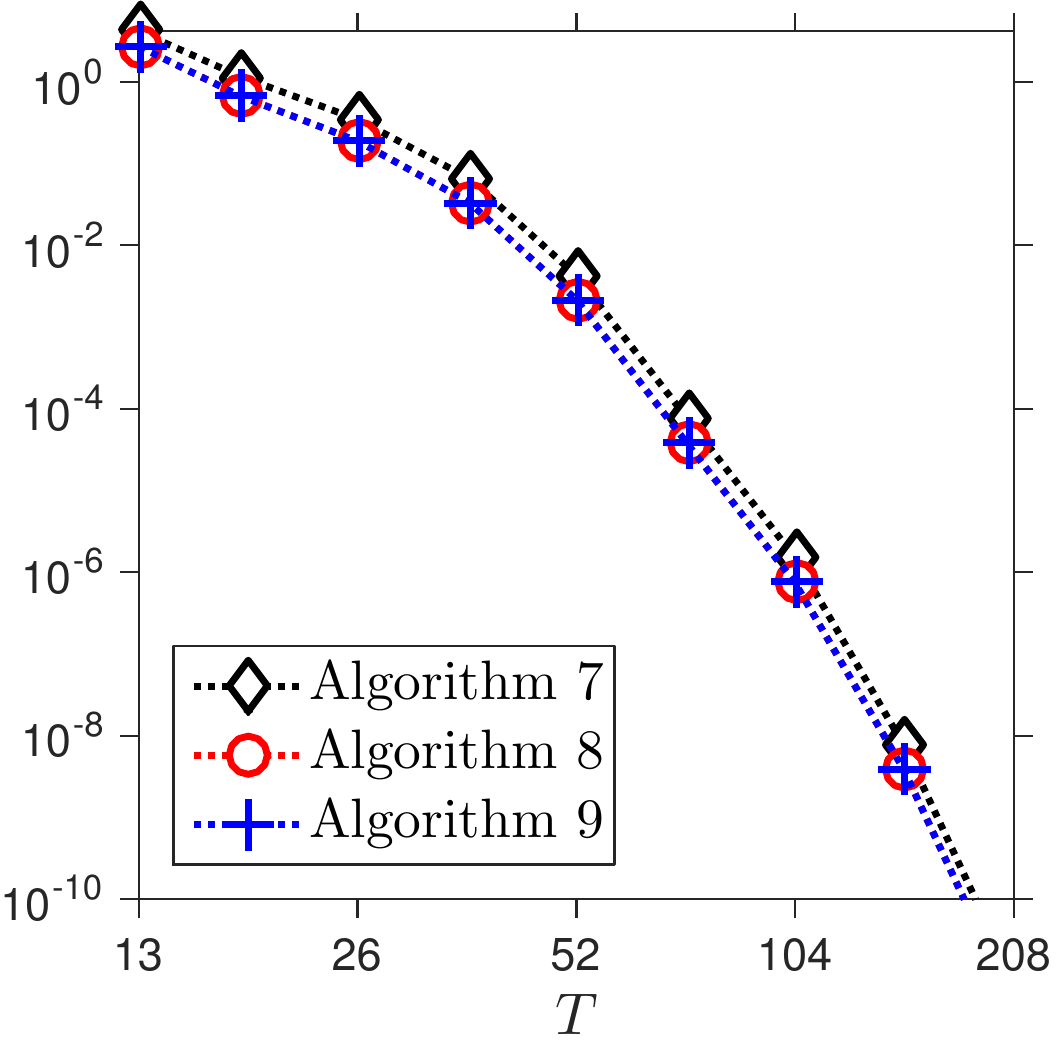}
\caption{\texttt{Data} ($r = 5$)}
\label{fig:DATA5-789}
\end{center}
\end{subfigure}
\end{center}

\vspace{0.5em}

\caption{\textbf{Oracle performance of sketching algorithms for structured fixed-rank matrix approximation
as a function of storage cost.} For each of the input matrices described in \cref{sec:input-matrix-examples},
we compare the oracle performance of the unstructured approximation (\cref{alg:fixed-rank-recon}),
the conjugate symmetric approximation (\cref{alg:sym-fixed-rank-recon}),
and the positive-semidefinite approximation (\cref{alg:psd-fixed-rank-recon}).
The matrix dimensions are $m = n = 10^3$ for the synthetic examples
and $m = n = 25,921$ for the matrix \texttt{Data}
from the phase retrieval application.  Each approximation has rank $r = 5$, unless otherwise stated.
The variable $T$ on the horizontal axis is (proportional to) the total storage used by each
sketching method.  Each data series displays the best relative error
\cref{eqn:relative-error} that the specified algorithm can achieve with storage $T$.
See \cref{sec:oracle-error} for details.}
\label{fig:oracle-structured}
\end{figure}

\begin{figure}[htp!]
\begin{center}
\begin{subfigure}{.325\textwidth}
\begin{center}
\includegraphics[height=1.5in]{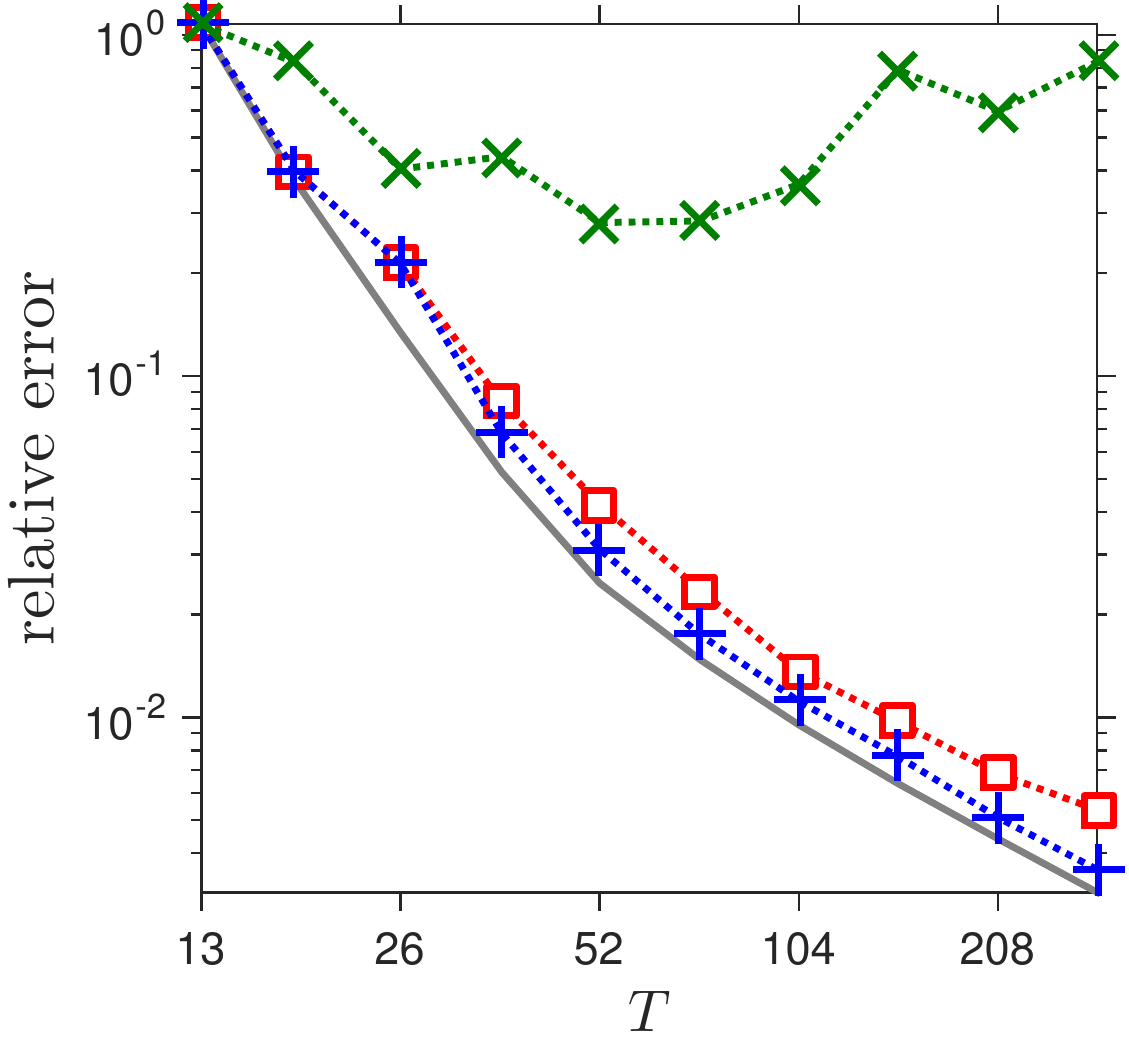}
\caption{\texttt{LowRankMedNoise}}
\label{fig:MED-theory}
\end{center}
\end{subfigure}
\begin{subfigure}{.325\textwidth}
\begin{center}
\includegraphics[height=1.5in]{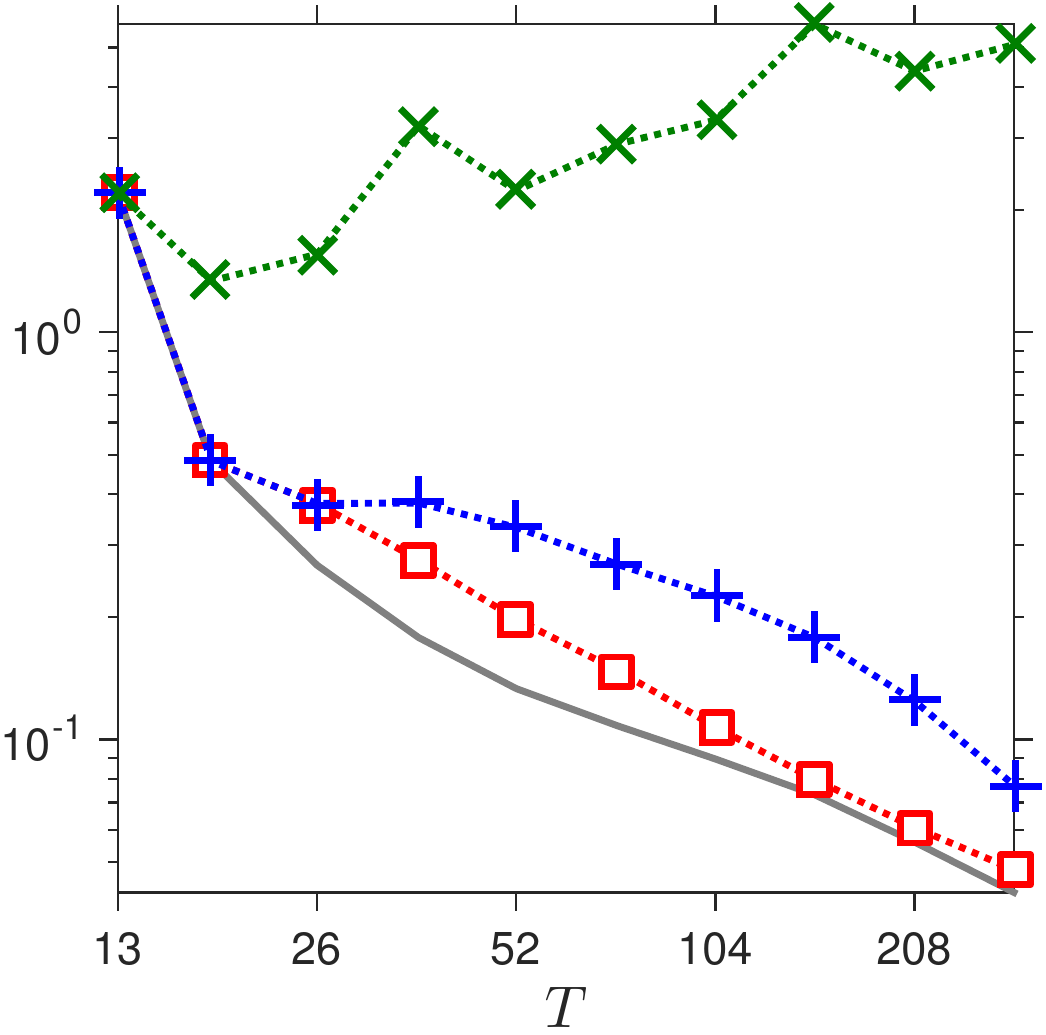}
\caption{\texttt{LowRankHiNoise}}
\label{fig:HI-theory}
\end{center}
\end{subfigure}
\end{center}

\vspace{.5em}

\begin{center}
\begin{subfigure}{.325\textwidth}
\begin{center}
\includegraphics[height=1.5in]{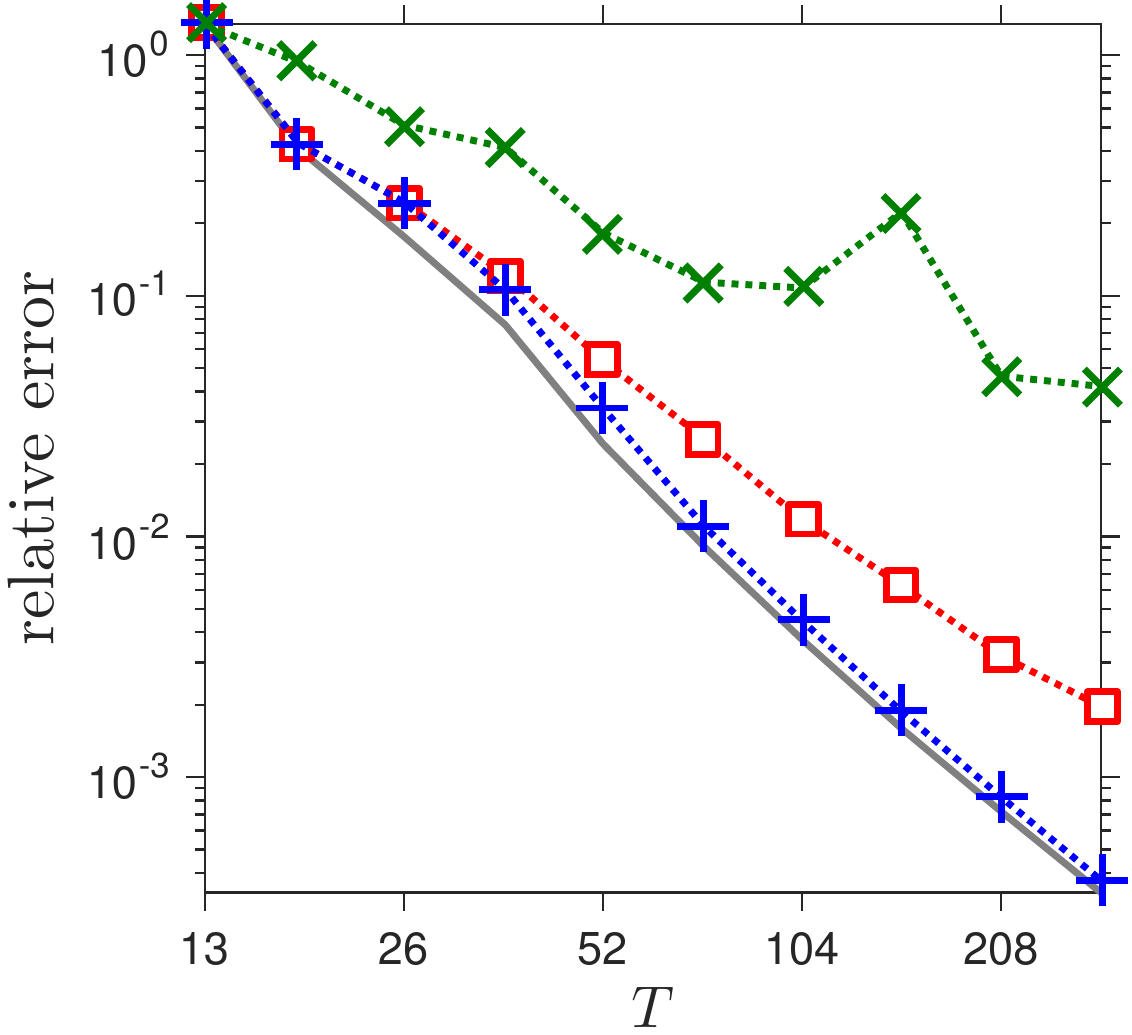}
\caption{\texttt{PolyDecaySlow}}
\label{fig:PSLOW-theory}
\end{center}
\end{subfigure}
\begin{subfigure}{.325\textwidth}
\begin{center}
\includegraphics[height=1.5in]{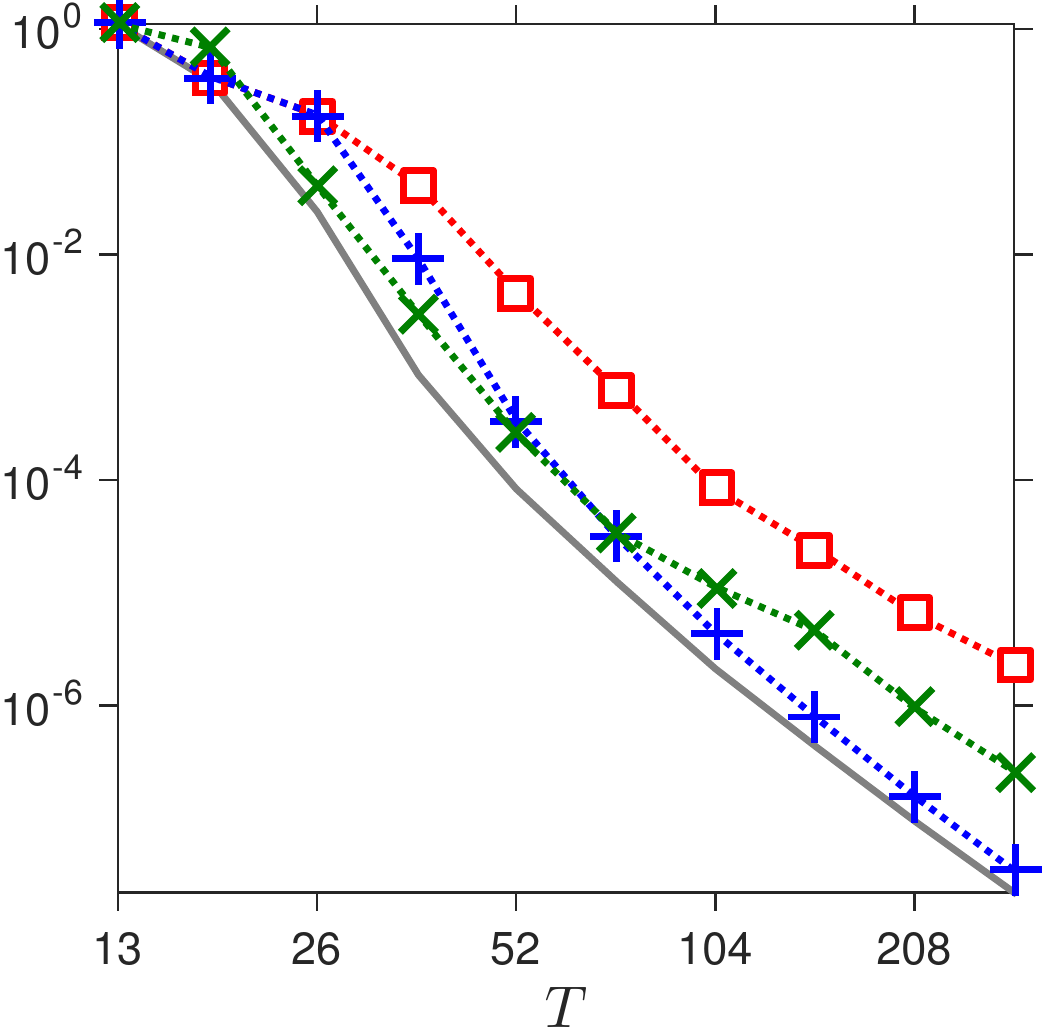}
\caption{\texttt{PolyDecayFast}}
\label{fig:PFAST-theory}
\end{center}
\end{subfigure}
\begin{subfigure}{.325\textwidth}
\begin{center}
\includegraphics[height=1.5in]{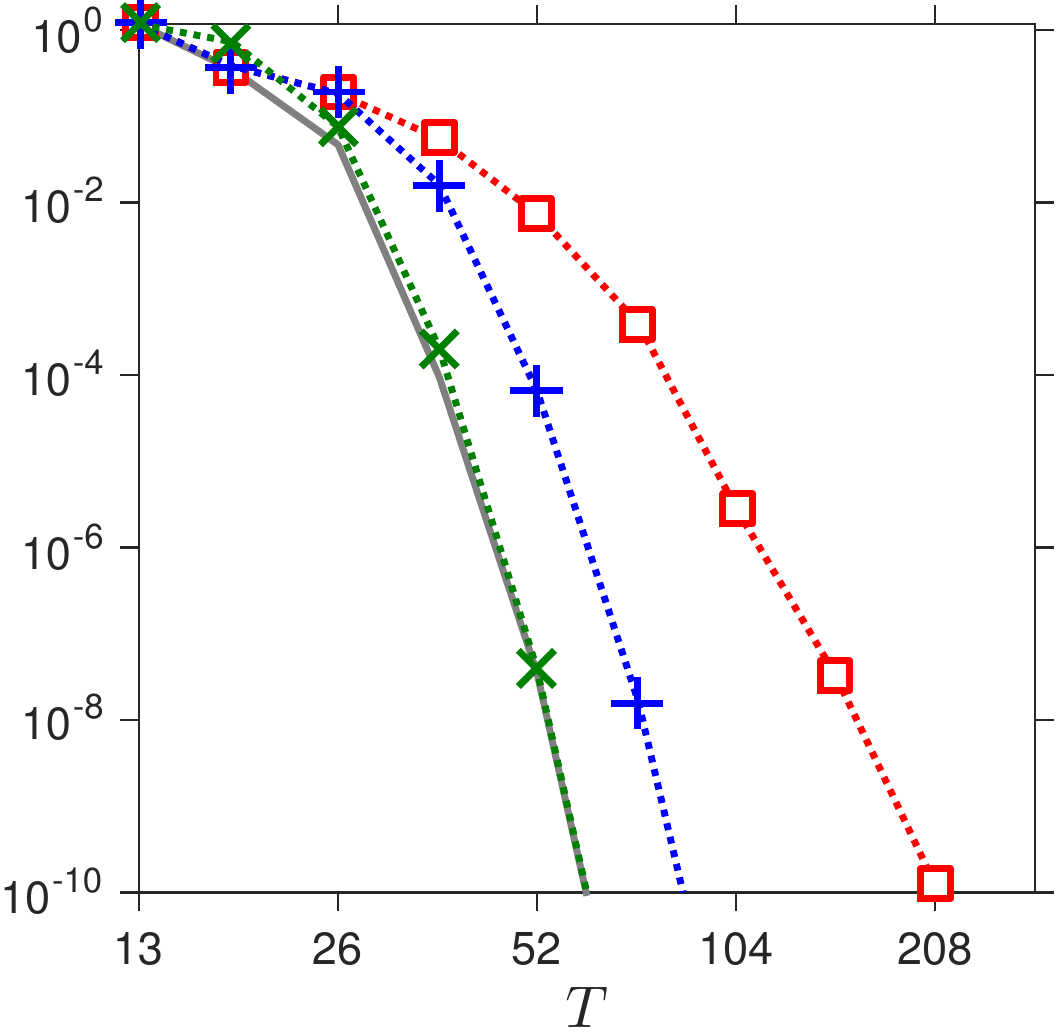}
\caption{\texttt{ExpDecaySlow}}
\label{fig:ESLOW-theory}
\end{center}
\end{subfigure}
\end{center}

\vspace{0.5em}

\begin{center}
\begin{subfigure}{.325\textwidth}
\begin{center}
\includegraphics[height=1.5in]{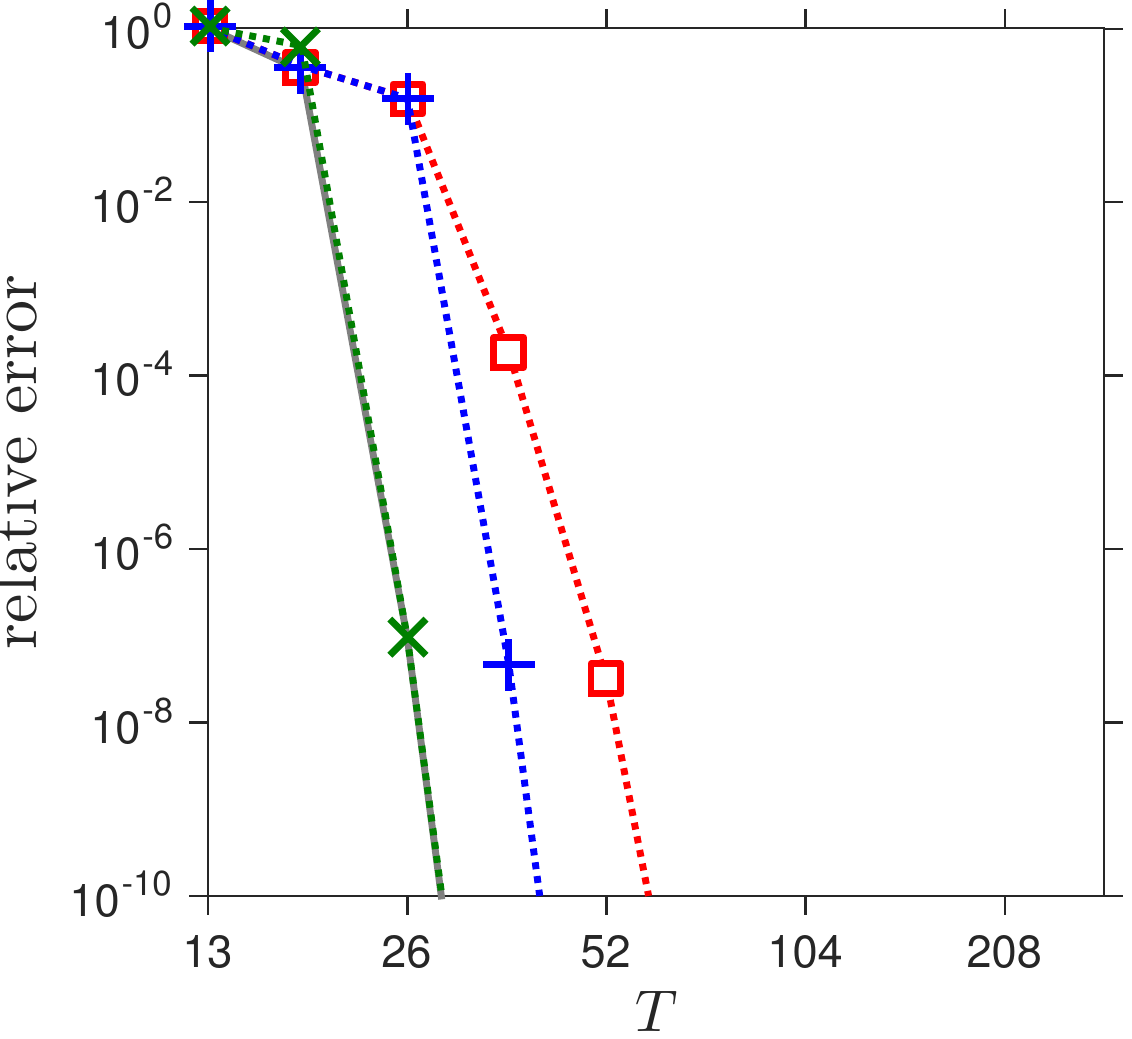}
\caption{\texttt{ExpDecayFast}}
\label{fig:EFAST-theory}
\end{center}
\end{subfigure}
\begin{subfigure}{.325\textwidth}
\begin{center}
\includegraphics[height=1.5in]{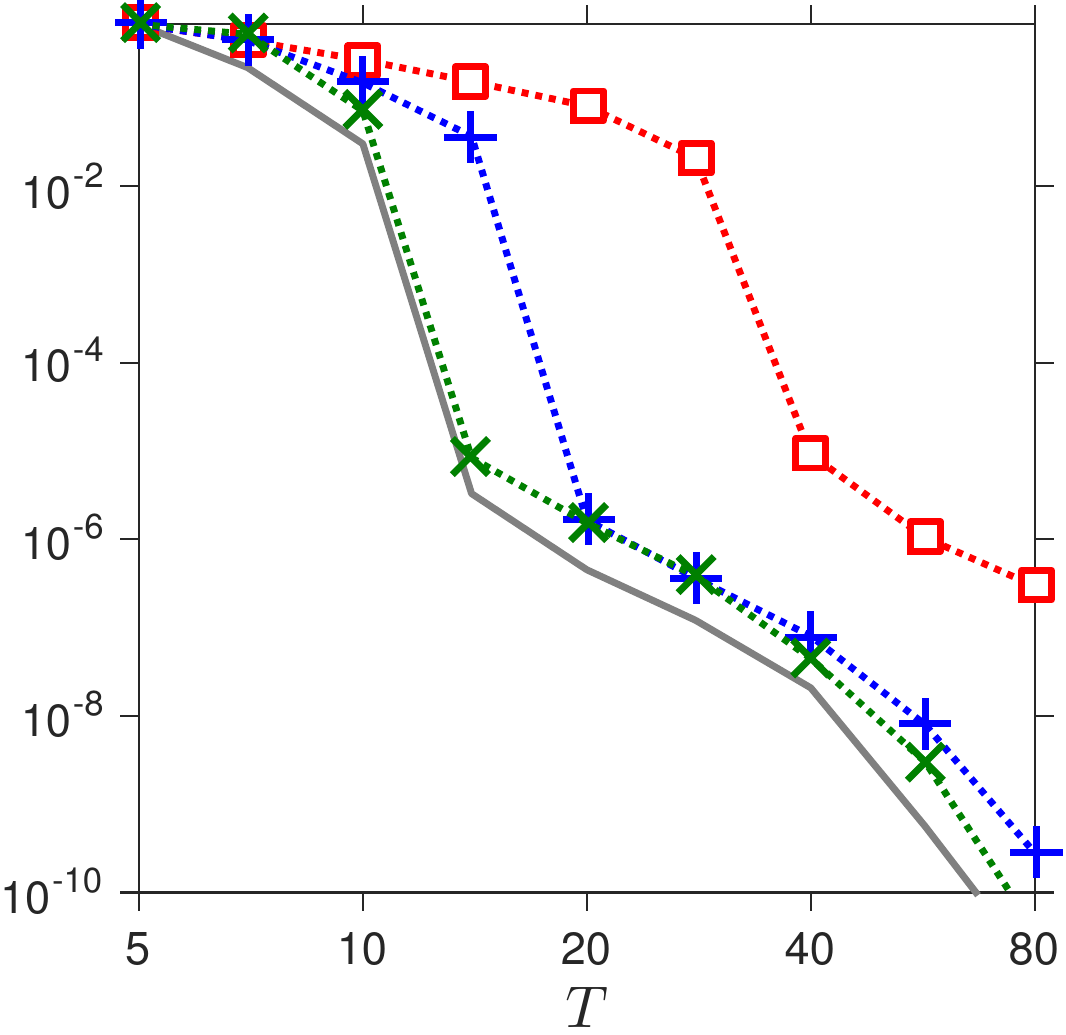}
\caption{\texttt{Data} ($r = 1$)}
\label{fig:DATA1-theory}
\end{center}
\end{subfigure}
\begin{subfigure}{.325\textwidth}
\begin{center}
\includegraphics[height=1.5in]{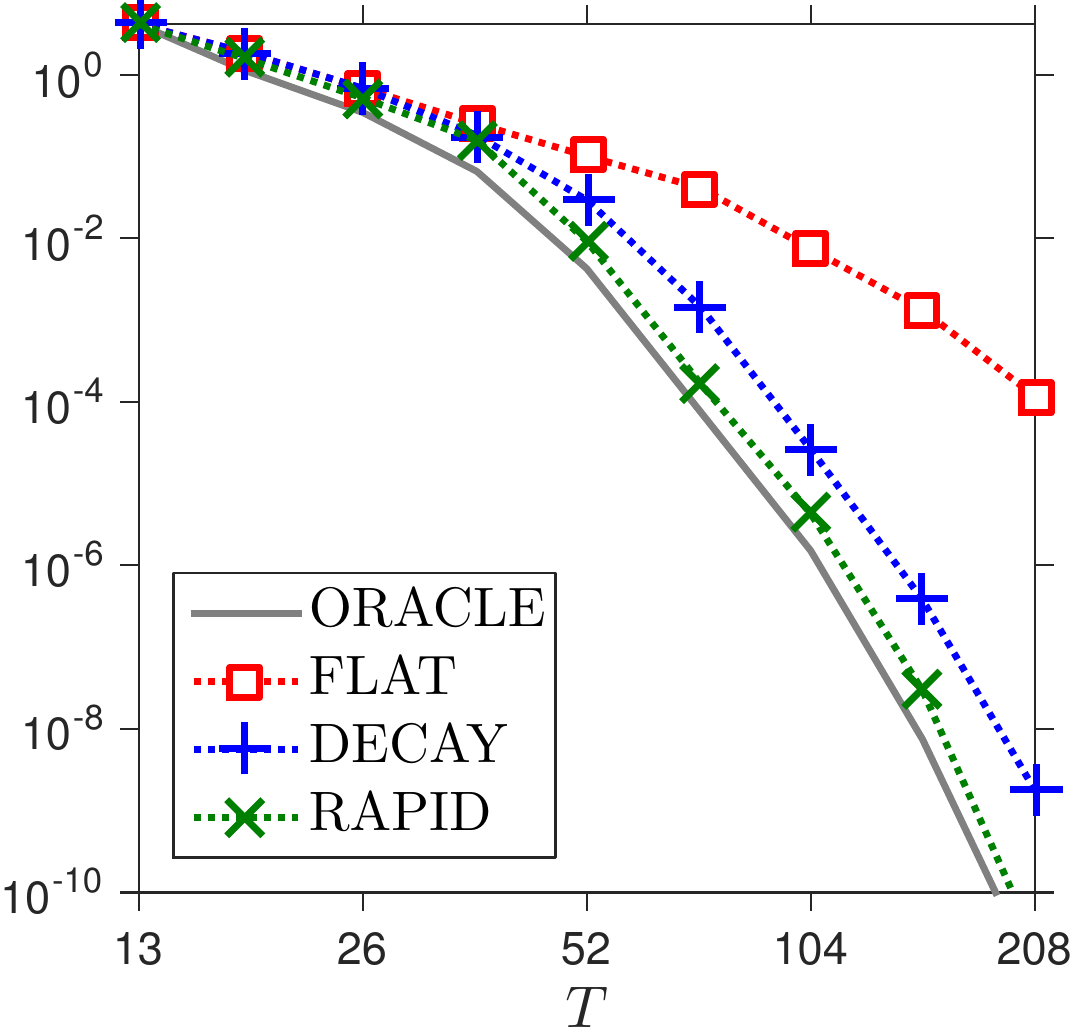}
\caption{\texttt{Data} ($r = 5$)}
\label{fig:DATA5-theory}
\end{center}
\end{subfigure}
\end{center}

\caption{\textbf{Performance of a sketching algorithm for fixed-rank matrix approximation
with a priori parameter choices.}
For each of the input matrices described in \cref{sec:input-matrix-examples},
we compare the oracle performance of the fixed-rank approximation, \cref{alg:fixed-rank-recon},
against its performance at theoretically motivated parameter choices.
The matrix dimensions are $m = n = 10^3$ for the synthetic examples
and $m = n = 25,921$ for the matrix \texttt{Data}
from the phase retrieval application.  Each approximation has rank $r = 5$, unless otherwise stated.
The variable $T$ on the horizontal axis is (proportional to) the total storage used by each
sketching method.  The oracle performance is drawn from \cref{fig:oracle-performance}.
Each data series displays the relative error \cref{eqn:relative-error}
that \cref{alg:fixed-rank-recon} achieves for a specific parameter selection.
The parameter choice \texttt{FLAT} \cref{eqn:flat-params}
is designed for matrices with a flat spectral tail;  \texttt{DECAY} \cref{eqn:decay-params}
is for a slowly decaying spectrum; \texttt{RAPID} \cref{eqn:fast-decay-params}
is for a rapidly decaying spectrum. See \cref{sec:theory-params} for details.} \label{fig:theory-params}
\end{figure}

\appendix

\section{Analysis of the Low-Rank Approximation}

In this appendix, we develop theoretical results on the performance
of the basic low-rank approximation~\cref{eqn:Ahat} implemented
in \cref{alg:simple-low-rank-recon,alg:detailed-low-rank-recon}.

\subsection{Facts about Random Matrices}

Our arguments require classical formulae for the expectations
of functions of a standard normal matrix.  In the real case,
these results are~\cite[Prop.~A.1 and A.6]{HMT11:Finding-Structure}.
The complex case follows from the same principles, so we omit the details.

\begin{fact} \label{fact:expect-gauss-frob}
Let $\mtx{G} \in \F^{t \times s}$ be a standard normal matrix.
For all matrices $\mtx{B}$ and $\mtx{C}$ with conforming dimensions,
\begin{equation} \label{eqn:expect-gauss-frob}
\Expect \fnormsq{ \mtx{BGC} } = \beta \fnormsq{\mtx{B}} \fnormsq{\mtx{C}}.
\end{equation}
Furthermore, if $t > s + \alpha$,
\begin{equation} \label{eqn:expect-gauss-pinv-frob}
\Expect \fnormsq{\mtx{G}^\dagger} = \frac{1}{\beta} \cdot \frac{s}{t - s - \alpha}
	= \frac{1}{\beta} \cdot f(s, t).
\end{equation}
The numbers $\alpha$ and $\beta$ are given by~\cref{eqn:alpha-parameter};
the function $f$ is introduced in~\cref{eqn:f-intro}.
\end{fact}

\subsection{Results from Randomized Linear Algebra}

Our arguments also depend heavily on the analysis of randomized
low-rank approximation developed in~\cite[Sec.~10]{HMT11:Finding-Structure}.
We state these results using the familiar notation
from \cref{sec:mult-sketching,sec:low-rank-recon}.

\begin{fact}[Halko et al.~2011] \label{fact:hmt-err}
Fix $\mtx{A} \in \F^{m \times n}$.  Let $\varrho$ be a natural number
such that $\varrho < k - \alpha$.
Draw the random test matrix $\mtx{\Omega} \in \F^{k \times n}$
from the standard normal distribution.
Then the matrix $\mtx{Q}$ computed by~\cref{eqn:def-Q} satisfies
$$
\Expect_{\mtx{\Omega}} \fnormsq{ \mtx{A} - \mtx{QQ}^* \mtx{A} }
	\leq (1 + f(\varrho, k)) \cdot \tau_{\varrho+1}^2(\mtx{A}).
$$
The number $\alpha$ is given by~\cref{eqn:alpha-parameter};
the function $f$ is introduced in~\cref{eqn:f-intro}.
\end{fact}

\noindent
This result follows immediately from the
proof of~\cite[Thm.~10.5]{HMT11:Finding-Structure}
using \cref{fact:expect-gauss-frob} to handle both
the real and complex case simultaneously.

\subsection{Proof of \cref{thm:err-frob}: Frobenius Error Bound}
\label{sec:proof-low-rank-recon}

In this section, we establish a second Frobenius-norm error bound
for the low-rank approximation~\cref{eqn:Ahat}.
We maintain the notation from \cref{sec:mult-sketching,sec:low-rank-recon},
and we state explicitly when we are making distributional assumptions
on the test matrices.

\subsubsection{Decomposition of the Approximation Error}

\Cref{fact:hmt-err} formalizes the intuition that $\mtx{A} \approx \mtx{Q}(\mtx{Q}^* \mtx{A})$.
The main object of the proof is to demonstrate that $\mtx{X} \approx \mtx{Q}^* \mtx{A}$.
The first step in the argument is to break down the approximation error into these two parts.

\begin{lemma} \label{lem:err-decomp}
Let $\mtx{A}$ be an input matrix, and
let $\hat{\mtx{A}} = \mtx{QX}$ be the approximation defined in~\cref{eqn:Ahat}.
The approximation error decomposes as
$$
\fnormsq{ \mtx{A} - \hat{\mtx{A}} }
	= \fnormsq{ \mtx{A} - \mtx{QQ}^* \mtx{A} } + \fnormsq{ \mtx{X} - \mtx{Q}^*\mtx{A} }.
$$
\end{lemma}

\noindent
We omit the proof, which is essentially just the Pythagorean theorem.

\subsubsection{Approximating the Second Factor}

Next, we develop an explicit expression for the error in
the approximation $\mtx{X} \approx \mtx{Q}^* \mtx{A}$.
It is convenient to
construct a matrix $\mtx{P} \in \F^{n \times (n-k)}$
with orthonormal columns that satisfies
\begin{equation} \label{eqn:def-P}
\mtx{PP}^* = \Id - \mtx{QQ}^*.
\end{equation}
Introduce the matrices
\begin{equation} \label{eqn:Psis}
\mtx{\Psi}_1 := \mtx{\Psi} \mtx{P} \in \F^{\ell \times (n-k)}
\quad\text{and}\quad
\mtx{\Psi}_2 := \mtx{\Psi} \mtx{Q} \in \F^{\ell \times k}.
\end{equation}
We are now prepared to state the result.

\begin{lemma} \label{lem:subspace-err}
Assume that the matrix $\mtx{\Psi}_2$ has full column-rank.  Then
\begin{equation} \label{eqn:B-Q*A}
\mtx{X} - \mtx{Q}^* \mtx{A} = \mtx{\Psi}_2^\dagger \mtx{\Psi}_1 (\mtx{P}^*\mtx{A}) .
\end{equation}
The matrices $\mtx{\Psi}_1$ and $\mtx{\Psi}_2$ are defined in \cref{eqn:Psis}.
\end{lemma}

\begin{proof}
Recall that $\mtx{W} = \mtx{\Psi} \mtx{A}$, and calculate that
$$
\begin{aligned}
\mtx{W} = \mtx{\Psi} \mtx{A}
	= \mtx{\Psi}\mtx{PP}^* \mtx{A} + \mtx{\Psi} \mtx{QQ}^* \mtx{A}
	= \mtx{\Psi}_1 (\mtx{P}^*\mtx{A}) + \mtx{\Psi}_2 (\mtx{Q}^* \mtx{A}).
\end{aligned}
$$
The second relation holds because $\mtx{PP}^* + \mtx{QQ}^* = \Id$.
Then we use \cref{eqn:Psis} to identify $\mtx{\Psi}_1$ and $\mtx{\Psi}_2$.
By hypothesis, the matrix  $\mtx{\Psi}_2$ has full column-rank,
so we can left-multiply the last display by $\mtx{\Psi}_2^\dagger$ to attain
$$
\mtx{\Psi}_2^\dagger \mtx{W} = \mtx{\Psi}_2^\dagger \mtx{\Psi}_1 (\mtx{P}^* \mtx{A}) + \mtx{Q}^* \mtx{A}.
$$
Turning back to~\cref{eqn:def-X},
we identify $\mtx{X} = \mtx{\Psi}_2^\dagger \mtx{W}$.
\end{proof}

\subsubsection{The Expected Frobenius-Norm Error in the Second Factor}

We are now prepared to compute the average Frobenius-norm error in approximating
$\mtx{Q}^*\mtx{A}$ by means of the matrix $\mtx{X}$.  In contrast
to the previous steps, this part of the argument relies on
distributional assumptions on the test matrix $\mtx{\Psi}$.
Remarkably, for a Gaussian test matrix, $\mtx{X}$
is even an unbiased estimator
of the factor $\mtx{Q}^*\mtx{A}.$

\begin{lemma} \label{lem:avg-subspace-err}
Assume that $\mtx{\Psi} \in \F^{\ell \times n}$
is a standard normal matrix that is independent from $\mtx{\Omega}$.
Then
$$
\Expect_{\mtx{\Psi}}[  \mtx{X} - \mtx{Q}^* \mtx{A} ] = \mtx{0}.
$$
Furthermore,
$$
\Expect_{\mtx{\Psi}} \fnormsq{ \mtx{X}  - \mtx{Q}^* \mtx{A} }
	= f(k,\ell) \cdot \fnormsq{ \mtx{A} - \mtx{QQ}^* \mtx{A} }.
$$
\end{lemma}

\begin{proof}
Observe that $\mtx{P}$ and $\mtx{Q}$ are partial isometries with orthogonal ranges.
Owing to the marginal property of the standard normal distribution,
the random matrices $\mtx{\Psi}_1$ and $\mtx{\Psi}_2$
are statistically independent standard normal matrices.
In particular, $\mtx{\Psi}_2 \in \F^{\ell \times k}$ almost surely has full column-rank
because~\cref{eqn:param-assumption} requires that $\ell \geq k$.

First, take the expectation of the identity~\cref{eqn:B-Q*A} to see that
$$
\Expect_{\mtx{\Psi}}[ \mtx{X} -  \mtx{Q}^* \mtx{A} ]
	= \Expect_{\mtx{\Psi}_2}\Expect_{\mtx{\Psi}_1} [ \mtx{\Psi}_2^\dagger \mtx{\Psi}_1 \mtx{P}^* \mtx{A} ]
	= \mtx{0}.
$$
In the first relation, we use the statistical
independence of $\mtx{\Psi}_1$ and $\mtx{\Psi}_2$
to write the expectation as an iterated expectation.
Then we observe that $\mtx{\Psi}_1$ is a matrix with zero mean.

Next, take the expected squared Frobenius norm of~\cref{eqn:B-Q*A} to see that
$$
\begin{aligned}
\Expect_{\mtx{\Psi}} \fnormsq{ \mtx{X} - \mtx{Q}^* \mtx{A} }
	&= \Expect_{\mtx{\Psi}_2} \Expect_{\mtx{\Psi}_1} \fnormsq{ \mtx{\Psi}_2^\dagger \mtx{\Psi}_1 (\mtx{P}^* \mtx{A}) } \\
	&= \beta \cdot \Expect_{\mtx{\Psi}_2} \big[ \fnormsq{\mtx{\Psi}_2^\dagger} \cdot \fnormsq{\mtx{P}^* \mtx{A}} \big]
	= f(k, \ell) \cdot  \fnormsq{\mtx{P}^*\mtx{A}}.
\end{aligned}
$$
The last two identities follow from~\cref{eqn:expect-gauss-frob}
and~\cref{eqn:expect-gauss-pinv-frob} respectively,
where we use the fact that $\mtx{\Psi}_2 \in \F^{\ell \times k}$.
To conclude, note that
$$
\fnormsq{\mtx{P}^*\mtx{A}} = \fnormsq{\mtx{PP}^*\mtx{A} }
	= \fnormsq{\mtx{A} - \mtx{QQ}^* \mtx{A} }.
$$
The first relation holds because $\mtx{P}$ is a partial isometry and the Frobenius
norm is unitarily invariant.  Last, we apply the definition~\cref{eqn:def-P} of $\mtx{P}$.
\end{proof}

\subsubsection{Proof of \cref{thm:err-frob}}

We are now prepared to complete the proof of
the Frobenius-norm error bound stated in \cref{thm:err-frob}.
For this argument, we assume that the test matrices
$\mtx{\Omega}\in \F^{n \times k}$ and $\mtx{\Psi} \in \F^{\ell \times m}$
are drawn independently from the standard normal distribution.

According to \cref{lem:err-decomp},
$$
\fnormsq{ \mtx{A} - \hat{\mtx{A}} }
	= \fnormsq{ \mtx{A} - \mtx{QQ}^* \mtx{A} } + \fnormsq{ \mtx{X} - \mtx{Q}^* \mtx{A} }.
$$
Take the expectation of the last display to reach
$$
\begin{aligned}
\Expect \fnormsq{\mtx{A} - \hat{\mtx{A}}}
	&= \Expect_{\mtx{\Omega}} \fnormsq{\mtx{A} - \mtx{QQ}^* \mtx{A}}
	+ \Expect_{\mtx{\Omega}} \Expect_{\mtx{\Psi}} \fnormsq{ \mtx{X} - \mtx{Q}^* \mtx{A}} \\
	&= (1+f(k,\ell)) \cdot \Expect_{\mtx{\Omega}} \fnormsq{\mtx{A} - \mtx{QQ}^* \mtx{A}} \\
	&\leq (1+ f(k,\ell)) \cdot (1+f(\varrho, k)) \cdot \tau_{\varrho+1}^2(\mtx{A}).
\end{aligned}
$$
In the first line, we use the independence of the two random matrices to
write the expectation as an iterated expectation.  To reach the second line,
we apply \cref{lem:avg-subspace-err} to the second term.
Invoke the randomized linear algebra result, \cref{fact:hmt-err}.
Finally, minimize over eligible indices $\varrho < k - \alpha$.

\bibliographystyle{siamplain}

\end{document}